\pdfoutput=1
\RequirePackage{silence}
\documentclass[a4paper,11pt]{amsart}
\usepackage[hmarginratio={1:1},vmarginratio={1:1},lmargin=70.0pt,tmargin=70.0pt]{geometry}
\usepackage{lmodern}

\usepackage{ifdraft}
\ifdraft{\usepackage[draft]{showkeys}}{\usepackage[final]{showkeys}}
\usepackage{calligra}
\usepackage[T1]{fontenc}
\usepackage[ugly]{nicefrac}
\usepackage{mathtools}
\mathtoolsset{mathic}
\usepackage[final]{microtype}
\usepackage{multicol}
\usepackage{tabularx,tabu}
\usepackage{latexsym,exscale,amsfonts,amssymb,mathtools}
\usepackage[shortlabels]{enumitem} 
\usepackage{amsmath,amsthm,amsfonts,amssymb,amscd,textcomp,bbm}
\usepackage{stmaryrd}
\SetSymbolFont{stmry}{bold}{U}{stmry}{m}{n}
\usepackage[normalem]{ulem}
\usepackage{thmtools}
\usepackage{fancybox}
\usepackage[table]{xcolor}
\usepackage{mathrsfs}
\DeclareMathAlphabet{\mathscrbf}{OMS}{mdugm}{b}{n}
\DeclareFontEncoding{LS1}{}{}
\DeclareFontSubstitution{LS1}{stix}{m}{n}
\DeclareMathAlphabet{\mathpzc}{LS1}{stixscr}{m}{n}

\usepackage{caption} 
\usepackage{tabulary}
\usepackage{booktabs}

\definecolor{mygray}{gray}{0.6}
\definecolor{mygraydark}{gray}{0.4}
\definecolor{mygraylight}{gray}{0.8}

\definecolor{cherry}{RGB}{222,49,99}
\definecolor{cream}{RGB}{255,253,208}
\definecolor{corn}{RGB}{251,236,93}
\definecolor{citron}{RGB}{190,180,90}

\definecolor{spinach}{RGB}{46,139,87}
\definecolor{tomato}{RGB}{255,99,71}
\definecolor{pumpkin}{RGB}{224,180,80}

\definecolor{orchid}{RGB}{143,40,194}
\definecolor{lava}{RGB}{207,16,32}
\definecolor{mydarkblue}{RGB}{10,10,150}

\definecolor{myorange}{RGB}{225,127,0}
\definecolor{mygreen}{RGB}{0,225,0}
\definecolor{mypurple}{RGB}{128,0,128}
\definecolor{myred}{RGB}{255,0,0}
\definecolor{myblue}{RGB}{0,0,195}
\definecolor{myyellow}{RGB}{210,210,0}

\usepackage[all]{xy}
\SelectTips{cm}{}

\usepackage{tikz}
\tikzstyle{densely dotted}=[dash pattern=on \pgflinewidth off .5pt]
\tikzset{anchorbase/.style={baseline={([yshift=-0.5ex]current bounding box.center)}},
tinynodes/.style={font=\tiny, text height=0.25ex, text depth=0.05ex},
smallnodes/.style={font=\scriptsize, text height=0.75ex, text depth=0.15ex},
ssoergel/.style={line width=1.0,color=spinach},
tsoergel/.style={line width=1.0,color=tomato,densely dotted},
usual/.style={line width=1.0,color=black},
JW/.style={line width=1.0,densely dotted,color=black},
pQJW/.style={line width=1.0,densely dashed,color=black,fill=corn!60},
pJW/.style={line width=1.0,color=black,fill=orchid!70},
}
\usetikzlibrary{cd}
\usetikzlibrary{decorations}
\usetikzlibrary{decorations.markings}
\usetikzlibrary{decorations.pathreplacing}
\usetikzlibrary{decorations.pathmorphing}
\usetikzlibrary{arrows.meta,shapes,positioning,matrix,calc}
\usetikzlibrary{shapes.callouts}
\tikzstyle directed=[postaction={decorate,decoration={markings,
mark=at position #1 with {\arrow[line width=0.15mm, black]{>}}}}]
\tikzstyle odirected=[postaction={decorate,decoration={markings,
mark=at position #1 with {\arrow[line width=0.15mm, myorange]{>}}}}]
\tikzstyle smarked=[postaction={decorate,decoration={markings,
mark=at position #1 with {\fill[spinach] (0,0) circle (.065cm);}}}]
\tikzstyle tmarked=[postaction={decorate,decoration={markings,
mark=at position #1 with {\fill[tomato] (0,0) circle (.065cm);}}}]
\tikzstyle wmarked=[postaction={decorate,decoration={markings,
mark=at position #1 with {\fill[white] (0,0) circle (.1cm);}}}]

\allowdisplaybreaks

\newcommand{\tm}{\text{-}}
\newcommand{\neatfrac}[2]{{\raisebox{.2em}{$#1$}\!\left/\raisebox{-.2em}{$#2$}\right.}}

\newcommand{\ie}{\textsl{i.e. }}
\newcommand{\eg}{\textsl{e.g. }}
\newcommand{\cf}{\textsl{cf. }}

\newcommand{\muta}{\textsl{mutatis mutandis}}
\newcommand{\loccit}{\textsl{loc. cit.}}

\renewcommand{\dots}{\text{...}}

\newcommand{\C}{\mathbb{C}}
\newcommand{\Z}{\mathbb{Z}}
\newcommand{\R}{\mathbb{R}}
\newcommand{\Q}{\mathbb{Q}}
\newcommand{\Ring}{\mathbb{S}}
\newcommand{\N}[1][]{\mathbb{N}_{#1}}

\newcommand{\vcirc}{\circ}
\newcommand{\hcirc}{\otimes}
\newcommand{\munit}{\mathbbm{1}}

\newcommand{\setstuff}[1]{\mathrm{#1}}
\newcommand{\catstuff}[1]{\mathbf{#1}}
\newcommand{\functorstuff}[1]{\mathpzc{#1}}

\newcommand{\obstuff}[1]{\mathtt{#1}}
\newcommand{\morstuff}[1]{\mathrm{#1}}

\newcommand{\idmor}{\mathbbm{1}}

\newcommand{\End}{\setstuff{End}}
\newcommand{\Hom}{\setstuff{Hom}}
\newcommand{\spanFp}[1]{\setstuff{span}_{\F}({#1})}
\newcommand{\spanQ}[1]{\setstuff{span}_{\Q}({#1})}

\newcommand{\mainfunctor}{\functorstuff{F}}
\newcommand{\losp}{losp}

\newcommand{\ppar}{\mathsf{p}}
\newcommand{\F}[1][\ppar]{\mathbb{F}_{\ppar}}
\newcommand{\K}[1][\ppar]{\mathbb{K}}
\newcommand{\ord}[1][\ppar]{\nu_{#1}}
\newcommand{\spe}[1]{\overline{#1}}
\newcommand{\eve}{\setstuff{Eve}}
\newcommand{\pbase}[2]{[#1]_{#2}}
\newcommand{\ancest}[1][v]{\setstuff{A}(#1)}
\newcommand{\supp}[1][v]{\setstuff{supp}(#1)}
\newcommand{\fsupp}[1][v]{\setstuff{fsupp}(#1)}
\newcommand{\mother}[1][v]{\obstuff{m}_{#1}}
\newcommand{\motherr}[2]{\obstuff{m}^{#2}_{#1}}
\newcommand{\generation}[1][v]{\obstuff{g}_{#1}}
\newcommand{\fancest}[2]{a_{#1,#2}}

\newcommand{\zigzag}{\obstuff{Z}}
\newcommand{\zigzagmod}{p\catstuff{Mod}\tm\zigzag}
\newcommand{\zigzagmode}{p\catstuff{Mod}\tm\zigzag_{e{-}1}}

\newcommand{\idemy}[1][v]{e_{#1}}

\newcommand{\tlfunctor}{\functorstuff{D}}

\newcommand{\SLtwo}{\mathrm{SL}_{2}(\K)}
\newcommand{\tilt}{\catstuff{Tilt}}

\newcommand{\tmod}{\obstuff{T}}
\newcommand{\wmod}{\Delta}
\newcommand{\dwmod}{\nabla}
\newcommand{\lmod}{\obstuff{L}}

\newcommand{\TL}[1][\K]{{#1}\catstuff{TL}}
\newcommand{\sbas}{\setstuff{B}}
\newcommand{\sbasel}{\morstuff{X}}
\newcommand{\sbasc}{x}
\newcommand{\dist}{\mathsf{d}}

\newcommand{\idtl}[1][v{-}1]{\idmor_{#1}}

\newcommand{\flip}[1][\phantom{a}]{{#1}^{\star}}
\newcommand{\td}[1]{\setstuff{td}(#1)}

\newcommand{\pTr}{\setstuff{pTr}}

\newcommand{\qjw}[1][v{-}1]{\tilde{\obstuff{e}}_{#1}}
\newcommand{\pqjw}[1][v{-}1]{\overline{\obstuff{e}}_{#1}}
\newcommand{\pjw}[1][v{-}1]{\obstuff{e}_{#1}}

\newcommand{\qjwm}[1][v{-}1]{\mathtt{#1}}
\newcommand{\pqjwm}[1][v{-}1]{\mathtt{#1}}
\newcommand{\pjwm}[1][v{-}1]{\mathtt{#1}}

\newcommand{\Up}[1]{\mathrm{U}_{#1}}
\newcommand{\up}[1]{\mathrm{u}_{#1}}
\newcommand{\upo}[1]{\tilde{\mathrm{u}}_{#1}}
\newcommand{\Down}[1]{\mathrm{D}_{#1}}
\newcommand{\down}[1]{\mathrm{d}_{#1}}
\newcommand{\downo}[1]{\tilde{\mathrm{d}}_{#1}}

\newcommand{\loopdown}[2]{\mathrm{L}^{#1}_{#2}}

\newcommand{\trap}[2]{\tilde{\mathrm{L}}^{{#1}}_{{#2}}}

\newcommand{\ealg}{\zigzag}
\newcommand{\funcf}{\obstuff{f}}
\newcommand{\funcg}{\obstuff{g}}
\newcommand{\funch}{\obstuff{h}}
\newcommand{\funcF}[1][S]{\obstuff{f}_{#1}}
\newcommand{\funcG}[1][S]{\obstuff{g}_{#1}}
\newcommand{\funcH}[1][S]{\obstuff{h}_{#1}}
\newcommand{\hull}[1][S]{\overline{#1}}
\newcommand{\condE}{E}
\newcommand{\condA}{A}
\newcommand{\condO}{O}
\newcommand{\condW}{Z}
\newcommand{\condWm}{Z_{-}}

\newcommand{\tru}[5]{
\draw [JW] (#4+#2/2+#2/4,#5) to (#4+#1,#5) to (#4+#1,#5+#2) to (#4,#5+#2) to (#4+#2/2+#2/4,#5);
\node at (#4+#1/2+#1/16,#5+#2/2-.05) {#3};
}

\newcommand{\trd}[5]{
\draw [JW] (#4+#2/2+#2/4,#5) to (#4+#1,#5) to (#4+#1,#5-#2) to (#4,#5-#2) to (#4+#2/2+#2/4,#5);
\node at (#4+#1/2+#1/16,#5-#2/2-.05) {#3};
}

\newcommand{\tr}[5]{
\draw [JW] (#4+#1,#5) to (#4+#1,#5-#2) 
to (#4,#5-#2) to (#4+#2/2+#2/4,#5) to (#4,#5+#2) to (#4+#1,#5+#2) to (#4+#1,#5);
\node at (#4+#1/2+#1/16,#5-.05) {#3};
}

\newcommand{\TRU}[1]{
\begin{tikzpicture}[anchorbase,tinynodes]
\tru{1.1}{.6}{#1}{0}{0}
\end{tikzpicture}
}

\newcommand{\TRD}[1]{
\begin{tikzpicture}[anchorbase,tinynodes]
\trd{1.1}{.6}{#1}{0}{0}
\end{tikzpicture}
}
\newcommand{\TRUD}[2]{
\begin{tikzpicture}[anchorbase,tinynodes]
\tru{1.1}{.3}{#1}{0}{0}
\trd{1.1}{.3}{#2}{0}{0}
\end{tikzpicture}
}

\newcommand{\TR}[1]{
\begin{tikzpicture}[anchorbase,tinynodes]
\tr{1.1}{.3}{#1}{0}{0}
\end{tikzpicture}
}

\newcommand{\dsoergel}[1][\K]{{#1}\catstuff{S}}
\newcommand{\atypeA}{\tilde{\mathrm{A}}_{1}}

\newcommand{\upop}[1][\ssymbol]{\mathrm{U}_{#1}}
\newcommand{\downop}[1][\ssymbol]{\mathrm{D}_{#1}}

\newtheorem{theoremm}{Theorem}[section]

\declaretheoremstyle[
headfont=\bfseries, 
notebraces={[}{]},
bodyfont=\normalfont\itshape,
headpunct={},
postheadspace=1em,
spacebelow=10pt,
spaceabove=10pt, 
]{ourtheo}

\declaretheoremstyle[
headfont=\normalfont\bfseries,
notefont=\mdseries,
notebraces={(}{)},
bodyfont=\normalfont\slshape,
headpunct={},
postheadspace=1em,
spacebelow=10pt,
spaceabove=10pt, 
]{ourdef}

\declaretheorem[style=ourtheo,name=Theorem,numberlike=theoremm]{theorem}
\declaretheorem[style=ourtheo,name=Lemma,numberlike=theoremm]{lemma}
\declaretheorem[style=ourtheo,name=Proposition,numberlike=theoremm]{proposition}

\declaretheorem[style=ourtheo,name=Lemma,qed=$\square$,numberlike=theoremm]{lemmaqed}
\declaretheorem[style=ourtheo,name=Proposition,qed=$\square$,numberlike=theoremm]{propositionqed}

\declaretheorem[style=ourdef,name=Definition,numberlike=theorem]{definition}
\declaretheorem[style=ourdef,name=Example,numberlike=theorem]{example}
\declaretheorem[style=ourdef,name=Remark,numberlike=theorem]{remark}
\declaretheorem[style=ourdef,name=Convention,numberlike=theorem]{convention}

\declaretheorem[style=ourtheo,name=Theorem]{introtheorem}
\declaretheorem[style=ourtheo,name=Corollary,qed=$\square$]{introcorollary}

\allowdisplaybreaks

\numberwithin{equation}{section}
\usepackage[hypertexnames=false]{hyperref}
\usepackage{cleveref,bookmark}
\hypersetup{
pdftoolbar=true,        
pdfmenubar=true,        
pdffitwindow=false,     
pdfstartview={FitH},    
pdftitle={Quivers for \texorpdfstring{$\mathrm{SL}_{2}$}{SL2} tilting modules},    
pdfauthor={Daniel Tubbenhauer and Paul Wedrich},     
pdfsubject={},   
pdfcreator={Daniel Tubbenhauer and Paul Wedrich},   
pdfproducer={Daniel Tubbenhauer and Paul Wedrich}, 
pdfkeywords={}, 
pdfnewwindow=true,      
colorlinks=true,       
linkcolor=mydarkblue,          
citecolor=teal,        
filecolor=magenta,      
urlcolor=orchid,          
linkbordercolor=lava,
citebordercolor=teal,
urlbordercolor=orchid,  
linktocpage=true
}

\renewcommand{\theequation}{\thesection-\arabic{equation}}
\let\fullref\autoref
\def\makeautorefname#1#2{\expandafter\def\csname#1autorefname\endcsname{#2}}
\makeautorefname{equation}{Equation}%
\makeautorefname{footnote}{footnote}%
\makeautorefname{item}{item}%
\makeautorefname{figure}{Figure}%
\makeautorefname{table}{Table}%
\makeautorefname{part}{Part}%
\makeautorefname{appendix}{Appendix}%
\makeautorefname{chapter}{Chapter}%
\makeautorefname{section}{Section}%
\makeautorefname{subsection}{Section}%
\makeautorefname{subsubsection}{Section}%
\makeautorefname{paragraph}{Paragraph}%
\makeautorefname{subparagraph}{Paragraph}%
\makeautorefname{theorem}{Theorem}%
\makeautorefname{theo}{Theorem}%
\makeautorefname{thm}{Theorem}%
\makeautorefname{addendum}{Addendum}%
\makeautorefname{addend}{Addendum}%
\makeautorefname{add}{Addendum}%
\makeautorefname{maintheorem}{Main theorem}%
\makeautorefname{mainthm}{Main theorem}%
\makeautorefname{corollary}{Corollary}%
\makeautorefname{corol}{Corollary}%
\makeautorefname{coro}{Corollary}%
\makeautorefname{cor}{Corollary}%
\makeautorefname{lemma}{Lemma}%
\makeautorefname{lemm}{Lemma}%
\makeautorefname{lem}{Lemma}%
\makeautorefname{sublemma}{Sublemma}%
\makeautorefname{sublem}{Sublemma}%
\makeautorefname{subl}{Sublemma}%
\makeautorefname{proposition}{Proposition}%
\makeautorefname{proposit}{Proposition}%
\makeautorefname{propos}{Proposition}%
\makeautorefname{propo}{Proposition}%
\makeautorefname{prop}{Proposition}%
\makeautorefname{property}{Property}
\makeautorefname{proper}{Property}
\makeautorefname{scholium}{Scholium}%
\makeautorefname{step}{Step}%
\makeautorefname{conjecture}{Conjecture}%
\makeautorefname{conject}{Conjecture}%
\makeautorefname{conj}{Conjecture}%
\makeautorefname{question}{Question}
\makeautorefname{questn}{Question}
\makeautorefname{quest}{Question}
\makeautorefname{ques}{Question}
\makeautorefname{qn}{Question}
\makeautorefname{definition}{Definition}%
\makeautorefname{defin}{Definition}%
\makeautorefname{defi}{Definition}%
\makeautorefname{def}{Definition}%
\makeautorefname{dfn}{Definition}%
\makeautorefname{notation}{Notation}
\makeautorefname{nota}{Notation}
\makeautorefname{notn}{Notation}
\makeautorefname{remark}{Remark}%
\makeautorefname{rema}{Remark}%
\makeautorefname{rem}{Remark}%
\makeautorefname{rmk}{Remark}%
\makeautorefname{rk}{Remark}%
\makeautorefname{remarks}{Remarks}%
\makeautorefname{rems}{Remarks}%
\makeautorefname{rmks}{Remarks}%
\makeautorefname{rks}{Remarks}%
\makeautorefname{example}{Example}%
\makeautorefname{examp}{Example}%
\makeautorefname{exmp}{Example}%
\makeautorefname{exam}{Example}%
\makeautorefname{exa}{Example}%
\makeautorefname{algorithm}{Algorithm}%
\makeautorefname{algo}{Algorithm}%
\makeautorefname{alg}{Algorithm}%
\makeautorefname{axiom}{Axiom}%
\makeautorefname{axi}{Axiom}%
\makeautorefname{ax}{Axiom}%
\makeautorefname{case}{Case}%
\makeautorefname{claim}{Claim}%
\makeautorefname{clm}{Claim}%
\makeautorefname{assumption}{Assumption}%
\makeautorefname{assumpt}{Assumption}%
\makeautorefname{conclusion}{Conclusion}%
\makeautorefname{concl}{Conclusion}%
\makeautorefname{conc}{Conclusion}%
\makeautorefname{condition}{Condition}%
\makeautorefname{condit}{Condition}%
\makeautorefname{cond}{Condition}%
\makeautorefname{construction}{Construction}%
\makeautorefname{construct}{Construction}%
\makeautorefname{const}{Construction}%
\makeautorefname{cons}{Construction}%
\makeautorefname{criterion}{Criterion}%
\makeautorefname{criter}{Criterion}%
\makeautorefname{crit}{Criterion}%
\makeautorefname{exercise}{Exercise}%
\makeautorefname{exer}{Exercise}%
\makeautorefname{exe}{Exercise}%
\makeautorefname{problem}{Problem}%
\makeautorefname{problm}{Problem}%
\makeautorefname{probm}{Problem}%
\makeautorefname{prob}{Problem}%
\makeautorefname{solution}{Solution}%
\makeautorefname{soln}{Solution}%
\makeautorefname{sol}{Solution}%
\makeautorefname{summary}{Summary}%
\makeautorefname{summ}{Summary}%
\makeautorefname{sum}{Summary}%
\makeautorefname{operation}{Operation}%
\makeautorefname{oper}{Operation}%
\makeautorefname{observation}{Observation}%
\makeautorefname{observn}{Observation}%
\makeautorefname{obser}{Observation}%
\makeautorefname{obs}{Observation}%
\makeautorefname{ob}{Observation}%
\makeautorefname{convention}{Convention}%
\makeautorefname{convent}{Convention}%
\makeautorefname{conv}{Convention}%
\makeautorefname{cvn}{Convention}%
\makeautorefname{warning}{Warning}%
\makeautorefname{warn}{Warning}%
\makeautorefname{note}{Note}%
\makeautorefname{fact}{Fact}%
\makeautorefname{hope}{Expectation}%
\makeautorefname{hope2}{Conjecture}%

\newcommand{\nnfootnote}[1]{%
\begin{NoHyper}
\renewcommand\thefootnote{}\footnote{#1}%
\addtocounter{footnote}{-1}%
\end{NoHyper}
}

\begin{document}
\title[Quivers for \texorpdfstring{$\mathrm{SL}_{2}$}{SL2} tilting modules]
{Quivers for \texorpdfstring{$\mathrm{SL}_{2}$}{SL2} tilting modules}
\author[Daniel Tubbenhauer and Paul Wedrich]{Daniel Tubbenhauer and Paul Wedrich}

\address{D.T.: Institut f{\"u}r Mathematik, Universit{\"a}t Z{\"u}rich, 
Winterthurerstrasse 190, Campus Irchel, Office Y27J32, CH-8057 Z{\"u}rich, 
Switzerland, \href{www.dtubbenhauer.com}{www.dtubbenhauer.com}}
\email{daniel.tubbenhauer@math.uzh.ch}

\address{P.W.: Mathematical Sciences Institute, The Australian National University, 
Hanna Neumann Building,
Canberra ACT 2601, Australia, \href{http://paul.wedrich.at}{paul.wedrich.at}}
\email{p.wedrich@gmail.com}

\nnfootnote{\textit{Mathematics Subject Classification 2010.} Primary: 20G05, 20C20; Secondary: 16D90, 17B10, 20G40.}
\nnfootnote{\textit{Keywords.} Modular representation theory, tilting modules, 
diagrammatic algebra, generators and relations, Temperley--Lieb, positive characteristic.}

\begin{abstract}
Using diagrammatic methods, we define a quiver with relations depending on a prime $\mathsf{p}$ and 
show that the associated path algebra describes 
the category of tilting modules for 
$\mathrm{SL}_{2}$ in characteristic $\mathsf{p}$.
Along the way we obtain a presentation for 
morphisms between $\mathsf{p}$-Jones--Wenzl projectors.
\end{abstract}

\maketitle

\vspace{-.8cm}

\tableofcontents

\renewcommand{\theequation}{\thesection-\arabic{equation}}

\addtocontents{toc}{\protect\setcounter{tocdepth}{1}}

\vspace{-.8cm}

\section{Introduction}\label{sec:intro}

Let $\K$ denote an algebraically closed field and $\tilt=\tilt\big(\SLtwo\big)$ the additive, 
$\K$-linear category of (left-)tilting modules for the algebraic group $\SLtwo$. 
This category can be described as the full subcategory of $\SLtwo$-modules 
which is monoidally generated by the vector representation $\tmod(1)\cong\K^{2}$, and which is closed under
taking finite direct sums and direct summands.

The purpose of this paper is to give a \emph{generators and relations
presentation} of $\tilt$ by identifying it with the category of projective
modules for the path algebra of an explicitly described quiver with relations.
This quiver can be interpreted as the \emph{semi-infinite Ringel dual} of
$\SLtwo$ in the sense of \cite{BS18}. For $\K$ of characteristic zero this is
trivial as $\tilt$ is semisimple, and the indecomposable tilting modules are
indeed the simple modules. The quantum analog at a complex root of unity is
related to the zigzag algebra with vertex set $\N[0]$ and a starting condition, see
\eg \cite{AnTu-tilting}. 

The focus of this paper is on the case of positive characteristic $\ppar$, for which we represent
$\tilt$ as a quotient $\zigzag=\zigzag_{\ppar}$ of the path algebra of an infinite, fractal-like quiver, 
a truncation of which is illustrated for $\ppar=3$ in \fullref{figure:main-3}.

\begin{figure}[ht]
\includegraphics[scale=.3]{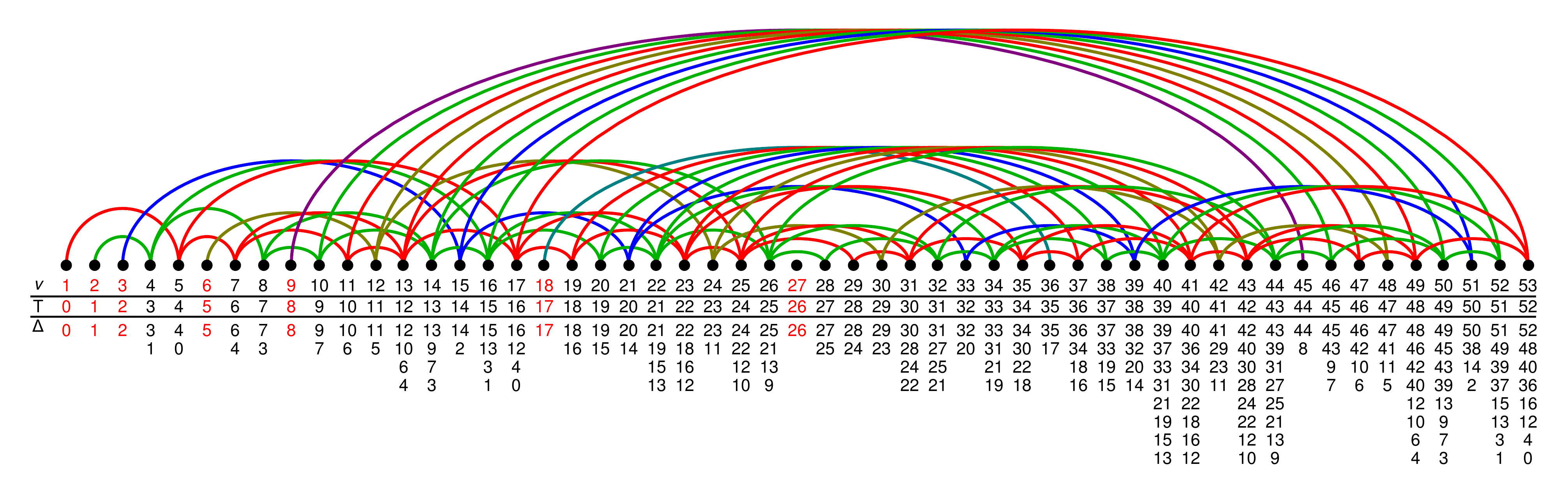}
\caption{The full subquiver containing the first 
$53$ vertices of the quiver underlying $\zigzag_{3}$.}
\label{figure:main-3}
\end{figure}

\subsection*{The main result}

From now on let $\K$ be an algebraically closed field of characteristic $\ppar>0$, and 
let $\SLtwo$ be the corresponding special linear group.
Recall that the indecomposable tilting modules for $\SLtwo$ are classified (up to isomorphism) by their highest weight 
$v-1\in\N[0]$, and we pick a collection of representatives denoted by $\tmod(v-1)$.

\begin{introtheorem}\label{theorem:main}
There is an algebra isomorphism
\begin{gather*}
\mainfunctor\colon\zigzag\xrightarrow{\cong}
{\textstyle\bigoplus_{v,w\in\N}} \Hom_{\tilt}\big(\tmod(v-1),\tmod(w-1)\big),
\end{gather*}
which sends the constant path on the vertex $v-1$ to the idempotent for the summand $\tmod(v-1)$.
\end{introtheorem}

Let $\zigzagmod$ denote the
category of finitely-generated, projective (right-)modules for $\zigzag$. By semi-infinite Ringel duality \cite[Section 4]{BS18}, we have the following consequence of \fullref{theorem:main}.

\begin{introcorollary}\label{corollary:main}
There is an equivalence of additive, $\K$-linear categories
\begin{gather*}
\mainfunctor^{\prime}\colon
\tilt\xrightarrow{\cong}\zigzagmod,
\end{gather*}
sending indecomposable tilting modules to indecomposable projectives.
\end{introcorollary}

Several classical facts about $\SLtwo$-modules are reflected in the
presentation of the algebra $\zigzag$. For example, a path from $v-1$ to $w-1$
can only be non-zero in $\zigzag$ if $\tmod(v-1)$ and $\tmod(w-1)$ share a
common Weyl factor.
More specifically, if the $\ppar$-adic expansion
$v=\pbase{a_{j},\dots,a_{0}}{\ppar}=\sum_{i=0}^{j}a_{i}\ppar^{i}$ has exactly
$r+1$ non-zero digits, then $\tmod(v-1)$ has exactly $2^{r}$ Weyl factors
$\wmod(w-1)$ where $w$ is obtained by negating some of the non-zero digits $a_i$
for $i<j$. In this case, $\zigzag$ contains $r$ arrows from $v-1$ to those
$w-1<v-1$ that are obtained by negating a single digit.

Our assignment of morphisms to arrows uses the Temperley--Lieb
category. In contrast to other descriptions of morphisms between indecomposable
tilting modules for $\SLtwo$, this presentation of $\zigzag$ is well-adapted to
study $\tilt$ as a monoidal category.

The Weyl factors in
indecomposable tilting modules are illustrated in the lines $v$, $\tmod$,
$\wmod$ in \fullref{figure:main-3}, where the colors distinguish arrows in
different blocks, the connected components of the quiver, and each reddish
number corresponds to the unique simple tilting module in its block.

\subsection*{The algebra \texorpdfstring{$\zigzag$}{Z} in a nutshell}

We define the algebra $\zigzag$ as a quotient of 
the path algebra of an infinite, fractal-like quiver over the prime field $\F\subset\K$. (In particular, we can always extend the algebra $\zigzag$ to an algebra over $\K$.)
We will use this introduction to sketch the main features of $\zigzag$ 
and relegate the precise statement to \fullref{theorem:main-tl-section}.

\begin{enumerate}[label=$\bullet$]

\setlength\itemsep{.15cm}

\item \textit{The underlying quiver.} 
We identify the vertex set with $\N[0]$ and the constant path at the vertex $v-1$ 
will be denoted $\idemy[v{-}1]$ (it corresponds to $\tmod(v-1)$). 
If $v=\pbase{a_{j},\dots,a_{0}}{\ppar}$, then for every 
digit $a_{i}\neq 0$ 
with $i\neq j$ there is a pair of arrows
\begin{gather*}
\downop[{i}]\idemy[v{-}1]\colon
(v-1)\to({v[i]-1}),
\quad\quad
\upop[{i}]\idemy[{v[i]{-}1}]\colon
({v[i]-1})\to (v-1),
\end{gather*}
where $v[i]=\pbase{a_{j},\dots,a_{i+1},-a_{i},a_{i-1},\dots,a_{0}}{\ppar}
=v-2a_{i}\ppar^{i}$.

\item\textit{Some relations.}
Up to some additional rules in special cases
(which we ignore for the sake of this introduction),
there are five types of relations among paths, 
which hold whenever both sides are defined and satisfy certain admissibility conditions.
\smallskip
\begin{enumerate}[label=(\arabic*)]

\setlength\itemsep{.15cm}

\item\emph{Idempotents.}
$\idemy[{v{-}1}]\idemy[{w{-}1}]=\delta_{v,w}\idemy[{v{-}1}]$,
$\idemy[{w{-}1}]\morstuff{F}\idemy[{v{-}1}]=\morstuff{F}\idemy[{v{-}1}]$
and $\idemy[{w{-}1}]\morstuff{F}\idemy[{v{-}1}]=\idemy[{w{-}1}]\morstuff{F}$, 
where $\morstuff{F}$ is a word in the generators starting at $v-1$ and ending at $w-1$.
(Throughout, we use such relations to absorb all but one idempotent in each string of generators.)

\item\emph{Nilpotency.} $\downop[{i}]^{2}\idemy[{v{-}1}]=\upop[{i}]^{2}\idemy[{v{-}1}]=0$.

\item\emph{Far-commutativity.} 
$\downop[{i}]\downop[{j}]\idemy[{v{-}1}]=\downop[{j}]\downop[{i}]\idemy[{v{-}1}]$, 
$\upop[{i}]\downop[{j}]\idemy[{v{-}1}]=\downop[{j}]\upop[{i}]\idemy[{v{-}1}]$, as well as 
$\upop[{i}]\upop[{j}]\idemy[{v{-}1}]=\upop[{j}]\upop[{i}]\idemy[{v{-}1}]$ whenever $|i-j|>1$.

\item\emph{Adjacency relations.}
$\downop[{i+1}]\upop[{i}]\idemy[{v{-}1}]=\downop[{i}]\downop[{i+1}]\idemy[{v{-}1}]$ and 
$\downop[{i}]\upop[{i+1}]\idemy[{v{-}1}]=\upop[{i+1}]\upop[{i}]\idemy[{v{-}1}]$, 
and scaled versions $\downop[{i+1}]\downop[{i}]\idemy[{v{-}1}]=\funcg^{\prime}\upop[{i}]\downop[{i+1}]\idemy[{v{-}1}]$ 
and $\upop[{i}]\upop[{i+1}]\idemy[{v{-}1}]=\funcg^{\prime\prime}\upop[{i+1}]\downop[{i}]\idemy[{v{-}1}]$.

\item\emph{Zigzag.} 
$\downop[{i}]\upop[{i}]\idemy[{v{-}1}]=\funcg\upop[{i}]\downop[{i}]\idemy[{v{-}1}] 
+\funcf\upop[{i+1}]\upop[{i}]\downop[{i}]\downop[{i+1}]\idemy[{v{-}1}]$.

\end{enumerate}

\noindent Here $\funcg$, $\funcg^{\prime}$, $\funcg^{\prime\prime}$ and $\funcf$ are scalars 
that depend on $\ppar$ and the digit $a_{i+1}$.

\item\emph{Hom spaces.} 
For $v,w\in\N$ the $\F$-vector space  
$\idemy[{w{-}1}]\zigzag\idemy[{v{-}1}]$ is spanned by paths of the 
form $\idemy[{w{-}1}]\upop[{j_{k}}]\cdots
\upop[{j_{1}}]\downop[{i_{1}}]\cdots\downop[{i_{l}}]\idemy[{v{-}1}]$ with $
j_{k}>\cdots>j_{1},i_{1}<\cdots<i_{l}$,
\ie paths that descend before ascending again. In particular, we have $\idemy[{v{-}1}]\zigzag\idemy[{v{-}1}]\cong\F$ 
whenever $v=a\ppar^{j}$ for $1\leq a<\ppar$, 
which reflects the fact that the corresponding tilting module $\tmod(v-1)$ is simple.

\item\emph{Endomorphism algebras.} 
Let $v>0$ have $r+1$ non-zero digits with indices $i_{r+1}>\cdots>i_{1}$. Then we have
the following identifications of $\F$-algebras
\begin{gather*}
\begin{gathered}
\idemy[{v{-}1}]\zigzag\idemy[{v{-}1}]
\cong
\F[][\upop[{i_{r}}]\downop[{i_{r}}],\dots,\upop[{i_{1}}]\downop[{i_{1}}]]
\Big/
\big\langle(\upop[{i_{r}}]\downop[{i_{r}}])^{2},\dots,
(\upop[{i_{1}}]\downop[{i_{1}}])^{2}\big\rangle.
\end{gathered}
\end{gather*} 
This leads to a description of the endomorphism algebra of $\tmod(v-1)$ 
which could have been expected from Donkin's tensor product 
theorem \cite[Proposition 2.1]{Do-tilting-alg-groups}.

\end{enumerate} 

We would like to highlight that we will meet a \emph{law of small primes} (\losp) repeatedly. 
By this we mean the appearance of exceptional relations in cases of $\ppar$-adic expansions with digits 
$0$, $1$, $\ppar-2$, or $\ppar-1$. 
These relations are exceptional in the sense that they contrast with the relations shown above, which describe 
the behavior of \emph{generic} $\ppar$-adic expansions for large primes $\ppar$. 
Nevertheless, exceptional relations are relevant for all primes, and for $\ppar=2$ only exceptional relations apply.

\subsection*{A word about the proof of \fullref{theorem:main}}

The basis for our work is the classical fact that the \emph{Temperley--Lieb} algebra controls the 
finite-dimensional representation theory of $\SLtwo$. 
The second main ingredient is an explicit description 
of $\ppar$-Jones--Wenzl projectors \cite{BuLiSe-tl-char-p}, 
which are characteristic $\ppar$ analogs of the classical Jones--Wenzl projectors, 
that diagrammatically encode the projections 
\begin{gather*}
\tmod(1)^{\otimes(v-1)}\to\tmod(v-1)\to\tmod(1)^{\otimes(v-1)}. 
\end{gather*}

The bulk of this paper is devoted to a careful study 
of morphisms between $\ppar$-Jones--Wenzl projectors over $\F$ 
and the linear relations between them. 
This work was supported by extensive computer experimentation using 
Mathematica and SageMATH.

\subsection*{Relations to other work}

To the best of our knowledge, the quiver underlying the tilting category is new:
We study $\tilt$ as a finitely presented category. So our main concern are the 
\emph{relations among composites of generating morphisms}, 
rather than just the \emph{combinatorics of objects} or the 
\emph{dimensions of morphism spaces}, which appear in the classical literature. 

We would like to acknowledge and reinforce that the $\SLtwo$ representation theory is, 
of course, well-understood on the level of the modules, 
see \eg \cite{CaCl-submodule-weyl-a1}, 
\cite{AnJoLa-sl2-projectives}, \cite{Do-tilting-alg-groups},
\cite{ErHe-ringel-schur}, \cite{ErHe-ringel-schur-symmetric-group} and \cite{DoHe-char-p-sl2}. 
Further, various other quivers associated to $\SLtwo$ are known, describing \eg 
rational modules \cite{MiTu-rational-gl2} or the
extension algebra for simple \cite{MiTu-simple-ext-gl2} or Weyl modules 
\cite{MiTu-weyl-ext-gl2}.

\subsection*{A graded extension and translation functors}

It is possible to give a similar quiver description of $\tilt$ as a positively graded module category of 
the diagrammatic Soergel category $\dsoergel$ for the Weyl group of type $\atypeA$, acting by translation functors. 
The first step in such an extension uses the quantum Satake equivalence (at $q=1$) \cite{El-q-satake} to
connect the Temperley--Lieb diagrammatic calculus to $\dsoergel$. 
In fact, $\zigzag$ faithfully describes the degree zero part of the antispherical module category for $\dsoergel$.
The second step uses ideas from \cite{RiWi-tilting-p-canonical} to relate $\dsoergel$ and the principal block $\tilt_{0}\subset\tilt$
as long as $\ppar>2$. Along this route, $\tilt$ also inherits a grading from $\dsoergel$.

In this case, the algebra $\zigzag$ essentially describes the degree zero part 
of the principal block $\tilt_{0}$, 
while the positive degree part is generated by additional degree
$1$ arrows $\upop\colon v\to v+1$ and $\downop\colon v+1\to v$, 
which interact non-trivially with other paths. 
Note another fractal-like structure: $\zigzag$ describes $\tilt$, but also the 
degree zero part of $\tilt_{0}\subset\tilt$.
We will not pursue this extension in the present paper.

\subsection*{Characteristic zero and higher rank cases}

Throughout we could allow the case of characteristic zero, for which $\tilt$ is semisimple. 
In a more interesting variant, one replaces $\SLtwo$ by its quantum group analog at a complex root of unity, 
using the Jones--Wenzl projectors from \cite{GoWe-tl-root-of-unity}. The role of $\zigzag$ 
is then played by the zigzag algebra with vertex set $\N[0]$ 
and a starting condition, and we would recover a result of \cite{AnTu-tilting}. 
In this sense we think of $\zigzag$ as a positive characteristic version of the zigzag algebra.

We also like to highlight that, to the best of our knowledge, a quiver 
underlying tilting modules for higher rank groups is still unknown, even 
for the quantum group analog in characteristic zero, \cf \cite[Section 5C]{MaMaMiTu-trihedral} 
for some first steps in this direction. 

We expect the diagrammatic methods used in this paper to 
generalize to $\mathrm{SL}_{N}(\K)$ and $\mathrm{GL}_{N}(\K)$. This would involve developing characteristic $\ppar$ analogs 
of so-called \emph{clasps}, living in the corresponding web calculus, 
see \eg \cite{CaKaMo-webs-skew-howe} or \cite{TuVaWe-super-howe}, 
defined over $\F$.

\medskip

\noindent\textbf{Acknowledgments.} D.T. likes to thank Henning Haahr Andersen,
Ben Elias, Hankyung Ko, Catharina Stroppel and Geordie Williamson for teaching
him everything he knows about tilting modules (which is basically nothing, but
that is his fault alone). P.W. would like to thank James Borger, Ben Elias,
Alexandra Grant, Anthony Licata and Scott Morrison for useful discussions.
Special thanks to Henning Haahr Andersen, Jieru Zhu and an anonymous referee for valuable
comments on a draft of this paper, to Nicolas Libedinsky for sharing a draft of
the lovely paper \cite{BuLiSe-tl-char-p}, which started this project,
and Noah Snyder for pointing out a missing argument in the proof of
\fullref{proposition:TLtilt}.

Parts of this research were conducted while D.T. visited the Australian National University, 
supported by the Mathematical Sciences Research Visitor Program (MSRVP), 
and while P.W. visited Universit{\"a}t Z{\"u}rich. Their hospitality and support
are gratefully acknowledged. P.W. was supported by Australian Research Council grants `Braid groups and higher representation theory'
DP140103821 and `Low dimensional categories' DP160103479

\section{The Temperley--Lieb calculus}\label{sec:TL} 

Let $\catstuff{C}=(\catstuff{C},\hcirc,\munit_{\catstuff{C}},\hspace*{-.1cm}\flip)$ be a pivotal 
category with (strict) monoidal composition $\hcirc$, unit $\munit_{\catstuff{C}}$, 
and duality $\hspace*{-.1cm}\flip$. 
We usually write $\morstuff{F}\morstuff{G}:=\morstuff{F}\vcirc\morstuff{G}$ for the composition of morphisms.
We read string diagrams for morphisms in $\catstuff{C}$ from bottom to top and left to right, \eg
\begin{gather*}
(\idmor\hcirc\morstuff{G})(\morstuff{F}\hcirc\idmor) 
=
\begin{tikzpicture}[anchorbase,scale=.22,tinynodes]
\draw[thin,black,fill=mygray,fill opacity=.35] (-5.5,2) rectangle (-2,.5);
\draw[thin,black,fill=mygray,fill opacity=.35] (-.5,3) rectangle (3,4.5);
\draw[thick,densely dotted] (-5.5,2.5) node[left,yshift=-2pt]{$\vcirc$} 
to (3.5,2.5) node[right,yshift=-2pt]{$\vcirc$};
\draw[thick,densely dotted] (-1.25,5) 
node[above,yshift=-2pt]{$\hcirc$} to (-1.25,0) node[below]{$\hcirc$};
\draw[usual] (-5,2) to (-5,3.75) node[right,xshift=1pt]{$\dots$} to (-5,5);
\draw[usual] (-2.5,2) to (-2.5,5);
\draw[usual] (0,0) to (0,1.25) node[right,xshift=1pt]{$\dots$} to (0,3);
\draw[usual] (2.5,0) to (2.5,3);
\draw[usual] (-5,0) to (-5,.25) node[right,xshift=1pt]{$\dots$} to (-5,.5);
\draw[usual] (-2.5,0) to (-2.5,.5);
\draw[usual] (2.5,4.5) to (2.5,5);
\draw[usual] (0,4.5) to (0,4.75) node[right,xshift=1pt]{$\dots$} to (0,5);
\node at (-3.75,1.0) {$\morstuff{F}$};
\node at (1.25,3.5) {$\morstuff{G}$};
\end{tikzpicture}
=
\begin{tikzpicture}[anchorbase,scale=.22,tinynodes]
\draw[thin,black,fill=mygray,fill opacity=.35] (-5.5,1.75) rectangle (-2,3.25);
\draw[thin,black,fill=mygray,fill opacity=.35] (-.5,1.75) rectangle (3,3.25);
\draw[usual] (-5,3.25) to (-5,4.125) node[right,xshift=1pt]{$\dots$} to (-5,5);
\draw[usual] (-2.5,3.25) to (-2.5,5);
\draw[usual] (0,0) to (0,.625) node[right,xshift=1pt]{$\dots$} to (0,1.75);
\draw[usual] (2.5,0) to (2.5,1.75);
\draw[usual] (-5,0) to (-5,.625) node[right,xshift=1pt]{$\dots$} to (-5,1.75);
\draw[usual] (-2.5,0) to (-2.5,1.75);
\draw[usual] (2.5,3.25) to (2.5,5);
\draw[usual] (0,3.25) to (0,4.125) node[right,xshift=1pt]{$\dots$} to (0,5);
\node at (-3.75,2.25) {$\morstuff{F}$};
\node at (1.25,2.25) {$\morstuff{G}$};
\end{tikzpicture}
=
\begin{tikzpicture}[anchorbase,scale=.22,tinynodes]
\draw[thin,black,fill=mygray,fill opacity=.35] (5.5,2) rectangle (2,.5);
\draw[thin,black,fill=mygray,fill opacity=.35] (.5,3) rectangle (-3,4.5);
\draw[thick, densely dotted] (5.5,2.5) 
node[right,yshift=-2pt]{$\vcirc$} to (-3.5,2.5) node[left,yshift=-2pt]{$\vcirc$};
\draw[thick, densely dotted] (1.25,5) 
node[above,yshift=-2pt]{$\hcirc$} to (1.25,0) node[below]{$\hcirc$};
\draw[usual] (2.5,2) to (2.5,3.75) node[right,xshift=1pt]{$\dots$} to (2.5,5);
\draw[usual] (5,2) to (5,5);
\draw[usual] (0,0) to (0,3);
\draw[usual] (-2.5,0) to (-2.5,1.25) node[right,xshift=1pt]{$\dots$} to (-2.5,3);
\draw[usual] (2.5,0) to (2.5,.25) node[right,xshift=1pt]{$\dots$} to (2.5,.5);
\draw[usual] (5,0) to (5,.5);
\draw[usual] (-2.5,4.5) to (-2.5,4.75) node[right,xshift=1pt]{$\dots$} to (-2.5,5);
\draw[usual] (0,4.5) to (0,5);
\node at (3.75,1.0) {$\morstuff{G}$};
\node at (-1.25,3.5) {$\morstuff{F}$};
\end{tikzpicture}
=(\morstuff{F}\hcirc\idmor)(\idmor\hcirc\morstuff{G}).
\end{gather*} 
The duality maps are pictured as cup and cap string diagrams, subject to the expected string-straightening relations. 
The pivotal structure additionally allows the rotation of string diagrams 
and guarantees that planar-isotopic diagrams represent the same morphism.

Let $\Ring$ be any commutative and unital ring.
(For us $\Ring$ will usually be $\Q$ or $\F\subset\K$, 
the prime field of $\K$. However, it also makes sense to 
formulate everything for $\Q_{\ppar}$ and $\Z_{\ppar}$.)

The category $\TL[\Ring]$ (see \eg \cite{KaLi-TL-recoupling}) can be described as the pivotal
$\Ring$-linear category with objects indexed by $m\in\N[0]$, and with morphisms
from $m$ to $n$ being $\Ring$-linear combinations of unoriented string diagrams
drawn in a horizontal strip $\R\times[0,1]$ between $m$ marked points on the
lower boundary $\R\times\{0\}$ and $n$ marked points on the upper boundary
$\R\times\{1\}$, considered up to planar isotopy relative to the boundary and
the relation that a circle evaluates to $-2$. The composition and tensor product
operations are as described above.

Particular cases of the isotopy and circle relations are
\begin{gather*}
\begin{tikzpicture}[anchorbase,scale=.25,tinynodes]
\draw[usual] (0,-2) to (0,0) to[out=90,in=180] (1,1) to[out=0,in=90] (2,0);
\draw[usual] (2,0) to[out=270,in=180] (3,-1) to[out=0,in=270] (4,0) to (4,2);
\end{tikzpicture}
=
\begin{tikzpicture}[anchorbase,scale=.25,tinynodes]
\draw[usual] (0,-2) to (0,2);
\end{tikzpicture}
\;,\quad
\begin{tikzpicture}[anchorbase,scale=.25,tinynodes]
\draw[usual] (0,-2) to (0,0) to[out=90,in=0] (-1,1) to[out=180,in=90] (-2,0);
\draw[usual] (-2,0) to[out=270,in=0] (-3,-1) to[out=180,in=270] (-4,0) to (-4,2);
\end{tikzpicture}
=
\begin{tikzpicture}[anchorbase,scale=.25,tinynodes]
\draw[usual] (0,-2) to (0,2);
\end{tikzpicture}
\;,\quad
\begin{tikzpicture}[anchorbase,scale=.25,tinynodes]
\draw[usual] (0,0) to[out=90,in=180] (1,1) to[out=0,in=90] (2,0);
\draw[usual] (0,0) to[out=270,in=180] (1,-1) to[out=0,in=270] (2,0);
\end{tikzpicture}
=-2.
\end{gather*}

In the following we will use 
labeled strands as shorthand notation for bundles of parallel strands:
\begin{gather*}
\begin{tikzpicture}[anchorbase,scale=.25,tinynodes]
\draw[usual] (0,0) to (0,1.5) node[right,xshift=-2pt]{$m$} to (0,3);
\node at (1,-1) {$\phantom{a}$};
\node at (1,4) {$\phantom{a}$};
\end{tikzpicture}
:=
\idtl[m]
=
\begin{tikzpicture}[anchorbase,scale=.25,tinynodes]
\draw[usual] (0,0) to (0,3);
\draw[usual] (2,0) to (2,3);
\node at (1,1.5) {$...$};
\node at (1,-1) {$\phantom{a}$};
\node at (1,4) {$\text{$m$ strands}$};
\end{tikzpicture}
,\quad
\begin{tikzpicture}[anchorbase,scale=.25,tinynodes]
\draw[usual] (0,0) to[out=90,in=180] (1,1) node[above,yshift=-2pt]{$m$} to[out=0,in=90] (2,0);
\node at (1,-1) {$\phantom{a}$};
\node at (1,4) {$\phantom{a}$};
\end{tikzpicture}
:=
\begin{tikzpicture}[anchorbase,scale=.25,tinynodes]
\draw[usual] (0,0) to[out=90,in=180] (1,1) to[out=0,in=90] (2,0);
\draw[usual] (-2,0) to[out=90,in=180] (1,3) to[out=0,in=90] (4,0);
\node at (-1,.5) {$...$};
\node at (3,.5) {$...$};
\node at (1,-1) {$\phantom{a}$};
\node at (1,4) {$\text{$m$ caps}$};
\end{tikzpicture}
,\quad
\begin{tikzpicture}[anchorbase,scale=.25,tinynodes]
\draw[usual] (0,3) to[out=270,in=180] (1,2) node[below]{$m$} to[out=0,in=270] (2,3);
\node at (1,-1) {$\phantom{a}$};
\node at (1,4) {$\phantom{a}$};
\end{tikzpicture}
:=
\begin{tikzpicture}[anchorbase,scale=.25,tinynodes]
\draw[usual] (0,3) to[out=270,in=180] (1,2) to[out=0,in=270] (2,3);
\draw[usual] (-2,3) to[out=270,in=180] (1,0) to[out=0,in=270] (4,3);
\node at (-1,2.5) {$...$};
\node at (3,2.5) {$...$};
\node at (1,-1) {$\phantom{a}$};
\node at (1,4) {$\text{$m$ cups}$};
\end{tikzpicture}
.
\end{gather*}
We even omit these numbers or the lines altogether if no confusion can arise.

The category $\TL[\Ring]$ furthermore admits a contravariant, $\Ring$-linear 
involution which reflects string diagrams in a horizontal line. 
Several arguments in the following will use this up-down symmetry. However, we will usually not have a left-right symmetry.

Recall that
$\Hom_{\TL[\Ring]}(m,n)$ is a free $\Ring$-module 
with a basis $\sbas$ given by crossingless matchings.
The \emph{through-degree} $\td{\sbasel_{i}}$ of $\sbasel_{i}\in\sbas$ is the number of strands connecting the bottom to the top. 
More generally, the through-degree of a general morphism 
$\morstuff{F}=\sum_{\sbasel_{i}\in\sbas}\sbasc_{i}\sbasel_{i}$ 
is defined via $\td{\morstuff{F}}:=\max\{\td{\sbasel_{i}}\mid\sbasc_{i}\neq 0\}$.
Note that $\td{\morstuff{FG}}\leq\min\big(\td{\morstuff{F}},\td{\morstuff{G}}\big)$, and thus, 
$\catstuff{td}_{i}(m,n):=\{\morstuff{F}\in\Hom_{\TL[\Ring]}(m,n)\mid\td{\morstuff{F}}\leq i\}$ form 
a sequence of nested ($\vcirc$-)ideals in $\TL[\Ring]$.

Instead of $m$, the number of strands, let us now use $v=m+1\in\N$, which will be crucial number for everything that follows.

\begin{definition}\label{definition:JW}
For $v\in\N$ the JW projectors $\qjw\in\Hom_{\TL[\Q]}(v{-}1,v{-}1)$ are the morphisms, which are
recursively defined by
\begin{gather}\label{eq:jw-recursion}
\qjw[0]
:=\emptyset\, ,
\quad
\qjw[1]:=
\begin{tikzpicture}[anchorbase,scale=.25,tinynodes]
\draw[usual] (0,0) to (0,2);
\end{tikzpicture}\,
,\quad
\qjw
:=
\begin{tikzpicture}[anchorbase,scale=.25,tinynodes]
\draw[JW] (-1.5,0) rectangle (1.5,2);
\node at (0,.9) {$\qjwm$};
\end{tikzpicture}
:=
\begin{tikzpicture}[anchorbase,scale=.25,tinynodes]
\draw[JW] (-1.5,0) rectangle (1.5,2);
\node at (0,.9) {$\qjwm[{v{-}2}]$};
\draw[usual] (-2,0) to (-2,2);
\end{tikzpicture}
+
\tfrac{v{-}2}{v{-}1}\cdot
\begin{tikzpicture}[anchorbase,scale=.25,tinynodes]
\draw[JW] (-2.5,1) rectangle (2.5,3);
\node at (0,1.9) {$\qjwm[{v{-}2}]$};
\draw[JW] (-.5,-1) rectangle (2.5,1);
\node at (1,-.1) {$\qjwm[{v{-}3}]$};
\draw[usual] (-1.5,1) to[out=270, in=0] (-2.25,.25) to[out=180, in=270] (-3,1) to (-3,3);
\draw[usual] (-1.5,-1) to[out=90, in=0] (-2.25,-.25) to[out=180, in=90] (-3,-1) to (-3,-3);
\draw[JW] (-2.5,-1) rectangle (2.5,-3);
\node at (0,-2.1) {$\qjwm[{v{-}2}]$};
\end{tikzpicture}
\quad
\text{ if }v>2,
\end{gather}
where we use a box with $v-1$ 
bottom and top strands to indicate $\qjw$.
\end{definition}

\begin{lemmaqed}(See \eg \cite[Section
3]{KaLi-TL-recoupling}.)\label{lemma:0-properties} We have
$\flip[{(\qjw)}]=\qjw$ and $\td{\qjw}=v-1$. Furthermore, the following
properties hold, which are best expressed diagrammatically.\\
\noindent
\begin{minipage}{.33\textwidth}
\begin{gather}\label{eq:0absorb}
\begin{tikzpicture}[anchorbase,scale=.25,tinynodes]
\draw[JW] (-.6,0) rectangle (1.6,2);
\draw[usual] (-1,0) to (-1,2);
\draw[usual] (2,0) to (2,2);
\node at (.5,.9) {$\qjwm[w{-}1]$};
\draw[JW] (-1.5,-2) rectangle (2.5,0);
\node at (0.5,-1.1) {$\qjwm$};
\node at (0,2.5) {$\phantom{a}$};
\node at (0,-2.5) {$\phantom{a}$};
\end{tikzpicture}
=
\begin{tikzpicture}[anchorbase,scale=.25,tinynodes]
\draw[JW] (-1.5,-1) rectangle (2.5,1);
\node at (0.5,-.1) {$\qjwm$};
\node at (0,2.5) {$\phantom{a}$};
\node at (0,-2.5) {$\phantom{a}$};
\end{tikzpicture}
=
\begin{tikzpicture}[anchorbase,scale=.25,tinynodes]
\draw[JW] (-.6,0) rectangle (1.6,-2);
\draw[usual] (-1,0) to (-1,-2);
\draw[usual] (2,0) to (2,-2);
\node at (.5,-1.1) {$\qjwm[w{-}1]$};
\draw[JW] (-1.5,0) rectangle (2.5,2);
\node at (0.5,0.9) {$\qjwm$};
\node at (0,2.5) {$\phantom{a}$};
\node at (0,-2.5) {$\phantom{a}$};
\end{tikzpicture},
\end{gather}
\end{minipage}
\begin{minipage}{.27\textwidth}
\begin{gather}\label{eq:0kill}
\quad
\begin{tikzpicture}[anchorbase,scale=.25,tinynodes]
\draw[JW] (-1.2,-1) rectangle (1.2,1);
\node at (0,-.1) {$\qjwm$};
\draw[usual] (-.5,1) to[out=90,in=180] (0,1.5) node[above,yshift=-2pt]{$k$} to[out=0,in=90] (.5,1);
\node at (0,2.5) {$\phantom{a}$};
\node at (0,-2.5) {$\phantom{a}$};
\end{tikzpicture}
=
0
=
\begin{tikzpicture}[anchorbase,scale=.25,tinynodes]
\draw[JW] (-1.2,1) rectangle (1.2,-1);
\node at (0,-.1) {$\qjwm$};
\draw[usual] (-.5,-1) to[out=270,in=180] (0,-1.5) node[below,yshift=-1pt]{$k$} to[out=0,in=270] (.5,-1);
\node at (0,2.5) {$\phantom{a}$};
\node at (0,-2.5) {$\phantom{a}$};
\end{tikzpicture},
\quad
\end{gather}
\end{minipage}
\begin{minipage}{.33\textwidth}
\begin{gather}\label{eq:0trace}
\begin{tikzpicture}[anchorbase,scale=.25,tinynodes]
\draw[JW] (-1.5,-1) rectangle (1,1);
\node at (-.25,-.1) {$\qjwm$};
\draw[usual] (-1,1) to[out=90,in=0] (-1.5,1.5) to[out=180,in=90] 
(-2,1) to (-2,0) node[left,xshift=2pt]{$k$} to (-2,-1) 
to[out=270,in=180] (-1.5,-1.5) to[out=0,in=270] (-1,-1);
\node at (0,2.5) {$\phantom{a}$};
\node at (0,-2.5) {$\phantom{a}$};
\end{tikzpicture}
=(-1)^{k}
\tfrac{v}{v{-}k}\cdot
\begin{tikzpicture}[anchorbase,scale=.25,tinynodes]
\draw[JW] (-1.8,-1) rectangle (1.8,1);
\node at (0,-.1) {$\qjwm[{v{-}1{-}k}]$};
\node at (0,2.5) {$\phantom{a}$};
\node at (0,-2.5) {$\phantom{a}$};
\end{tikzpicture}.
\end{gather}
\end{minipage}

\noindent Here $1\leq w\leq v$ and the projector of thickness $w-1$ in
\eqref{eq:0absorb} can be at any place. Similarly, the cap or cup in
\eqref{eq:0kill} can be at any place and of any thickness.
\end{lemmaqed}

\subsection{Characteristic \texorpdfstring{$\ppar$}{p} notions}\label{subsec:quiver-pnotions} 

As already suggested by the recursion \eqref{eq:jw-recursion}, the 
JW projectors have rational coefficients with respect to $\sbas$ and 
typically cannot be defined in $\TL[\F]$. To formalize this,
consider the $\ppar$-adic valuation $\ord\colon\Q\to\Z\cup\{\infty\}$, 
defined for $n\in\Z$ as $\ord(n)=\max\{m\in\N[0]\mid\ppar^{m}|n\}$ 
(including $\ord(0)=\infty$) and for $c=\neatfrac{r}{s}\in\Q$ as $\ord(c):=\ord(r)-\ord(s)$.

\begin{definition}\label{definition:coeffs-2}
For a non-zero $\morstuff{F}=\sum_{\sbasel_{i}\in\sbas}\sbasc_{i}\sbasel_{i}\in\TL[\Q]$ we let 
$\ord(\morstuff{F}):=\min_{i}\{\ord(\sbasc_{i})\}$.
We call such a morphism $\ppar$-admissible if
$\ord(\morstuff{F})\geq 0$.
\end{definition}

To highlight morphisms that might not be $\ppar$-admissible, 
we use $\tilde{\phantom{.}}$ as \eg in \eqref{eq:jw-recursion}.
Note that $\morstuff{F}=\sum_{\sbasel_{i}\in\sbas}\sbasc_{i}\sbasel_{i}\in\TL[\Q]$ 
is $\ppar$-admissible if and only if every coefficient
$\sbasc_{i}$ can be presented as a reduced 
fraction $\neatfrac{r}{s}$ with $\ppar{\hspace{0pt}\not|\hspace{1pt}}s$. 
In this case, $\morstuff{F}$ represents 
an element $\spe{\morstuff{F}}$ of $\TL[\F]$, which is zero if and only if $\ord(\morstuff{F})>0$. 
If we write $\morstuff{F}=\morstuff{F}_{0}+\morstuff{F}_{>0}$ 
with $\ord(\morstuff{F}_{0})=0$ and $\ord(\morstuff{F}_{>0})>0$, then $\spe{\morstuff{F}}=\spe{\morstuff{F}_{0}}$.

\begin{example}\label{example:jw-well-defined}
We have $\ord(\qjw[v{-}1])=0$ for $v\leq\ppar$, which corresponds to the fact that 
the characteristic zero Weyl module $\wmod(v-1)=\tmod(v-1)$ stays simple when reduced modulo $\ppar$. 
However, for $v>\ppar$, one typically has $\ord(\qjw[v{-}1])<0$, 
and in such cases the projectors $\qjw[v{-}1]$ cannot be defined in $\TL[\F]$. 
\end{example}

However, there are alternative idempotents $\pqjw[v{-}1]\in\TL[\Q]$ satisfying 
$\ord(\pqjw[v{-}1])\geq 0$ and we will consider their specializations 
$\pjw[v{-}1]:=\spe{\pqjw[v{-}1]}\in\TL[\F]$. To this end, recall that 
we write $v=\pbase{a_{j},\dots,a_{0}}{\ppar}=\sum_{i}a_{i}p^{i}$ 
for the $\ppar$-adic expansion of $v\in\N$
with digits $0\leq a_{i}<\ppar$ and $a_{j}\neq 0$. (More generally, we also write
$\pbase{b_j,\dots,b_{0}}{\ppar}:=\sum_{i}b_{i}p^{i}$ for any $b_i\in\Z$.) 

\begin{definition}\label{definition:ancestry}
If $v=\pbase{a_{j},\dots,a_{0}}{\ppar}\in\N$ has only 
a single non-zero digit, then $v$ is called an eve. The set of eves is 
denoted by $\eve$. 
If $v\notin\eve$, then the mother $\mother$ of $v$ is 
obtained by setting the rightmost non-zero digit 
of $v$ to zero. 
We will also consider the set 
$\ancest:=\{\mother,\motherr{v}{2}:=\mother[\mother],\dots\}$ 
of (matrilineal) ancestors of $v$, 
whose size $\generation$ is called the generation of $v$.
\end{definition}

Note that $\ancest=\emptyset$ if and only if $v\in\eve$, and for 
$v\notin\eve$ we write $\eve(v)$ for its eve.

\begin{definition}\label{definition:support} 
For $v=\pbase{a_{j},\dots,a_{0}}{\ppar}$, the 
support $\supp\subset \N$ is the set of the $2^{\generation}$ integers of the form
$w=\pbase{a_{j},\pm a_{j-1},\dots,\pm a_{0}}{\ppar}$. The 
integers $v[i]=\pbase{a_{j},\dots,a_{i+1},-a_{i},a_{i-1},\dots,a_{0}}{\ppar}$ 
for $a_i\neq 0$ form the fundamental support $\fsupp\subset\supp$ of $v$.
\end{definition}

\begin{example}\label{example:ancestry}
Let $\ppar=3$. Then $v=23=\pbase{2,1,2}{3}$ has $\generation[23]=2$, 
and $\mother[23]=21=\pbase{2,1,0}{3}$ and $\motherr{23}{2}=\eve(23)=18=\pbase{2,0,0}{3}$. Hence, the ancestry chart of $23$ is
\begin{gather*}
\ancest[23]=
\begin{tikzpicture}[anchorbase,scale=.25,tinynodes]
\node at (10,0) {$18\in\eve$};
\node at (10,2.5) {$\generation[21]=1$};
\node at (10,5) {$\generation[23]=2$};
\draw[thin,directed=1] (-4,0) node[left,xshift=.2pt]{$\motherr{23}{2}$} to (-1,0);
\draw[thin,directed=1] (-4,4) node[left,xshift=.2pt]{$\mother[23]$} to (-1,4) to (.25,3.3);
\draw[ultra thick,mygray,densely dotted] (0,0) to (-3,2.5);
\draw[ultra thick,mygray,densely dotted] (0,0) to (-1,2.5);
\draw[ultra thick,spinach] (0,0) to (1,2.5);
\draw[ultra thick,mygray,densely dotted] (0,0) to (3,2.5);
\draw[ultra thick,mygray,densely dotted] (1,2.5) to (0,5);
\draw[ultra thick,spinach] (1,2.5) to (2,5);
\draw[ultra thick,mygray,densely dotted] (3,2.5) to (4,5);
\draw[ultra thick,mygray,densely dotted] (3,2.5) to (6,5);
\draw[spinach,fill=white] (0,.2) circle (.8cm);
\draw[mygray,fill=mygray!25] (-3,2.7) circle (.8cm);
\draw[mygray,fill=mygray!25] (-1,2.7) circle (.8cm);
\draw[spinach,fill=white] (1,2.7) circle (.8cm);
\draw[mygray,fill=mygray!25] (3,2.7) circle (.8cm);
\draw[mygray,fill=mygray!25] (0,5.2) circle (.8cm);
\draw[spinach,fill=white] (2,5.2) circle (.8cm);
\draw[mygray,fill=mygray!25] (4,5.2) circle (.8cm);
\draw[mygray,fill=mygray!25] (6,5.2) circle (.8cm);
\node[tomato] at (0,0) {$18$};
\node[mygray] at (-3,2.5) {$19$};
\node[mygray] at (-1,2.5) {$20$};
\node[spinach] at (1,2.5) {$21$};
\node[mygray] at (3,2.5) {$24$};
\node[mygray] at (0,5) {$22$};
\node[spinach] at (2,5) {$23$};
\node[mygray] at (4,5) {$25$};
\node[mygray] at (6,5) {$26$};
\end{tikzpicture}
\,.
\end{gather*}
Moreover, $\supp[23]=\{23=\pbase{2,1,2}{3},19=
\pbase{2,1,-2}{3},17=\pbase{2,-1,2}{3},13=\pbase{2,-1,-2}{3}\}$ and 
$\fsupp[23]=\{19,17\}$.
In pictures,
\begin{gather}\label{eq:funny-example}
\begin{tikzpicture}[anchorbase,scale=.25,tinynodes]
\draw[thick,spinach] (2,0) to [out=45,in=135] (10,0);
\draw[thick,spinach] (10,0) to [out=45,in=180] (12,1.15) to [out=0,in=135] (14,0);
\draw[thick,spinach] (14,0) to [out=45,in=135] (22,0);
\draw[thick,spinach] (22,0) to [out=45,in=180] (24,1.15) to [out=0,in=135] (26,0);
\draw[thick,spinach] (26,0) to [out=45,in=135] (34,0);
\draw[thick,spinach] (34,0) to [out=45,in=180] (36,1.15) to [out=0,in=135] (38,0);
\draw[thick,spinach] (38,0) to [out=45,in=135] (46,0);
\draw[thick,myorange,densely dashed] (26,0) to [out=45,in=135] (46,0);
\draw[thick,myorange,densely dashed] (2,0) to [out=45,in=135] (34,0);
\draw[thick,myorange,densely dashed] (10,0) to [out=45,in=135] (26,0);
\draw[thick,spinach] (34,0) to [out=45,in=135] (46,0);
\draw[thick,spinach] (10,0) to [out=45,in=135] (34,0);
\draw[thick,spinach] (14,0) to [out=45,in=135] (26,0);
\draw[thin,myorange,odirected=1] (36,-2.25) to (26,-2.25) to (26,-.75);
\draw[thick,myorange,odirected=1] (40,-2.75) node[below]{$\supp[23]$} to 
(40,-2.25) to (46,-2.25) to (46,-.75);
\draw[thick,myorange,odirected=1] (38,-2.25) to (38,-1);
\draw[thick,myorange,odirected=1] (40,-2.25) to (34,-2.25) to (34,-1);
\draw[thin] (34,-2.75) node[below]{$\fsupp[23]$} to (34,-2.25);
\draw[thin] (34,-2.25) to (38,-2.25);
\draw[thin,directed=1] (38,-2.25) to (38,-.75);
\draw[thin,directed=1] (34,-2.25) to (34,-.75);
\draw[spinach,fill=white] (2,.2) circle (.8cm);
\draw[mygray,fill=mygray!25] (4,.2) circle (.8cm);
\draw[mygray,fill=mygray!25] (6,.2) circle (.8cm);
\draw[mygray,fill=mygray!25] (8,.2) circle (.8cm);
\draw[spinach,fill=white] (10,.2) circle (.8cm);
\draw[mygray,fill=mygray!25] (12,.2) circle (.8cm);
\draw[spinach,fill=white] (14,.2) circle (.8cm);
\draw[mygray,fill=mygray!25] (16,.2) circle (.8cm);
\draw[mygray,fill=mygray!25] (18,.2) circle (.8cm);
\draw[mygray,fill=mygray!25] (20,.2) circle (.8cm);
\draw[spinach,fill=white] (22,.2) circle (.8cm);
\draw[mygray,fill=mygray!25] (24,.2) circle (.8cm);
\draw[spinach,fill=corn!25] (26,.2) circle (.8cm);
\draw[mygray,fill=mygray!25] (28,.2) circle (.8cm);
\draw[mygray,fill=mygray!25] (30,.2) circle (.8cm);
\draw[mygray,fill=mygray!25] (32,.2) circle (.8cm);
\draw[spinach,fill=corn!75] (34,.2) circle (.8cm);
\draw[mygray,fill=mygray!25] (36,.2) circle (.8cm);
\draw[spinach,fill=corn!75] (38,.2) circle (.8cm);
\draw[mygray,fill=mygray!25] (40,.2) circle (.8cm);
\draw[mygray,fill=mygray!25] (42,.2) circle (.8cm);
\draw[mygray,fill=mygray!25] (44,.2) circle (.8cm);
\draw[spinach,fill=corn!75] (46,.2) circle (.8cm);
\node[tomato] at (2,0) {$1$};
\node[tomato] at (4,0) {$2$};
\node[tomato] at (6,0) {$3$};
\node[mygray] at (8,0) {$4$};
\node[spinach] at (10,0) {$5$};
\node[tomato] at (12,0) {$6$};
\node[spinach] at (14,0) {$7$};
\node[mygray] at (16,0) {$8$};
\node[tomato] at (18,0) {$9$};
\node[mygray] at (20,0) {$10$};
\node[spinach] at (22,0) {$11$};
\node[mygray] at (24,0) {$12$};
\node[spinach] at (26,0) {$13$};
\node[mygray] at (28,0) {$14$};
\node[mygray] at (30,0) {$15$};
\node[mygray] at (32,0) {$16$};
\node[spinach] at (34,0) {$17$};
\node[tomato] at (36,0) {$18$};
\node[spinach] at (38,0) {$19$};
\node[mygray] at (40,0) {$20$};
\node[mygray] at (42,0) {$21$};
\node[mygray] at (44,0) {$22$};
\node[spinach] at (46,0) {$23$};
\end{tikzpicture}
,
\;
\begin{tikzpicture}[anchorbase,scale=.25,tinynodes]
\node at (0,7) {$\supp[23]$};
\draw[thin] (-3,1.35) to (3,1.35);
\draw[thin] (-3,4.05) to (3,4.05);
\draw[ultra thick,spinach] (0,0) to (-2.5,2.5);
\draw[ultra thick,spinach] (0,0) to (2.5,2.5);
\draw[ultra thick,spinach] (0,5) to (-2.5,2.5);
\draw[ultra thick,spinach] (0,5) to (2.5,2.5);
\draw[spinach,fill=corn!25] (0,.2) circle (.8cm);
\draw[spinach,fill=corn!75] (-2.5,2.7) circle (.8cm);
\draw[spinach,fill=corn!75] (2.5,2.7) circle (.8cm);
\draw[spinach,fill=corn!75] (0,5.2) circle (.8cm);
\node[spinach] at (0,0) {$13$};
\node[spinach] at (-2.5,2.5) {$17$};
\node[spinach] at (2.5,2.5) {$19$};
\node[spinach] at (0,5) {$23$};
\end{tikzpicture}
,
\end{gather}
where we have highlighted in yellow the support of $23$. The solid green arcs indicate successive 
inclusions in fundamental supports, and dashed orange arcs indicate successive inclusions in non-fundamental supports, all starting from $23$.
\end{example}

To account for \losp~ we need the following admissibility conditions.

\begin{definition}\label{definition:adm}
Let $S\subset\N[0]$ be a finite set. 
We consider partitions $S=\bigsqcup_{i}S_{i}$ of $S$ into subsets $S_{i}$ 
of consecutive integers, which we call \emph{stretches} 
(in the $\ppar$-adic expansion of $v$). For the purpose of this definition, we fix the \emph{coarsest} such partition. 

The set $S$ is called down-admissible for $v=\pbase{a_{j},\dots,a_{0}}{\ppar}$ if:
\begin{enumerate}[label=(\roman*)]

\setlength\itemsep{.15cm}

\item[(d1)] \label{definition:adm-i} $a_{\min(S_{i})}\neq 0$ for every $i$, and

\item[(d2)] \label{definition:adm-ii} if $s\in S$ and $a_{s+1}=0$, then $s+1\in S$.
\end{enumerate}
If $S\subset \N[0]$ is down-admissible for $v=\pbase{a_{j},\dots,a_{0}}{\ppar}$, then we define
\begin{gather*}
v[S]:=\pbase{a_{j},
\epsilon_{j-1}a_{j-1},\dots,\epsilon_{0}\,a_{0}}{\ppar},\quad
\epsilon_{k}
=
\begin{cases}
1 &\text{if }k\notin S,
\\
-1 &\text{if }k\in S.
\end{cases}
\end{gather*}

Conversely, $S$ is 
up-admissible for $v=\pbase{a_{j},\dots,a_{0}}{\ppar}$ if the following conditions are satisfied:
\begin{enumerate}[label=(\roman*)]

\setlength\itemsep{.15cm}

\item[(u1)] \label{definition:adm-iii} $a_{\min(S_{i})}\neq 0$ for every $i$, and

\item[(u2)] \label{definition:adm-iiii} if $s\in S$ and $a_{s+1}=\ppar-1$, then $s+1\in S$.

\end{enumerate}
If $S\subset\N[0]$ is up-admissible for $v=\pbase{a_{j},\dots,a_{0}}{\ppar}$, then we define
\begin{gather*}
v(S):=
\pbase{a_{r(S)}^{\prime},\dots,a_{0}^{\prime}}{\ppar},
\quad
a_{k}^{\prime}= 
\begin{cases}
a_{k} &\text{if } k\notin S, k-1\notin S,
\\
a_{k}+2 &\text{if } k\notin S, k-1\in S,
\\
-a_{k} &\text{if } k\in S,
\end{cases}
\end{gather*}
where we extend the digits of $v$ by $a_{h}=0$ for $h>j$ if necessary.  

If $S$ is up-admissible, then we denote by $\hull\subset\N[0]$ the
\emph{down-admissible hull} of $S$, the smallest down-admissible
set containing $S$, if it exists.  
\end{definition}

\begin{example}\label{example:ad-sets} Let $\ppar=7$. The set $S=\{5,4,3|0\}$
(here and in the following, we use vertical bars to highlight the coarsest
partition into stretches) is down-admissible but not up-admissible for
$v=\pbase{4,5,0,2,0,6,1}{7}$. On the other hand, $S^{\prime}=\{5,4,3|1,0\}$ is
up-admissible, but not down-admissible for $v$, and we get
\begin{gather*}
v[5,4,3|0]=\pbase{4,\underline{5,0,2},0,6,\underline{1\!\!\phantom{,}}}{7}=\pbase{4,-5,0,-2,0,6,-1}{7}
,
\\
v(5,4,3|1,0)=\pbase{4,\uwave{5,0,2\!\!\phantom{,}},0,\uwave{6,1}}{7}=\pbase{6,-5,0,-2,2,-6,-1}{7}.
\end{gather*}
Here we have underlined the stretches of digits in $\underline{S}$ 
and $\uwave{S^{\prime}}$. Furthermore, $\hull[S^{\prime}]=\{5,4,3,2,1,0\}$.
\end{example}

\begin{example}\label{example:ad-sets-0}
We think of the operations $v\mapsto v(S)$ and $v\mapsto v[S]$ as reflecting $v$ down and 
up along $S$, respectively. The admissibility restrictions ensure 
that the down-admissible sets $S$ are in bijection with the elements 
$v[S]\in\supp$ and that reflecting down and up are inverse operations as 
we will see in \fullref{lemma:stretch-lemma}. Explicitly, for $\ppar=3$ and $S=\{1,0\}$ one gets
\begin{gather*}
13(1,0)=\pbase{1,\uwave{1,1}}{3}=\pbase{3,-1,-1}{3}=23
,
\quad
23[1,0]=\pbase{2,\underline{1,2}}{3}=\pbase{2,-1,-2}{3}=13.
\end{gather*} 
See also \eqref{eq:funny-example}.
\end{example}

\begin{definition}\label{definition:adm-next}
For two non-empty sets $A,B\subset\N[0]$ we define
\begin{gather*}
\dist(A,B)=\min(|a-b|\mid a\in A,b\in B).
\end{gather*}
We say $A$ and $B$ are \emph{distant} if $\dist(A,B)>1$, \emph{adjacent} if $\dist(A,B)=1$,
or \emph{overlapping} if $\dist(A,B)=0$.
\end{definition} 

If $S$ and $S^{\prime}$ are down- or up-admissible for $v$ and $S\cap S^{\prime}=\emptyset$, 
then $S\cup S^{\prime}$ will also be down- or up-admissible, respectively. Conversely, 
if $S$ is down- or up-admissible for $v$ and $S^{\prime}\subset S$, then $S^{\prime}$ need not be down- or up-admissible for $v$. 

For down- or up-admissible sets $S$, a central object in the following will be
the \emph{finest partition} into down- or up-admissible subsets
$S=\bigsqcup_{k=0}^{r(S)}S_{k}$ (the number $r(S)+1$ is the size of this
partition), which we order by size of their elements $S_{k}>S_{k{-}1}$. Note
that the elements of $S_{k}$ are necessarily consecutive integers, and that this
partition is typically finer than the partition considered in
\fullref{definition:adm}. We call the $S_{k}$ \emph{minimal down-} or
\emph{up-admissible stretches} of $v$, respectively. It is easy to check that
\begin{gather*}
v[S]=v[S_{r(S)}]\cdots[S_{0}],
\quad
v(S)=v(S_{0})\cdots(S_{r(S)}),
\end{gather*}
for down- or up-admissible $S$, respectively.

\begin{example}\label{example:ad-sets-2} 
For the set $S=\{5,4,3|0\}$
(partitioned into stretches by the bar) and $v$ as in
\fullref{example:ad-sets} the finest down-admissible partition is
$S=\{5|4,3|0\}= S_{2}\cup S_{1}\cup S_{0}$ where
$v[S_{0}],v[S_{1}],v[S_{2}]\in\fsupp$. More generally, the down-admissible sets
$S$ with $v[S]\in\fsupp$ are exactly the minimal down-admissible stretches for
$v$.
\end{example}

If $S^{\prime}$ is also down- or up-admissible and distant from $S$, \ie 
$\dist(S,S^{\prime})>1$, then we have:
\begin{gather}\label{eq:reflectcommute} 
v[S][S^{\prime}]=v[S^{\prime}][S], 
\quad
v(S)(S^{\prime})=v(S^{\prime})(S), 
\quad 
v(S)[S^{\prime}]=v[S^{\prime}](S).
\end{gather}

If $S$ and $T$ are subsets of $\N[0]$, we write $T>S$ to indicate the
requirement that every element in $T$ be strictly greater than every element in
$S$. We have the following equivalences of admissibilities.

\begin{lemma}\label{lemma:new-admissible}
Consider stretches $S^{\prime}>S$ with $\dist(S,S^{\prime})=1$.
\begin{enumerate}[label=(\alph*)]

\setlength\itemsep{.15cm}

\item \label{lemma:new-admissible-a} $S$ is down-admissible 
for $v$ and $S^{\prime}$ is down-admissible for $v[S]$ if and only 
if $S^{\prime}$ is down-admissible for $v$ and $S$ is 
up-admissible for $v[S^{\prime}]$. In this case we have $v[S][S^{\prime}]=v[S^{\prime}](S)$.

\item \label{lemma:new-admissible-b} $S^{\prime}$ is up-admissible 
for $v$ and $S$ is up-admissible for $v(S^{\prime})$ if and only 
if $S$ is down-admissible for $v$ and $S^{\prime}$ is 
up-admissible for $v[S]$. In this case we have $v(S^{\prime})(S)=v[S](S^{\prime})$.
\end{enumerate} 
\end{lemma}

\begin{proof}
We prove (a). For this 
we write $v=\pbase{a_{j},\dots, a_{0}}{\ppar}$, 
$S=\{s,s+1,\dots,s^{\prime}-1\}$ and $S^{\prime}=\{s^{\prime},s^{\prime}+1,\dots,t-1\}$.

$S$ is down-admissible for $v$ if and only if $a_{s}\neq 0$ and $a_{s^{\prime}}\neq 0$, and we get
\begin{gather*}
v[S]=\pbase{a_{j},\dots, a_{t},a_{t{-}1},
\dots,a_{s^{\prime}{+}1},a_{s^{\prime}}-1,\ppar-a_{s^{\prime}-1}-1,\dots,\ppar-a_{s},a_{s{-}1},\dots,a_{0}}{\ppar}.
\end{gather*}
Now $S^{\prime}$ is down-admissible for $v[S]$ if and only if $a_{s^{\prime}}\neq 1$ and $a_{t}\neq 0$, and we get
\begin{gather*}
v[S][S^{\prime}]{=}\pbase{a_{j},\dots, a_{t}-1,\ppar-a_{t{-}1}-1,\dots,
\ppar-a_{s^{\prime}}+1,\ppar-a_{s^{\prime}-1}-1,\dots,\ppar-a_{s},a_{s{-}1},\dots,a_{0}}{\ppar}.
\end{gather*}
Conversely, $S^{\prime}$ is down-admissible for $v$ if and only if $a_{s^{\prime}}\neq 0$ and $a_{t}\neq 0$, and we get
\begin{gather*}
v[S^{\prime}]=\pbase{a_{j},\dots, a_{t}-1,\ppar-a_{t{-}1}-1,\dots,
\ppar-a_{s^{\prime}},a_{s^{\prime}-1},\dots,a_{s},a_{s{-}1},\dots,a_{0}}{\ppar}.
\end{gather*}
Now $S$ is up-admissible for $S^{\prime}$ if and only if $a_{s}\neq 0$ and $a_{s^{\prime}}\neq 1$. 
This shows the equivalence of admissibilities. Furthermore, by reflecting 
$v[S^{\prime}]$ up along $S$, it is easy to see $v[S^{\prime}](S)=v[S][S^{\prime}]$. The case of (b) is analogous. 
\end{proof} 

\begin{lemma}\label{lemma:stretch-lemma}
Let $v\in\N$ and $S\subset\N[0]$ finite.
\begin{enumerate}[label=(\alph*)]

\setlength\itemsep{.15cm}

\item \label{lemma:stretch-lemma-a} If $S$ is up-admissible for $v$, then $S$ is down-admissible for $w=v(S)$ and $v=w[S]$.

\item \label{lemma:stretch-lemma-b} If $S$ is down-admissible for $v$, then $S$ is up-admissible for $u=v[S]$ and $v=u(S)$.

\end{enumerate}
\end{lemma}

\begin{proof} 
Let $v=\pbase{a_{j},\dots,a_{0}}{\ppar}$. 
By \eqref{eq:reflectcommute} it suffices to consider 
the case where $S=\{s,s+1,\dots,s^{\prime}-1\}$ 
is a single stretch. For \fullref{lemma:stretch-lemma}.(a), 
suppose that $S$ is up-admissible for $v$, \ie 
$a_{s}\neq 0$ and $a_{s^{\prime}}\neq \ppar-1$. We get
\begin{gather*}
\begin{aligned}
w=v(S)
&=\pbase{\dots,a_{s^{\prime}+1},a_{s^{\prime}}+2,
-a_{s^{\prime}-1},\dots,-a_{s+1},-a_{s},a_{s-1},\dots,a_{0}}{\ppar}
\\
&=\pbase{\dots,a_{s^{\prime}+1},a_{s^{\prime}}+1,
\ppar-a_{s^{\prime}-1}-1,\dots,\ppar-a_{s+1}-1,\ppar-a_{s},a_{s-1},\dots,a_{0}}{\ppar}.
\end{aligned}
\end{gather*}
Since $a_{s^{\prime}}+1\neq 0$ and 
$\ppar-a_s\neq 0$, $S$ is down-admissible for $w$ and we have: 
\begin{gather*}
v(S)[S]=\pbase{\dots,a_{s^{\prime}+1},
a_{s^{\prime}}+1,a_{s^{\prime}-1}-\ppar+1,
\dots,a_{s+1}-\ppar+1,a_{s}-\ppar,a_{s-1},\dots,a_{0}}{\ppar}=v.
\end{gather*}
The proof of (b) is completely analogous.
\end{proof}

\subsection{Bookkeeping for caps and cups}\label{subsection:bookkeeping}

For this section, we fix $v=\pbase{a_{j},\dots,a_{0}}{\ppar}$.

\begin{definition}\label{definition:cup-cap-operators} 
For $0\leq i\leq j$ we consider 
$w=\pbase{a_{j},\dots,a_{i+1},-a_{i},0,\dots,0}{\ppar}-1$ 
and $x=\pbase{a_{i-1},\dots,a_{0}}{\ppar}$
to define (down and up) diagrams in $\TL[\Q]$ via
\begin{gather*}
\down{i}\idtl
:=
\idtl[{x{+}w}]\down{i}\idtl
:=
\begin{tikzpicture}[anchorbase,scale=.25,tinynodes]
\draw[usual] (-2,0) to (-2,3) node[above,yshift=-2pt]{$x$};
\draw[usual] (2.5,0) to[out=90,in=0] 
(.5,1.25)node[above,yshift=-1pt]{$a_{i}\ppar^{i}$} to[out=180,in=90] (-1.5,0);
\draw[usual] (3,0) to (3,3) node[above,yshift=-2pt]{$w$};
\node at (0,-.25) {$\phantom{.}$};
\node at (0,3.25) {$\phantom{.}$};
\end{tikzpicture}
,\quad
\idtl\up{i}
:= 
\idtl\up{i}\idtl[{x{+}w}]
:= 
\begin{tikzpicture}[anchorbase,scale=.25,tinynodes]	
\draw[usual] (-2,3) to (-2,0) node[below]{$x$};
\draw[usual] (2.5,3) to[out=270,in=0] 
(.5,1.75) node[below,yshift=-2pt]{$a_{i}\ppar^{i}$} to[out=180,in=270] (-1.5,3);
\draw[usual] (3,3) to (3,0) node[below]{$w$};
\node at (0,-.25) {$\phantom{.}$};
\node at (0,3.25) {$\phantom{.}$};
\end{tikzpicture}
.
\end{gather*}
This includes the case of $a_i=0$, for which we have
$\down{i}\idtl=\idtl\up{i}=\idtl$. Note that we use symbols such
as $\idtl$ to indicate the source or target of these morphisms.

Now, suppose that $S=\{s_{k}>\cdots>s_{1}>s_{0}\}$ 
and $S^{\prime}=\{s_{l}^{\prime}>\cdots>s_{1}^{\prime}>s_{0}^{\prime}\}$ are down-, 
respectively, up-admissible for $v$. Then we set
\begin{gather}\label{eq:idem-notation}
\begin{gathered}
\down{S}\idtl
:=
\idtl[{v[S]{-}1}]\down{S}
:=
\idtl[{v[S]{-}1}]\down{s_{0}}\cdots\down{s_{k}}\idtl
,\\
\up{S^{\prime}}\idtl
:=
\idtl[{v(S){-}1}]\up{S^{\prime}}
:=
\idtl[{v(S){-}1}]\up{s_{l}^{\prime}}\cdots\up{s_{0}^{\prime}}\idtl.
\end{gathered}
\end{gather}
\end{definition}

In \eqref{eq:idem-notation} and in the following we use the usual notation of 
idempotented algebras to drop some of the involved idempotents.
Further, the different orderings of the factors in 
$\down{S}$ and $\up{S^{\prime}}$ ensure that stretches of 
consecutive integers in $S$ and $S^{\prime}$ give rise to nested 
caps and cups, respectively.

\begin{lemma}\label{lemma:isotopies}
For $S^{\prime}>S$ with $\dist(S^{\prime},S)=1$ the following hold.
\begin{enumerate}[label=(\alph*)]

\setlength\itemsep{.15cm}

\item \label{lemma:isotopies-a} $S^{\prime}$ is down-admissible 
for $v$ and $S$ is down-admissible for $v[S^{\prime}]$ if and only 
if $S$ and $S\cup S^{\prime}$ are down-admissible for $v$. In this case we have $\down{S}\down{S^{\prime}}\idtl
=\down{S\cup S^{\prime}}\idtl$.

\item \label{lemma:isotopies-b} $S$ is up-admissible 
for $v$ and $S^{\prime}$ is up-admissible for $v(S)$ if and 
only if $S^{\prime}$ and $S^{\prime}\cup S$ are up-admissible 
for $v$. In this case we have $\up{S^{\prime}}\up{S}\idtl=\up{S^{\prime}\cup S}\idtl$.

\item \label{lemma:isotopies-c} If $S^{\prime}$ is 
up-admissible for $v$ and $S$ is down-admissible for $v(S^{\prime})$, then $S^{\prime}\cup S$ 
is up-admissible for $v$. In this case we have $\down{S}\up{S^{\prime}}\idtl=\up{S^{\prime}\cup S}\idtl$.

\item \label{lemma:isotopies-d} If $S$ is 
up-admissible for $v$ and $S^{\prime}$ is down-admissible for $v(S)$, then
$S\cup S^{\prime}$ is down-admissible 
for $v$. In this case we have $\down{S^{\prime}}\up{S}\idtl=\down{S\cup S^{\prime}}\idtl$.
\end{enumerate} 
\end{lemma}

\begin{proof} 
The claims about admissibility are not hard to prove and follow, {\muta}, 
as in the proof of \fullref{lemma:new-admissible} given above. 
Finally, the equalities as \eg $\down{S}\down{S^{\prime}}\idtl=\down{S\cup S^{\prime}}\idtl$ are isotopies.
\end{proof}

\begin{definition}\label{definition:jw-cupscaps}
Using the same notation as in \fullref{definition:cup-cap-operators}, 
we define diagrams in $\TL[\Q]$
\begin{gather*}
\downo{i}\idtl
:=
\idtl[{x{+}w}]\downo{i}\idtl
:=
\begin{tikzpicture}[anchorbase,scale=.25,tinynodes]
\draw[JW] (0,0) rectangle (4,1);
\node at (0,1.85) {$a_{i}\ppar^{i}$};	
\draw[usual] (-2,0) to (-2,3) node[above,yshift=-2pt]{$x$};
\draw[usual] (1.5,1) to(1.5,2) to[out=90,in=0] (0,3) to[out=180,in=90]  (-1.5,2) to (-1.5,0);
\draw[usual] (2,1) to (2,3)node[above,yshift=-2pt]{$w$};
\node at (0,-.25) {$\phantom{.}$};
\node at (0,3.25) {$\phantom{.}$};
\end{tikzpicture}
,\quad
\idtl\upo{i}
:=
\idtl\upo{i}\idtl[{x{+}w}]
:=
\begin{tikzpicture}[anchorbase,scale=.25,tinynodes]
\draw[JW] (0,2) rectangle (4,3);
\node at (0,.9) {$a_{i}\ppar^{i}$};
\draw[usual] (-2,3) to (-2,0) node[below,yshift=-2pt]{$x$};
\draw[usual] (1.5,2) to (1.5,1) to[out=270,in=0] (0,0) to[out=180,in=270] (-1.5,1) to (-1.5,3);
\draw[usual] (2,2) to (2,0)node[below,yshift=-2pt]{$w$};
\node at (0,-.25) {$\phantom{.}$};
\node at (0,3.25) {$\phantom{.}$};
\end{tikzpicture}
.
\end{gather*}
The boxes represent JW projectors of the size implicit in the diagram, namely $w+a_i\ppar^i =v-a_i\ppar^i-x$.
\end{definition}

\begin{definition}\label{definition:jw-cupscaps2}
Suppose that $S=\{s_{k}>\cdots>s_{1}>s_{0}\}$ is down-admissible for $v$ 
and $S^{\prime}=\{s_{l}^{\prime}>\cdots>s_{1}^{\prime}>s_{0}^{\prime}\}$ is 
up-admissible for $v$. Then we define \emph{trapezes} and \emph{standard loops}
\begin{align*}
\TRD{$\qjwm[S]$}&:=
\downo{S}\idtl:=\qjw[{v[S]{-}1}]\downo{s_{0}}\cdots\downo{s_{k}}\idtl 
,\\
\TRU{$\qjwm[S^{\prime}]$}&:=
\upo{S^{\prime}}\idtl:=\idtl[{v(S^{\prime})-1}]\upo{s_{l}^{\prime}}\cdots\upo{s_{0}^{\prime}}\qjw[{v{-}1}],
\\
\TR{$\qjwm[S]$}&:=\trap{S}{v{-}1}:=\upo{S}\qjw[{v[S]{-}1}]\downo{S}.
\end{align*}
\end{definition}
Note that the diagrams defined in \fullref{definition:jw-cupscaps2} are not left-right symmetric. 

\begin{example}\label{example:trapezes}
For $v=\pbase{a,b,c}{\ppar}$ we have: 
\begin{gather*}
\TRD{$\qjwm[\emptyset]$}
=
\begin{tikzpicture}[anchorbase,scale=.25,tinynodes]	
\draw[JW] (-1.5,1) rectangle (1.5,2);
\draw[JW] (-1.5,-1) rectangle (1.5,-2);
\draw[usual] (0,-1) to (0,1);
\node at (0,-2.3) {$\phantom{a}$};
\end{tikzpicture}
,\quad
\TRD{$\qjwm[\{0\}]$}
=
\begin{tikzpicture}[anchorbase,scale=.25,tinynodes]
\draw[JW] (-1.5,1) rectangle (1.5,2);
\draw[JW] (-1.5,-1) rectangle (1.5,-2);
\draw[usual] (0,-1) to (0,1);
\draw[usual] (-1,-1) to[out=90,in=0] (-1.5,-.5) 
to[out=180,in=90] (-2,-1) to (-2,-2) node[below,yshift=0pt]{$c$};
\node at (0,-2.1) {$\phantom{a}$};
\end{tikzpicture}
,\quad
\TRD{$\qjwm[\{1\}]$}
=
\begin{tikzpicture}[anchorbase,scale=.25,tinynodes]
\draw[JW] (-1.5,1) rectangle (1.5,2);
\draw[JW] (-1.5,-1) rectangle (1.5,-2);
\draw[usual] (0,-1) to (0,1);
\draw[usual] (-1,-1) to[out=90,in=0] (-1.5,-.5) 
to[out=180,in=90] (-2,-1) to (-2,-2) node[below,yshift=0pt]{$b\ppar$};
\draw[usual] (-1,1) to[out=270,in=90] (-3,-.5) to (-3,-2) node[below,yshift=0pt]{$c$};
\node at (0,-2.1) {$\phantom{a}$};
\end{tikzpicture}
,\quad
\TRD{\,$\qjwm[\{1,0\}]$}
=
\begin{tikzpicture}[anchorbase,scale=.25,tinynodes]
\draw[JW] (-1.5,1) rectangle (1.5,2);
\draw[JW] (-1.5,-1) rectangle (1.5,-2);
\draw[usual] (1,-1) to (1,1);
\draw[usual] (-1,-1) to[out=90,in=0] (-1.5,-.5) 
to[out=180,in=90] (-2,-1) to (-2,-2) node[below,yshift=0pt]{$b\ppar$};
\draw[usual] (0,-1) to[out=90,in=0] (-1.5,0) 
to[out=180,in=90] (-3,-1) to (-3,-2) node[below,yshift=0pt]{$c$};
\node at (0,-2.1) {$\phantom{a}$};
\end{tikzpicture}
.
\end{gather*}
\end{example}

We record that $\td{\downo{S}\idtl}=v[S]-1$, 
$\td{\upo{S}\idtl}=v-1$, and $\td{\trap{S}{v{-}1}}=v[S]-1$.

\subsection{The \texorpdfstring{$\ppar$}{p}-Jones--Wenzl projectors}

For $v\in\N$ and $s\in \N[0]$ let
$\fancest{v}{s}$ denote the youngest
ancestor of $v$ whose $s$th 
digit is zero. (By convention, $\fancest{v}{{-}1}=v$.) 
For each down-admissible $S$ for $v$ we let 
\begin{gather}\label{eq:the-scalar}
\lambda_{v,S}
:={\textstyle\prod_{s\in S}}\,
(-1)^{a_{s}\ppar^s}\tfrac{\fancest{v}{s{-}1}[S]}{\fancest{v}{s}[S]}
\in\Q.
\end{gather}
Note that $\ord(\lambda_{v,S})=-|S|$.

\begin{example}\label{lemma:the-scalar}
Let $v=\pbase{1,2,6,4,0,6,6}{7}$ and $S=\{5|3,2,1,0\}$. 
Then we have $\sum_{s\in S} a_{s}=18$, so the overall sign is positive. 
The relevant reflected ancestors in the telescoping product \eqref{eq:the-scalar} are
$\fancest{v}{{-}1}[S]=\pbase{1,-2,6,-4,0,-6,-6}{7}$, $\fancest{v}{3}[S]=\pbase{1,-2,6,0,0,0,0}{7}$, 
$\fancest{v}{4}[S]=\pbase{1,-2,0,0,0,0,0}{7}$, and $\fancest{v}{5}[S]=\pbase{1,0,0,0,0,0,0}{7}$. So we get
\begin{gather*}
\lambda_{v,S}=
\tfrac{\pbase{1,-2,0,0,0,0,0}{7}\pbase{1,-2,6,-4,0,-6,-6}{7}}{\pbase{1,0,0,0,0,0,0}{7}\pbase{1,-2,6,0,0,0,0}{7}}
=
\tfrac{485105}{689087}
,
\quad
\ord[7](\lambda_{v,S})=-5.
\end{gather*}
\end{example}

The following is immediate from \eqref{eq:the-scalar}.

\begin{lemmaqed}\label{lemma:lambda}
If $S^{\prime}>S$ are down-admissible for $v$, then
$\lambda_{v,S\cup S^{\prime}}=\lambda_{v[S^{\prime}],S}\lambda_{v,S^{\prime}}$.
\end{lemmaqed}

As we will see below, the following definition is a reformulation of \cite[Section 2.3]{BuLiSe-tl-char-p}.

\begin{definition}\label{definition:p-jw-q}
For $v-1\in\N[0]$ the \emph{rational $\ppar$JW projector} $\pqjw[v{-}1]\in\Hom_{\TL[\Q]}(v{-}1,v{-}1)$ is defined to be
\begin{gather}\label{eq:pjwdef}
\begin{tikzpicture}[anchorbase,scale=.25,tinynodes]
\draw[pQJW] (-1.5,0) rectangle (1.5,2);
\node at (0,.9) {$\pqjwm[v{-}1]$};
\end{tikzpicture}
:=
\pqjw[v{-}1]
:=
{\textstyle\sum_{v[S]\in\supp[v]}}\,
\lambda_{v,S}\,\trap{S}{v{-}1}
=
{\textstyle\sum_{v[S]\in\supp[v]}}\,
\lambda_{v,S}\cdot\TR{$\qjwm[S]$}.
\end{gather}  
\end{definition}

\begin{example}\label{example:the-half-projectors} 
For $v=\pbase{a,b,c}{\ppar}$ we have
\begin{gather*}
\begin{tikzpicture}[anchorbase,scale=.25,tinynodes]
\draw[usual,white] (0,0) to (0,2) node[above,yshift=-2pt]{$2$};
\draw[usual,white] (0,-6) to (0,-8) node[below]{$2$};
\draw[pQJW] (-1.5,-4) rectangle (1.5,-2);
\node at (0,-3.1) {$\pqjwm[v{-}1]$};
\end{tikzpicture}
=
\begin{tikzpicture}[anchorbase,scale=.25,tinynodes]
\draw[JW] (-1.5,-1) rectangle (1.5,-2);
\draw[JW] (-1.5,-5) rectangle (1.5,-4);
\draw[JW] (-1.5,-8) rectangle (1.5,-7);
\draw[usual] (-1,-2) to (-1,-4);
\draw[usual] (0,-2) to (0,-4);
\draw[usual] (1,-2) to (1,-4);
\draw[usual] (-1,-5) to (-1,-7)(0,-5) to (0,-7)(1,-5) to (1,-7);
\end{tikzpicture}
+(-1)^{c}\tfrac{\pbase{a,b,{-}c}{\ppar}}{\pbase{a,b,0}{\ppar}}\cdot\hspace*{-.1cm}
\begin{tikzpicture}[anchorbase,scale=.25,tinynodes]
\draw[JW] (-1.5,-1) rectangle (1.5,-2);
\draw[JW] (-1.5,-5) rectangle (1.5,-4);
\draw[JW] (-1.5,-8) rectangle (1.5,-7);
\draw[usual] (0,-2) to (0,-4);
\draw[usual] (1,-2) to (1,-4);
\draw[usual] (-1,-2) to[out=270,in=0] (-1.5,-2.5) 
to[out=180,in=270] (-2,-2) to (-2,-1) node[above,yshift=-2pt]{$c$};
\draw[usual] (0,-7) to (0,-5);
\draw[usual] (1,-7) to (1,-5);
\draw[usual] (-1,-7) to[out=90,in=0] (-1.5,-6.5) 
to[out=180,in=90] (-2,-7) to (-2,-8) node[below]{$c$};
\end{tikzpicture}
+(-1)^{b\ppar}\tfrac{\pbase{a,{-}b,0}{\ppar}}{\pbase{a,0,0}{\ppar}}\cdot\hspace*{-.1cm}
\begin{tikzpicture}[anchorbase,scale=.25,tinynodes]
\draw[JW] (-1.5,-1) rectangle (1.5,-2);
\draw[JW] (-1.5,-5) rectangle (1.5,-4);
\draw[JW] (-1.5,-8) rectangle (1.5,-7);
\draw[usual] (0,-2) to (0,-4);
\draw[usual] (1,-2) to (1,-4);
\draw[usual] (-1,-2) to[out=270,in=0] (-1.5,-2.5) 
to[out=180,in=270] (-2,-2) to (-2,-1) node[above,yshift=-2pt]{$b\ppar$};
\draw[usual] (-1,-4) to[out=90,in=270] (-3,-2) to (-3,-1) node[above,yshift=-2pt]{$c$};
\draw[usual] (0,-7) to (0,-5);			
\draw[usual] (1,-7) to (1,-5);
\draw[usual] (-1,-7) to[out=90,in=0] (-1.5,-6.5) 
to[out=180,in=90] (-2,-7) to (-2,-8) node[below]{$b\ppar$};
\draw[usual] (-1,-5) to[out=270,in=90] (-3,-7) 
to (-3,-8) node[below]{$c$};
\end{tikzpicture}
+(-1)^{b\ppar+c}\tfrac{\pbase{a,{-}b,{-}c}{\ppar}}{\pbase{a,0,0}{\ppar}}\cdot\hspace*{-.1cm}
\begin{tikzpicture}[anchorbase,scale=.25,tinynodes]	
\draw[JW] (-1.5,-1) rectangle (1.5,-2);
\draw[JW] (-1.5,-5) rectangle (1.5,-4);
\draw[JW] (-1.5,-8) rectangle (1.5,-7);
\draw[usual] (1,-2) to (1,-4);
\draw[usual] (-1,-2) to[out=270,in=0] (-1.5,-2.5) 
to[out=180,in=270] (-2,-2) to (-2,-1) node[above,yshift=-2pt]{$b\ppar$};
\draw[usual] (0,-2) to[out=270,in=0] (-1.5,-3.5) 
to[out=180,in=270] (-3,-2) to (-3,-1) node[above,yshift=-2pt]{$c$};
\draw[usual] (1,-7) to (1,-5);
\draw[usual] (-1,-7) to[out=90,in=0] (-1.5,-6.5) 
to[out=180,in=90] (-2,-7) to (-2,-8) node[below]{$b\ppar$};
\draw[usual] (0,-7) to[out=90,in=0] (-1.5,-5.5) 
to[out=180,in=90] (-3,-7) to (-3,-8) node[below]{$c$};
\end{tikzpicture}
.
\end{gather*}
\end{example}

\begin{lemma}\label{lemma:pjw-well-defined}
The elements $\pqjw[v{-}1]$ agree with the 
ones defined in \cite[Section 2.3]{BuLiSe-tl-char-p}.
(Note however that we have mirrored their definition.)
\end{lemma}

\begin{proof}
Careful inspection of the recursive definition in \cite[Section 2.3]{BuLiSe-tl-char-p}. 
More precisely, in our notation their recursion works as follows. 
If $v\in\eve$, then $\pqjw[v{-}1]=\qjw[v{-}1]$. Otherwise,
\begin{gather}\label{eq:recursion-formula}
\begin{tikzpicture}[anchorbase,scale=.25,tinynodes]
\draw[pQJW] (-1.5,0) rectangle (1.5,2);
\node at (0,.9) {$\qjwm[v{-}1]$};
\end{tikzpicture}
=
{\textstyle\sum_{\mother[v][S]\in\supp[{\mother[v]}]}}\,
\lambda_{\mother[v],S}\,
\left(
\begin{tikzpicture}[anchorbase,tinynodes]
\tru{1.65}{.3}{$\qjwm[S]$}{0}{.2}
\draw[JW] (-.3,-.2) rectangle (1.65,.2);
\trd{1.65}{.3}{$\qjwm[S]$}{0}{-.2}
\node at (.675,-.05) {$\qjwm[{v[S]{-}1}]$};
\draw[usual] (-.15,.2) to (-.15,.5) node[above,yshift=-2pt]{$a_{s}\ppar^{s}$};
\draw[usual] (-.15,-.2) to (-.15,-.5) node[below,yshift=-1pt]{$a_{s}\ppar^{s}$};
\end{tikzpicture}
+(-1)^{a_{s}\ppar^s}\tfrac{v[S][s]}{\mother{[S]}}\cdot
\begin{tikzpicture}[anchorbase,tinynodes]
\tru{1.65}{.3}{$\qjwm[S]$}{0}{.2}
\draw[JW] (.45,-.2) rectangle (1.65,.2);
\trd{1.65}{.3}{$\qjwm[S]$}{0}{-.2}
\node at (1.05,-.05) {$\qjwm[{v[S][s]{-}1}]$};
\draw[usual] (.3,.2) to[out=270,in=0] (.075,.075) 
to[out=180,in=270] (-.15,.2) to (-.15,.5) node[above,yshift=-2pt]{$a_{s}\ppar^{s}$};
\draw[usual] (.3,-.2) to[out=90,in=0] (.075,-.075) 
to[out=180,in=90] (-.15,-.2) to (-.15,-.5) node[below,yshift=-1pt]{$a_{s}\ppar^{s}$};
\end{tikzpicture}
\right),
\end{gather}
where $a_{s}$ is the first non-zero digit of $v$.
\end{proof}

By \fullref{lemma:pjw-well-defined}, we can 
refer to results of \cite{BuLiSe-tl-char-p} without further notice.

\begin{propositionqed}(\cite[Theorem 2.6]{BuLiSe-tl-char-p}.)\label{proposition:pjw-well-defined}
For any $v\in\N$ we have $\ord(\pqjw[v{-}1])\geq 0$.
\end{propositionqed}

\begin{definition}\label{definition:p-jw}
We define the \emph{$\ppar$JW projectors} $\pjw[v{-}1]:=\spe{\pqjw[v{-}1]}\in\End_{\TL[\F]}(v{-}1)$.
\end{definition}

In illustrations we distinguish the three types of JW projectors via 
\begin{gather*}
\qjw[v{-}1]=
\begin{tikzpicture}[anchorbase,scale=.25,tinynodes]
\draw[JW] (-1.5,0) rectangle (1.5,2);
\node at (0,.9) {$\qjwm[v{-}1]$};
\end{tikzpicture}
,\quad
\pqjw[v{-}1]=
\begin{tikzpicture}[anchorbase,scale=.25,tinynodes]
\draw[pQJW] (-1.5,0) rectangle (1.5,2);
\node at (0,.9) {$\pqjwm[v{-}1]$};
\end{tikzpicture}
,\quad
\pjw[v{-}1]=
\begin{tikzpicture}[anchorbase,scale=.25,tinynodes]
\draw[pJW] (-1.5,0) rectangle (1.5,2);
\node at (0,.9) {$\pjwm[v{-}1]$};
\end{tikzpicture}
,
\end{gather*}
called JW, rational $\ppar$JW and $\ppar$JW projectors, respectively.

\begin{example}\label{example:some-jw-example}
Note that these projectors behave quite differently, \eg for the projectors 
as in \fullref{example:the-half-projectors} we have
\begin{gather*}
\begin{tikzpicture}[anchorbase,scale=.25,tinynodes]
\draw[JW] (-1,-1) rectangle (2,1);
\node at (.5,-.1) {$\qjwm[22]$};
\draw[usual] (-.5,1) to[out=90,in=180] (0,1.5) node[above,yshift=-2pt]{$5$} to[out=0,in=90] (.5,1);
\draw[usual] (1.5,1) to (1.5,2);
\node at (0,2.5) {$\phantom{a}$};
\node at (0,-2.5) {$\phantom{a}$};
\end{tikzpicture}
=0,
\quad
\begin{tikzpicture}[anchorbase,scale=.25,tinynodes]
\draw[pJW] (-1,-1) rectangle (2,1);
\node at (.5,-.1) {$\pjwm[22]$};
\draw[usual] (-.5,1) to[out=90,in=180] (0,1.5) node[above,yshift=-2pt]{$5$} to[out=0,in=90] (.5,1);
\draw[usual] (1.5,1) to (1.5,2);
\node at (0,2.5) {$\phantom{a}$};
\node at (0,-2.5) {$\phantom{a}$};
\end{tikzpicture}
=
\begin{tikzpicture}[anchorbase,scale=.25,tinynodes]
\draw[pJW] (-1,-1) rectangle (2,1);
\node at (.5,-.1) {$\pjwm[17]$};
\draw[usual] (-.5,1) to[out=90,in=0] (-1,1.5) node[above,yshift=-2pt]{$5$} to[out=180,in=90] (-1.5,1) to (-1.5,-1);
\draw[usual] (1.5,1) to (1.5,2);
\node at (0,2.5) {$\phantom{a}$};
\node at (0,-2.5) {$\phantom{a}$};
\end{tikzpicture}
.
\end{gather*}
\end{example}

\begin{proposition}\label{proposition:TLtilt} 
We have a pivotal, $\K$-linear functor 
\begin{gather*}
\tlfunctor\colon
\TL[\K]\to\tilt,
\quad
\tlfunctor(v-1)=\tmod(1)^{\otimes(v{-}1)},
\end{gather*}
which sends the idempotent $\pjw[v{-}1]$ to the projection 
$\tmod(1)^{\otimes(v{-}1)}\to\tmod(v-1)\to\tmod(1)^{\otimes(v{-}1)}$. This functor induces an equivalence of 
pivotal $\K$-linear categories upon additive Karoubi completion.  
\end{proposition}

\begin{proof}
By \fullref{proposition:pjw-well-defined} and the construction of $\TL[\K]$, 
the only non-trivial statement is the fully-faithfulness of $\tlfunctor$.
This is known; 
however, for completeness, let us give a short 
(but not new, \cf \cite[Theorem 2.58]{El-ladders-clasps} 
or \cite[Proposition 2.3]{AnStTu-semisimple-tilting}) argument for this. 
First, the same statement over $\C$ is a classical result and dates back to 
work of Rumer--Teller--Weyl. This implies that hom spaces on both sides have the same dimension over $\C$. 
The point is now the \emph{flatness} of both sides. Precisely, the standard basis $\sbas$ works for $\TL[\Z]$, 
showing that the dimensions of hom spaces in $\TL[\K]$ are independent 
of the characteristic. The same is true in the image of $\tlfunctor$: 
The module $\tmod(1)$ is a tilting module regardless of the characteristic, and the same 
thus holds for $\tmod(1)^{\otimes(v{-}1)}$. This implies that the hom spaces in 
$\tlfunctor(\TL[\K])$ are also independent of the characteristic. Finally, one can check that $\tlfunctor(\sbas)$ remains linear independent, and
the claim follows since all involved dimensions are finite and the same on both sides.
\end{proof}

\section{The quiver algebra}\label{sec:presentation}

\subsection{Generators and relations}\label{subsec:main-TL}

In order to prove \fullref{theorem:main}
we have to give a 
presentation of the algebra
\begin{gather}\label{eq:ealg}
\ealg:={\textstyle\bigoplus_{v,w\in\N}}\,\Hom_{\TL[\F]}
\big(
\pjw[v{-}1],\pjw[w{-}1]
\big)
\end{gather} 
by generators and relations. To this end, we first introduce notation for certain elements.

\begin{definition}\label{definition:updown}
Let $S$ and $S^{\prime}$ be down- and up-admissible for $v$, respectively. Then we define
\begin{gather}
\begin{aligned}
\Down{S}\pjw[v{-}1]
&:=\pjw[{v[S]{-}1}]\down{S}\pjw[v{-}1]
,
\\
\Up{S^{\prime}}\pjw[v{-}1]
&:=\pjw[v(S^{\prime}){-}1]\up{S^{\prime}}\pjw[v{-}1]
,
\\
\loopdown{S}{v{-}1}\pjw[v{-}1]
&:=\pjw[v{-}1]\up{S}\down{S}\pjw[v{-}1].
\end{aligned}
\end{gather}
We call the latter the \emph{$\ppar$loop} on $v-1$ down through $v[S]-1$.
\end{definition}

We will consider the morphisms $\Down{S}\pjw[v{-}1]$ and 
$\Up{S}\pjw[v{-}1]$ as generators for $\ealg$, but restrict to the 
cases when $S$ and $S^{\prime}$ are minimal admissible stretches of 
consecutive integers. Then these morphisms can be pictured as
\begin{gather*}
\Down{S}\pjw[v{-}1]=
\begin{tikzpicture}[anchorbase,scale=.25,tinynodes]
\draw[pJW] (2,.5) rectangle (-2,-.5);
\draw[pJW] (2,1.5) rectangle (-2,2.5);
\node at (0,-.2) {$\pjwm[v{-}1]$};
\draw[usual] (-.5,.5) to[out=90,in=180] (0,1) to[out=0,in=90] (.5,.5);
\draw[usual] (-1,.5) to (-1,1.5);
\draw[usual] (1,.5) to (1,1.5);
\node at (1,2.5) {$\phantom{a}$};
\node at (1,-.5) {$\phantom{a}$};
\end{tikzpicture}
,\quad
\Up{S^{\prime}}\pjw[v{-}1]=
\begin{tikzpicture}[anchorbase,scale=.25,tinynodes]
\draw[pJW] (2,.5) rectangle (-2,-.5);
\draw[pJW] (2,1.5) rectangle (-2,2.5);
\node at (0,-.2) {$\pjwm[v{-}1]$};
\draw[usual] (-.5,1.5) to[out=270,in=180] (0,1) to[out=0,in=270] (.5,1.5);
\draw[usual] (-1,.5) to (-1,1.5);
\draw[usual] (1,.5) to (1,1.5);
\node at (1,2.5) {$\phantom{a}$};
\node at (1,-.5) {$\phantom{a}$};
\end{tikzpicture}
,\quad
\loopdown{S}{v{-}1}\pjw[v{-}1]=
\begin{tikzpicture}[anchorbase,scale=.25,tinynodes]
\draw[pJW] (2,-.5) rectangle (-2,.5);
\draw[pJW] (2,1.5) rectangle (-2,2.5);
\draw[pJW] (2,3.5) rectangle (-2,4.5);
\node at (0,-.2) {$\pjwm[v{-}1]$};
\node at (0,3.8) {$\pjwm[v{-}1]$};
\draw[usual] (-.5,.5) to[out=90,in=180] (0,1) to[out=0,in=90] (.5,.5);
\draw[usual] (-.5,3.5) to[out=270,in=180] (0,3) to[out=0,in=270] (.5,3.5);
\draw[usual] (-1,.5) to (-1,1.5);
\draw[usual] (1,.5) to (1,1.5);
\draw[usual] (-1,2.5) to (-1,3.5);
\draw[usual] (1,2.5) to (1,3.5);
\node at (1,2.5) {$\phantom{a}$};
\node at (1,-.5) {$\phantom{a}$};
\end{tikzpicture}.
\end{gather*}

We define two functions $\funcf,\funcg\colon\F\to\F$ 
(where we again see \losp) via
\begin{gather*}
\funcf(a)=
\begin{cases} 
(-1)^{a}\tfrac{2}{a}
&\text{if } 1\leq a\leq\ppar-2,
\\
0 
&\text{if }a=0\text{ or }a=\ppar-1,
\end{cases}
\quad
\funcg(a) =
\begin{cases}
-(\tfrac{a+1}{a})
&\text{if } 1\leq a\leq\ppar-1,
\\
-2 &\text{if }a=0.
\end{cases}
\end{gather*}
Note that $\funcf(\ppar-1)=\funcg(\ppar-1)=0$ and $\funcg(a)=\funcg(\ppar-a-1)^{-1}$ for $a\neq 0,\ppar-1$. 
Then for each finite $S\subset\N[0]$ we define scaling 
operators $\funcF,\funcG,\funcH\in\ealg$ on $v=\pbase{a_{j},\dots,a_{0}}{\ppar}$ as
\begin{gather*}
\funcF\pjw[v{-}1] 
=
\funcf(a_{\max(S)+1})\pjw[v{-}1], 
\quad
\funcG\pjw[v{-}1] 
= 
\funcg(a_{\max(S)+1})\pjw[v{-}1],
\\
\funcH\pjw[v{-}1] 
=
\funcg(a_{\max(S)+1}-1)\pjw[v{-}1].
\end{gather*}
These are not considered as generators of $\ealg$, but as mere 
bookkeeping devices for the appearing scalars.

\begin{theorem}(Generators and relations.)\label{theorem:main-tl-section}
The algebra $\ealg$ is generated by $\pjw[v{-}1]$ for $v\in\N$, 
and elements $\Down{S}\pjw[v{-}1]$ and $\Up{S^{\prime}}\pjw[v{-}1]$, 
where $S$ and $S^{\prime}$ denote minimal down- and up-admissible stretches for $v$, respectively.
These generators are subject to the following complete set of relations.
\begin{enumerate}[label=(\arabic*)]

\setlength\itemsep{0.15cm}

\item \label{theorem:main-tl-section-1} \emph{Idempotents.}
\begin{gather*}
\pjw[v{-}1]\pjw[w{-}1]=\delta_{v,w}\pjw[v{-}1],
\quad 
\pjw[{v[S]{-}1}]\Down{S}\pjw[v{-}1]
=
\Down{S}\pjw[v{-}1],
\quad
\pjw[v(S^{\prime}){-}1]\Up{S^{\prime}}\pjw[v{-}1]
=
\Up{S^{\prime}}\pjw[v{-}1].
\end{gather*}

\item \label{theorem:main-tl-section-2} \emph{Containment.}
If $S^{\prime} \subset S$, then we have
\begin{gather*}
\Down{S^{\prime}}\Down{S}\pjw[v{-}1]=0,
\quad
\Up{S}\Up{S^{\prime}}\pjw[v{-}1]=0.
\end{gather*}

\item \label{theorem:main-tl-section-3} \emph{Far-commutativity.}
If $\dist(S,S^{\prime})>1$, then 
\begin{gather*}
\Down{S}\Down{S^{\prime}}\pjw[v{-}1]=\Down{S^{\prime}}\Down{S}\pjw[v{-}1],
\quad
\Down{S}\Up{S^{\prime}}\pjw[v{-}1]=\Up{S^{\prime}}\Down{S}\pjw[v{-}1],
\quad
\Up{S}\Up{S^{\prime}}\pjw[v{-}1]=\Up{S^{\prime}}\Up{S}\pjw[v{-}1].
\end{gather*}

\item \label{theorem:main-tl-section-4} \emph{Adjacency relations.}
If $\dist(S,S^{\prime})=1$ and $S^{\prime}>S$, then
\begin{gather*}
\Down{S^{\prime}}\Up{S}\pjw[v{-}1]=\Down{S{\cup}S^{\prime}}\pjw[v{-}1],
\quad
\Down{S}\Up{S^{\prime}}\pjw[v{-}1]=\Up{S^{\prime}{\cup}S}\pjw[v{-}1],
\\
\Down{S^{\prime}}\Down{S}\pjw[v{-}1]=\Up{S}\Down{S^{\prime}}\funcH[S]\pjw[v{-}1],
\quad
\Up{S}\Up{S^{\prime}}\pjw[v{-}1]=\funcH[S]\Up{S^{\prime}}\Down{S}\pjw[v{-}1].
\end{gather*}

\item \label{theorem:main-tl-section-5} \emph{Overlap relations.}
If $S^{\prime}\geq S$ with $S^{\prime}\cap S=\{s\}$ and $S^{\prime}\not\subset S$, then we have
\begin{gather*}
\Down{S^{\prime}}\Down{S}\pjw[v{-}1]=\Up{\{s\}}\Down{S}\Down{S^{\prime}{\setminus}\{s\}}\pjw[v{-}1],
\quad
\Up{S}\Up{S^{\prime}}\pjw[v{-}1]=\Up{S^{\prime}{\setminus}\{s\}}\Up{S}\Down{\{s\}}\pjw[v{-}1].
\end{gather*}

\item \label{theorem:main-tl-section-6} \emph{Zigzag.}
\begin{gather*}
\Down{S}\Up{S}\pjw[v{-}1]=\Up{\hull[S]}\Down{\hull[S]}\funcG[S]\pjw[v{-}1] 
+\Up{T}\Up{\hull[S]}\Down{\hull[S]}\Down{T}\funcF[S]\pjw[v{-}1].
\end{gather*}
Here, if the down-admissible hull $\hull[S]$, or 
the smallest minimal down-admissible stretch $T$ with $T>\hull[S]$ does not exist, then the involved symbols are 
zero by definition.
\end{enumerate}
 
\noindent\emph{(Basis)} The elements of the form
\begin{gather*}
\pjw[w{-}1]\Up{S_{i_{l}}^{\prime}}
\cdots\Up{S_{i_{0}}^{\prime}}\Down{S_{i_{0}}}\cdots\Down{S_{i_{k}}}\pjw[v{-}1],
\end{gather*}
with $S_{i_{l}}^{\prime}>\cdots>S_{i_{0}}^{\prime}$, and $S_{i_{0}}<\cdots<S_{i_{k}}$, 
form a basis for $\pjw[w{-}1]\zigzag\pjw[v{-}1]$. 

\noindent\emph{(Complete)} Any word 
$\pjw[w{-}1]\morstuff{F}\pjw[v{-}1]$ in the generators 
of $\ealg$ can be rewritten as a linear
combination of basis elements from \emph{(Basis)} using only the relations \emph{(1)--(6)}.
\end{theorem}

\begin{remark}
The algebra $\zigzag$ is a path algebra of an underlying quiver as follows. The idempotents $\pjw[v{-}1]$ correspond to vertices of a quiver, call these $v-1$. The elements $\Down{S}\pjw[v{-}1]$ and 
$\Up{S^{\prime}}\pjw[v{-}1]$ correspond to arrows starting at the vertex $v-1$, and either pointing to \emph{downwards} or \emph{upwards} (which is leftwards respectively rightwards in \autoref{figure:main-3}) to $v[S]-1$ or $v(S)-1$.
\end{remark}

\begin{remark}
In the special case of $v=w$,
\fullref{theorem:main-tl-section}.(Basis) says that $\ppar$loops form a basis of
the endomorphism spaces. Furthermore, we will see in \fullref{lemma:dualnumbers}
that all $\ppar$loops are products of $\ppar$loops
$\loopdown{S}{v{-}1}\pjw[v{-}1]$ for $S$ minimal down-admissible.
\end{remark}

\begin{remark}\label{remark:more-generators}
In \fullref{theorem:main-tl-section}.(4) and (6), the right-hand sides of the shown relations feature 
morphisms indexed by admissible subsets that are not necessarily minimal. 
We shall see in \fullref{lemma:loops-1} that 
such morphisms decompose into products of generators
\begin{gather}\label{eq:more-generators}
\Down{S}\pjw[v{-}1]
:=
\Down{S_{i_{1}}}\cdots\Down{S_{i_{k}}}\pjw[v{-}1],
\quad
\Up{S^{\prime}}\pjw[v{-}1]
:=
\Up{S^{\prime}_{i_{l}}}\cdots \Up{S^{\prime}_{i_{1}}}\pjw[v{-}1],
\end{gather}
where the products are taken over the minimal down- respectively up-admissible stretches $S_{i_{j}}$ and 
$S^{\prime}_{i_{j}}$, 
such that $S=\bigsqcup_{j} S_{i_{j}}$ and $S^{\prime}=\bigsqcup_{j} S^{\prime}_{i_{j}}$, with $S_{i_{1}}<\cdots<S_{i_{k}}$ and 
$S^{\prime}_{i_{l}}>\cdots>S^{\prime}_{i_{1}}$. 

In \fullref{theorem:main-tl-section} we use \eqref{eq:more-generators} as a \emph{shorthand notation}, 
but	one could also take $\Down{S}\pjw[v{-}1]$ and $\Up{S^{\prime}}\pjw[v{-}1]$ 
for (not necessarily minimal) admissible $S$ and $S^{\prime}$ as 
generators for $\ealg$. 
This requires listing the additional relations
\begin{gather}\label{eq:more-generatorstwo}
\Down{S}\pjw[v{-}1]
=
\Down{S_{1}}\Down{S_{2}}\pjw[v{-}1],
\quad
\Up{S^{\prime}}\pjw[v{-}1]
=
\Up{S^{\prime}_{2}}\Up{S^{\prime}_{1}}\pjw[v{-}1],
\end{gather}
for down-admissible $S_{1}<S_{2}$ with $S=S_{1}\cup S_{2}$ and up-admissible $S^{\prime}_{2}>S^{\prime}_{1}$ 
with $S^{\prime}=S^{\prime}_{2}\cup S^{\prime}_{1}$, 
in addition to the relations  
\fullref{theorem:main-tl-section}.(1-6) among minimal generators. 
One advantage of such a presentation is that it exhibits $\ealg$ as 
a quadratic algebra, since relations \fullref{theorem:main-tl-section}.(4-6)
now turn into quadratic relations with respect to the enlarged generating set.
\end{remark}

The proof of \fullref{theorem:main-tl-section} will occupy the remainder of this paper.
However, we already note that \fullref{theorem:main-tl-section}.(1) 
holds by the definition of $\ealg$ as the endomorphism algebra of a direct sum. 
Moreover, assuming the relations \fullref{theorem:main-tl-section}.(1-6), we get:

\begin{lemma}(Completeness---\fullref{theorem:main-tl-section}.(Complete).)\label{lemma:fullbasis}
Let $\pjw[w{-}1]\morstuff{F}\pjw[v{-}1]\in\ealg$. 
Then there is a finite sequence of relations \fullref{theorem:main-tl-section}.(1-6)
rewriting it as a linear combination of elements of the form \fullref{theorem:main-tl-section}.(Basis).
\end{lemma}

\begin{proof}
We can immediately restrict to the case where $\pjw[w{-}1]\morstuff{F}\pjw[v{-}1]$ is 
a product of generators of $\ealg$ (rather than a linear combination of such). In order to prove the claim, 
we will show that, if $\pjw[w{-}1]\morstuff{F}\pjw[v{-}1]$ is not of the desired form, 
then we can measure its complexity by counting \emph{out-of-order} pairs of the following forms, 
all other pairs are called \emph{in-order}.
\begin{enumerate}[label=(\roman*)]

\setlength\itemsep{.15cm} 

\item $\Down{S^{\prime}}\Down{S}$ or $\Up{S}\Up{S^{\prime}}$ for $S^{\prime}\geq S$.

\item $\Down{S}\Up{S^{\prime}}$.

\end{enumerate}
A case-by-case check will verify that we can use our relations 
to reduce these to in-order pairs, which then inductively shows the claim. 
For the case-by-case check we write down the 
list of all combinations how stretches $S$ and $S^{\prime}$ can 
meet. A priori, there 
are $13$ such cases illustrated by
\begin{gather*}
\begin{aligned}
\min(S)<\min(S^{\prime})\colon&
\quad
\text{1a)}\;
\begin{tikzpicture}[anchorbase,tinynodes]
\draw (.75,.1) to (.25,.1);
\draw[densely dashed] (-.25,-.1) to (-.75,-.1);
\end{tikzpicture}
,\quad
\text{1b)}\;
\begin{tikzpicture}[anchorbase,tinynodes]
\draw (.5,.1) to (0,.1);
\draw[densely dashed] (0,-.1) to (-.5,-.1);
\end{tikzpicture}
,\quad
\text{1c)}\;
\begin{tikzpicture}[anchorbase,tinynodes]
\draw (.4,.1) to (-.1,.1);
\draw[densely dashed] (.1,-.1) to (-.4,-.1);
\end{tikzpicture}
,\quad
\text{1d)}\;
\begin{tikzpicture}[anchorbase,tinynodes]
\draw (-.5,.1) to (.5,.1);
\draw[densely dashed] (0,-.1) to (-.5,-.1);
\end{tikzpicture}
,\quad
\text{1e)}\;
\begin{tikzpicture}[anchorbase,tinynodes]
\draw (-.5,.1) to (.5,.1);
\draw[densely dashed] (-.25,-.1) to (.25,-.1);
\end{tikzpicture}
,
\\
\min(S)>\min(S^{\prime})\colon&
\quad
\text{2a)}\;
\begin{tikzpicture}[anchorbase,tinynodes]
\draw[densely dashed] (.75,-.1) to (.25,-.1);
\draw (-.25,.1) to (-.75,.1);
\end{tikzpicture}
,\quad
\text{2b)}\;
\begin{tikzpicture}[anchorbase,tinynodes]
\draw[densely dashed] (.5,-.1) to (0,-.1);
\draw (0,.1) to (-.5,.1);
\end{tikzpicture}
,\quad
\text{2c)}\;
\begin{tikzpicture}[anchorbase,tinynodes]
\draw[densely dashed] (.4,-.1) to (-.1,-.1);
\draw (.1,.1) to (-.4,.1);
\end{tikzpicture}
,\quad
\text{2d)}\;
\begin{tikzpicture}[anchorbase,tinynodes]
\draw[densely dashed] (-.5,-.1) to (.5,-.1);
\draw (0,.1) to (-.5,.1);
\end{tikzpicture}
,\quad
\text{2e)}\;
\begin{tikzpicture}[anchorbase,tinynodes]
\draw[densely dashed] (-.5,-.1) to (.5,-.1);
\draw (-.25,.1) to (.25,.1);
\end{tikzpicture}
,
\\
\min(S)=\min(S^{\prime})\colon&
\quad
\text{3a)}\;
\begin{tikzpicture}[anchorbase,tinynodes]
\draw (0,.1) to (-.5,.1);
\draw[densely dashed] (0,-.1) to (-1,-.1);
\end{tikzpicture}
,\quad
\text{3b)}\;
\begin{tikzpicture}[anchorbase,tinynodes]
\draw (0,.1) to (-1,.1);
\draw[densely dashed] (0,-.1) to (-1,-.1);
\end{tikzpicture}
,\quad
\text{3c)}\;
\begin{tikzpicture}[anchorbase,tinynodes]
\draw (0,.1) to (-1,.1);
\draw[densely dashed] (0,-.1) to (-.5,-.1);
\end{tikzpicture}
.
\end{aligned}
\end{gather*}
where the solid line represents $S$ and the dashed line $S^{\prime}$, with smaller entries appearing further to the right. 
Some of the illustrated cases never arise when considering minimal admissible stretches
and the remaining cases are precisely covered by our relations.
Let us do this in detail for the out-of-order pair $\Down{S^{\prime}}\Down{S}$. 
First, the cases 2a)--2e) as well as 1e) and 3c) are ruled out by the assumption $S^{\prime}\geq S$. 
The case 1a) is far-commutativity, the case 1b) adjacency, while 1d) and 3b) are covered by containment. 
The relation 3a) does not occur as $S^{\prime}$ would not be minimal. 
The remaining case 1c) is only possible if $S^{\prime}\cap S=\{\min(S^{\prime})\}$, in which case we can 
apply the overlap relation.
\end{proof}

\subsection{Basic properties of \texorpdfstring{$\ppar$}{p}JW projectors}

We invite the reader to illustrate the statements and proofs of the
next lemmas using the explicit diagrammatic examples of trapezes from
\fullref{example:trapezes} and of $\ppar$JW projectors from
\fullref{example:the-half-projectors}.

\begin{lemmaqed}(See \cite[Proposition 3.2]{BuLiSe-tl-char-p}.)\label{lemma:idemp} 
Suppose that $S$ and $S^{\prime}$ are down-admissible for $v$. Then we have 
\begin{gather*}
\qjw[{v[S]{-}1}]\downo{S}\upo{S^{\prime}}\qjw[{v[S^{\prime}]{-}1}] 
=
\begin{tikzpicture}[anchorbase,tinynodes]
\draw[JW] (.15,.3) rectangle (1.25,.6);
\node at (0.7,.4) {$\qjwm[{v[S]{-}1}]$};
\draw[JW] (.15,-.3) rectangle (1.25,-.6);
\node at (0.7,.-.5) {\scalebox{.875}{$\qjwm[{v[S^{\prime}]{-}1}]$}};
\trd{1.325}{.3}{$\qjwm[S]$}{-.075}{.3}
\tru{1.325}{.3}{\scalebox{.9}{$\qjwm[S^{\prime}]$}}{-.075}{-.3}
\end{tikzpicture}
=
\delta_{S,S^{\prime}}\lambda^{-1}_{v,S^{\prime}}\cdot
\begin{tikzpicture}[anchorbase,tinynodes]
\draw[JW] (.15,-.3) rectangle (1.25,.3);
\node at (0.7,-.05) {$\qjwm[{v[S^{\prime}]{-}1}]$};
\end{tikzpicture}
=
\delta_{S,S^{\prime}}\lambda^{-1}_{v,S^{\prime}}\qjw[{v[S^{\prime}]{-}1}].
\end{gather*}
Thus, the summands $\lambda_{v,S}\trap{S}{v{-}1}$ 
in \eqref{eq:pjwdef} are orthogonal idempotents in $\TL[\Q]$.
\end{lemmaqed}

\begin{lemma}\label{lemma:capidem} 
Suppose $S$ is down-admissible for $v$, 
and $S^{\prime}=\{s,\dots,s^{\prime}-1\}$ is a minimal 
down-admissible stretch for $v$. Then we have 
\begin{gather*}
\begin{tikzpicture}[anchorbase,tinynodes]
\tru{1.2}{.6}{$\qjwm[S]$}{0}{0}
\draw[usual] (.125,.6) 
to[out=90,in=180] (0.25,.725) node[above,yshift=-2pt]{$S^{\prime}$} to[out=0,in=90] (.375,.6);
\end{tikzpicture}
=	
\begin{cases}
(-1)^{a_s\ppar^s}
\tfrac{\fancest{v}{s}[S]}{\fancest{v}{s{-}1}[S]}\cdot
\TRU{\,$\qjwm[S{\setminus}S^{\prime}]$}
& \text{ if }s\in S,s^{\prime}\notin S,
\\[5pt]
\TRU{\,$\qjwm[S\cup S^{\prime}]$}
& \text{ if }s\notin S,s^{\prime}\in S,
\\[5pt]
0 & \text{ otherwise}.
\end{cases}
\end{gather*}
\end{lemma}

We will also use the non-zero cases in the form:
\begin{gather}\label{eq:captrapeze}
\lambda_{v,S}\cdot
\begin{tikzpicture}[anchorbase,tinynodes]
\tru{1.2}{.6}{$\qjwm[S]$}{0}{0}
\draw[usual] (.125,.6) 
to[out=90,in=180] (0.25,.725) node[above,yshift=-2pt]{$S^{\prime}$} to[out=0,in=90] (.375,.6);
\end{tikzpicture}
=
\begin{cases}
\lambda_{v[S^{\prime}],S{\setminus}S^{\prime}}\cdot
\TRU{\,$\qjwm[S{\setminus}S^{\prime}]$}
&\text{ if }s\in S,s^{\prime}\notin S,
\\[5pt]
\lambda_{v,S}\cdot
\TRU{\,$\qjwm[S\cup S^{\prime}]$}
&\text{ if }s\notin S,s^{\prime}\in S.
\end{cases}
\end{gather}

\begin{proof} 
Relation \eqref{eq:0kill} implies that $\down{S^{\prime}}\upo{S}\qjw[{v[S]{-}1}]=0$ 
if either $s\in S$, $s^{\prime}\in S$ or $s\notin S$, $s^{\prime}\notin S$. 
For the other cases, we define $S_{+}=\{t\in S\mid t>s^{\prime}\}$ and 
$S_{-}=\{t\in S\mid t<s\}$. If $s\in S$ and $s^{\prime}\notin S$, then we use 
far commutativity, relation \eqref{eq:0trace}, and $\down{S^{\prime}}=\down{s}$ to compute
\begin{gather*}
\begin{aligned}
\down{S^{\prime}}\upo{S}\qjw[{v[S]{-}1}]=
\upo{S_{+}}\down{S^{\prime}}\upo{S^{\prime}}\upo{S_{-}}\qjw[{v[S]{-}1}] &=
(-1)^{a_{s}\ppar^s}
\tfrac{\fancest{v}{s}[S]}{\fancest{v}{s{-}1}[S]}\upo{S_{+}}\upo{S_{-}}\qjw[{v[S]{-}1}]\\
&=
(-1)^{a_{s}\ppar^s}
\tfrac{\fancest{v}{s}[S]}{\fancest{v}{s{-}1}[S]}\upo{S{\setminus}S^{\prime}}\qjw[{v[S]{-}1}].
\end{aligned}
\end{gather*}
Similarly, if $s\notin S$ but $s^{\prime}\in S$, we use far commutativity and 
an isotopy to compute
\begin{gather*}
\down{S^{\prime}}\upo{S}\qjw[{v[S]{-}1}]=
\down{S^{\prime}}\upo{S_{+}\cup\{s^\prime\}}\upo{S_{-}}\qjw[{v[S]{-}1}]=
\upo{S_{+}\cup\{s^{\prime}\}}\upo{S^{\prime}}\upo{S_{-}}\qjw[{v[S]{-}1}]=
\upo{S^{\prime}\cup S}\qjw[{v[S]{-}1}],
\end{gather*}
which finishes the proof.
\end{proof}

\begin{lemma}\label{lemma:capidem2} 
Suppose that $S^{\prime}=\{s,\dots,s^{\prime}-1\}$ is the smallest minimal down-admissible stretch for $v$ and let $S$ be down-admissible for $\fancest{v}{s}=\mother$. Then we have:
\begin{gather}\label{eq:capR}
\begin{tikzpicture}[anchorbase,scale=.3,tinynodes]
\tr{2.4}{0.6}{\raisebox{-.05cm}{$\qjwm[S]$}}{0}{0}
\draw[usual] (-.6,-.6) to (-.6,.6) 
to[out=90,in=180] (0,1.2) node[above,yshift=-2pt]{$S^{\prime}$} to[out=0,in=90] (.6,.6);
\node at (1,1.2) {$\phantom{a}$};
\node at (1,-.6) {$\phantom{a}$};
\end{tikzpicture}	
=
\begin{cases}
\TRUD{$\qjwm[S]$}{\scalebox{.875}{$\qjwm[S\cup S^{\prime}]$}}
=
\upo{S}\qjw[{v[S\cup S^{\prime}]{-}1}]\downo{S\cup S^{\prime}} 
& \text{ if }s^{\prime}\notin S,
\\[5pt]
\TRUD{\scalebox{.875}{$\qjwm[S\cup S^{\prime}]$}}{$\qjwm[S]$}
=
\upo{S\cup S^{\prime}}\qjw[{v[S]{-}1}]\downo{S}
& \text{ if }s^{\prime}\in S.
\end{cases}
\end{gather}
\end{lemma}

\begin{proof} 
Similar, but easier than the proof of \fullref{lemma:capidem}.
\end{proof}

\begin{lemma}\label{lemma:ab-jw}
Let $e=\eve(v)$ and $w\leq v=\pbase{a_{j},\dots,a_{0}}{\ppar}$. Then we have
\begin{gather*}
\begin{tikzpicture}[anchorbase,scale=.25,tinynodes]
\draw[pQJW] (-.5,0) rectangle (2.5,2);
\draw[usual] (-1,0) to (-1,2);
\node at (1,.9) {$\pqjwm[w{-}1]$};
\draw[JW] (-1.5,-2) rectangle (2.5,0);
\node at (.5,-1.1) {$\pqjwm[v{-}1]$};
\node at (1,2.5) {$\phantom{a}$};
\node at (1,-2.5) {$\phantom{a}$};
\end{tikzpicture}
=
\begin{tikzpicture}[anchorbase,scale=.25,tinynodes]
\draw[JW] (-1.5,-1) rectangle (2.5,1);
\node at (0.5,-.1) {$\pqjwm[v{-}1]$};
\node at (1,2.5) {$\phantom{a}$};
\node at (1,-2.5) {$\phantom{a}$};
\end{tikzpicture}
=
\begin{tikzpicture}[anchorbase,scale=.25,tinynodes]
\draw[pQJW] (-.5,0) rectangle (2.5,-2);
\draw[usual] (-1,0) to (-1,-2);
\node at (1,-1.1) {$\pqjwm[w{-}1]$};
\draw[JW] (-1.5,0) rectangle (2.5,2);
\node at (0.5,.9) {$\pqjwm[v{-}1]$};
\node at (1,2.5) {$\phantom{a}$};
\node at (1,-2.5) {$\phantom{a}$};
\end{tikzpicture},
\quad
\begin{tikzpicture}[anchorbase,scale=.25,tinynodes]
\draw[JW] (-.5,0) rectangle (2.5,2);
\draw[usual] (-1,0) to (-1,2);
\node at (1,.9) {$\pqjwm[e{-}1]$};
\draw[pQJW] (-1.5,-2) rectangle (2.5,0);
\node at (.5,-1.1) {$\pqjwm[v{-}1]$};
\node at (1,2.5) {$\phantom{a}$};
\node at (1,-2.5) {$\phantom{a}$};
\end{tikzpicture}
=
\begin{tikzpicture}[anchorbase,scale=.25,tinynodes]
\draw[pQJW] (-1.5,-1) rectangle (2.5,1);
\node at (0.5,-.1) {$\pqjwm[v{-}1]$};
\node at (1,2.5) {$\phantom{a}$};
\node at (1,-2.5) {$\phantom{a}$};
\end{tikzpicture}
=
\begin{tikzpicture}[anchorbase,scale=.25,tinynodes]
\draw[JW] (-.5,0) rectangle (2.5,-2);
\draw[usual] (-1,0) to (-1,-2);
\node at (1,-1.1) {$\pqjwm[e{-}1]$};
\draw[pQJW] (-1.5,0) rectangle (2.5,2);
\node at (0.5,.9) {$\pqjwm[v{-}1]$};
\node at (1,2.5) {$\phantom{a}$};
\node at (1,-2.5) {$\phantom{a}$};
\end{tikzpicture}
.
\end{gather*}
\end{lemma}

\begin{proof}
The first pair of equalities is clear since $\pqjw[w{-}1]$ contains 
$\idtl[w{-}1]$ with coefficient $1$ and otherwise only cap and cup 
diagrams, which are killed by \eqref{eq:0kill}. 
For a down-admissible set $S$, let $i(S)=\max\{s\in S\mid a_s\neq 0\}$. 
For the second pair of equalities we express $\pqjw[v{-}1]$ as
\begin{gather*}
\pqjw[v{-}1]
=
\qjw[v{-}1]+
{\textstyle\sum_{i=0}^{j-1}}
\underbrace{\left({\textstyle\sum_{v[S]\in\supp[v],i=i(S)}}\;
\lambda_{v,S}\, 
\trap{S}{v{-}1}\right)}_{:=\pqjw[v{-}1](i)}.
\end{gather*} 
It follows from \fullref{lemma:idemp} that the 
summands $\pqjw[v{-}1](i)$ are orthogonal idempotents. Note 
that we can write $\pqjw[v{-}1](i)=\upo{i}\morstuff{F}(v,i)\downo{i}$ for 
some morphism $\morstuff{F}(v,i)$. In particular $\pqjw[v{-}1](i)$ absorbs $\qjw[{\fancest{v}{i}{-}1}]$ 
or smaller, and it annihilates all $\qjw[k]$ for $k>\fancest{v}{i}-1$. In particular, it absorbs $\qjw[{e{-}1}]$.
\end{proof}

We prove now a significant generalization of \cite[Proposition 3.3]{BuLiSe-tl-char-p} 
and the analog of \eqref{eq:0absorb}.

\begin{proposition}(Classical absorbtion.)\label{proposition:p-properties}
Let $w\leq v$. Then we have
\begin{gather*}
\begin{tikzpicture}[anchorbase,scale=.25,tinynodes]
\draw[pJW] (-.5,0) rectangle (2.5,2);
\draw[usual] (-1,0) to (-1,2);
\node at (1,.9) {$\pjwm[w{-}1]$};
\draw[pJW] (-1.5,-2) rectangle (2.5,0);
\node at (.5,-1.1) {$\pjwm[v{-}1]$};
\node at (1,2.5) {$\phantom{a}$};
\node at (1,-2.5) {$\phantom{a}$};
\end{tikzpicture}
=
\begin{tikzpicture}[anchorbase,scale=.25,tinynodes]
\draw[pJW] (-1.5,-1) rectangle (2.5,1);
\node at (0.5,-.1) {$\pjwm[v{-}1]$};
\node at (1,2.5) {$\phantom{a}$};
\node at (1,-2.5) {$\phantom{a}$};
\end{tikzpicture}
=
\begin{tikzpicture}[anchorbase,scale=.25,tinynodes]
\draw[pJW] (-.5,0) rectangle (2.5,-2);
\draw[usual] (-1,0) to (-1,-2);
\node at (1,-1.1) {$\pjwm[w{-}1]$};
\draw[pJW] (-1.5,0) rectangle (2.5,2);
\node at (0.5,.9) {$\pjwm[v{-}1]$};
\node at (1,2.5) {$\phantom{a}$};
\node at (1,-2.5) {$\phantom{a}$};
\end{tikzpicture}
.
\end{gather*} 
\end{proposition}

\begin{proof} 
We distinguish two cases. If $w\leq e=\eve(v)$, then we have 
\begin{gather*}
\begin{gathered}
\pqjw[w{-}1]\pqjw[v{-}1]=\pqjw[w{-}1]\qjw[e{-}1]\pqjw[v{-}1]=
\qjw[e{-}1]\pqjw[v{-}1]=\pqjw[v{-}1]
\end{gathered}
\end{gather*}
and the other equation follows by reflection.

On the other hand, if $w\geq\eve(v)$, then 
$\ancest\cap\ancest[w]\neq\emptyset$. Let $z=\fancest{v}{s}=\fancest{w}{t}$ 
denote the youngest common ancestor of $v$ and $w$. It follows 
that $u:=\fancest{v}{s{-}1}$ is the oldest ancestor of $v$ with $u\geq w$. 
Now, we have $\qjw[v{-}1]\qjw[w{-}1]=\qjw[v{-}1]$ and $\qjw[v{-}1]\pqjw[w{-}1](j)=0$ for any $j$, as well as 
\begin{gather*}
\pqjw[v{-}1](i)\qjw[w{-}1]=
\begin{cases} 
\pqjw[v{-}1](i) &\text{ if }i<s,
\\
0 &\text{ if } i\geq s,
\end{cases}
\qquad 
\pqjw[v{-}1](i)\pqjw[w{-}1](j)=\delta_{\fancest{v}{i},\fancest{w}{j}}\pqjw[v{-}1](i).
\end{gather*}
The latter is a consequence of \fullref{lemma:idemp}. Moreover, 
for each $i\geq s$, there exists exactly one $j$, such that $\fancest{v}{i}=\fancest{w}{j}$. Thus, we have
\begin{gather*}
\begin{aligned}
\pqjw[v{-}1]\pqjw[w{-}1]
&=\qjw[v{-}1]\pqjw[w{-}1]
+{\textstyle\sum_{i}}\,\pqjw[v{-}1](i)\qjw[w{-}1]
+{\textstyle\sum_{i,j}}\,\pqjw[v{-}1](i)\pqjw[w{-}1](j)
\\
&=\qjw[v{-}1]+{\textstyle\sum_{i<s}}\,\pqjw[v{-}1](i)
+{\textstyle\sum_{i\geq s}}\,\pqjw[v{-}1](i)
=\pqjw[v{-}1].
\end{aligned}
\end{gather*}
The computation for $\pqjw[w{-}1]\pqjw[v{-}1]$ is analogous.
\end{proof}

\begin{example}\label{example:classical-ab}
For $\ppar=3$ we have
\begin{gather*}
\begin{tikzpicture}[anchorbase,scale=.25,tinynodes]
\draw[pJW] (-.5,0) rectangle (2.5,2);
\draw[usual] (-1,0) to (-1,2);
\node at (1,.9) {$\pjwm[2]$};
\draw[pJW] (-1.5,-2) rectangle (2.5,0);
\node at (.5,-1.1) {$\pjwm[3]$};
\node at (1,2.5) {$\phantom{a}$};
\node at (1,-2.5) {$\phantom{a}$};
\end{tikzpicture}
=
\begin{tikzpicture}[anchorbase,scale=.25,tinynodes]
\draw[pJW] (-1.5,-1) rectangle (2.5,1);
\node at (0.5,-.1) {$\pjwm[3]$};
\node at (1,2.5) {$\phantom{a}$};
\node at (1,-2.5) {$\phantom{a}$};
\end{tikzpicture}
=
\begin{tikzpicture}[anchorbase,scale=.25,tinynodes]
\draw[pJW] (-.5,0) rectangle (2.5,-2);
\draw[usual] (-1,0) to (-1,-2);
\node at (1,-1.1) {$\pjwm[2]$};
\draw[pJW] (-1.5,0) rectangle (2.5,2);
\node at (0.5,.9) {$\pjwm[3]$};
\node at (1,2.5) {$\phantom{a}$};
\node at (1,-2.5) {$\phantom{a}$};
\end{tikzpicture}
,\quad
\begin{tikzpicture}[anchorbase,scale=.25,tinynodes]
\draw[pJW] (.5,0) rectangle (-2.5,2);
\draw[usual] (1,0) to (1,2);
\node at (-1,.9) {$\pjwm[2]$};
\draw[pJW] (1.5,-2) rectangle (-2.5,0);
\node at (-.5,-1.1) {$\pjwm[3]$};
\node at (-1,2.5) {$\phantom{a}$};
\node at (-1,-2.5) {$\phantom{a}$};
\end{tikzpicture}
\neq
\begin{tikzpicture}[anchorbase,scale=.25,tinynodes]
\draw[pJW] (1.5,-1) rectangle (-2.5,1);
\node at (-.5,-.1) {$\pjwm[3]$};
\node at (-1,2.5) {$\phantom{a}$};
\node at (-1,-2.5) {$\phantom{a}$};
\end{tikzpicture}
\neq
\begin{tikzpicture}[anchorbase,scale=.25,tinynodes]
\draw[pJW] (.5,0) rectangle (-2.5,-2);
\draw[usual] (1,0) to (1,-2);
\node at (-1,-1.1) {$\pjwm[2]$};
\draw[pJW] (1.5,0) rectangle (-2.5,2);
\node at (-.5,.9) {$\pjwm[3]$};
\node at (-1,2.5) {$\phantom{a}$};
\node at (-1,-2.5) {$\phantom{a}$};
\end{tikzpicture}
.
\end{gather*}
\end{example}

We also have the following relations with no classical analog.

\begin{proposition}(Non-classical absorbtion and shortening.)\label{proposition:absorb}
Let $S$ be a
down-admissible stretch for $v$. Then we have
\begin{gather*}
\begin{aligned}
\pjw[{v[S]{-}1}]\down{S}\pjw[v{-}1]
=
\down{S}\pjw[v{-}1]
&\leftrightsquigarrow
\begin{tikzpicture}[anchorbase,scale=.25,tinynodes]
\draw[pJW] (2,0.5) rectangle (-2,-.5);
\draw[pJW] (2,1.5) rectangle (-2,2.5);
\node at (0,-.2) {$\pjwm[v{-}1]$};
\draw[usual] (-.5,.5) to[out=90,in=180] (0,1) to[out=0,in=90] (.5,.5);
\draw[usual] (-1,.5) to (-1,1.5);
\draw[usual] (1,.5) to (1,1.5);
\node at (1,2.5) {$\phantom{a}$};
\node at (1,-.5) {$\phantom{a}$};
\end{tikzpicture}
=
\begin{tikzpicture}[anchorbase,scale=.25,tinynodes]
\draw[pJW] (2,0.5) rectangle (-2,-.5);
\node at (0,-.2) {$\pjwm[v{-}1]$};
\draw[usual] (-.5,.5) to[out=90,in=180] (0,1) to[out=0,in=90] (.5,.5);
\draw[usual] (-1,.5) to (-1,2.5);
\draw[usual] (1,.5) to (1,2.5);
\node at (1,2.5) {$\phantom{a}$};
\node at (1,-.5) {$\phantom{a}$};
\end{tikzpicture}
,
\\
\pjw[{v{-}1}]\up{S}\pjw[{v[S]{-}1}]
=
\pjw[{v{-}1}]\up{S}
&\leftrightsquigarrow
\begin{tikzpicture}[anchorbase,scale=.25,tinynodes]
\draw[pJW] (2,0.5) rectangle (-2,-.5);
\draw[pJW] (2,1.5) rectangle (-2,2.5);
\node at (0,1.8) {$\pjwm[v{-}1]$};
\draw[usual] (-.5,1.5) to[out=270,in=180] (0,1) to[out=0,in=270] (.5,1.5);
\draw[usual] (-1,.5) to (-1,1.5);
\draw[usual] (1,.5) to (1,1.5);
\node at (1,2.5) {$\phantom{a}$};
\node at (1,-.5) {$\phantom{a}$};
\end{tikzpicture}
=
\begin{tikzpicture}[anchorbase,scale=.25,tinynodes]
\draw[pJW] (2,1.5) rectangle (-2,2.5);
\node at (0,1.8) {$\pjwm[v{-}1]$};
\draw[usual] (-.5,1.5) to[out=270,in=180] (0,1) to[out=0,in=270] (.5,1.5);
\draw[usual] (-1,-.5) to (-1,1.5);
\draw[usual] (1,-.5) to (1,1.5);
\node at (1,2.5) {$\phantom{a}$};
\node at (1,-.5) {$\phantom{a}$};
\end{tikzpicture}
,
\\
\pjw[{v[S]-1}]\down{S}
(\idtl[{v-\fancest{v}{s}}]\hcirc\pjw[\fancest{v}{s}{-}1])
=
\down{S}\pjw[v{-}1]
&\leftrightsquigarrow
\begin{tikzpicture}[anchorbase,scale=.25,tinynodes]
\draw[pJW] (2,0.5) rectangle (0,-.5);
\draw[pJW] (2,1.5) rectangle (-2,2.5);
\draw[usual] (-.5,-.5)to (-.5,.5) to[out=90,in=180] (0,1) to[out=0,in=90] (.5,.5);
\draw[usual] (-1,-.5) to (-1,1.5);
\draw[usual] (1,.5) to (1,1.5);
\node at (1,2.5) {$\phantom{a}$};
\node at (1,-.5) {$\phantom{a}$};
\end{tikzpicture}
=
\begin{tikzpicture}[anchorbase,scale=.25,tinynodes]
\draw[pJW] (2,0.5) rectangle (-2,-.5);
\node at (0,-.2) {$\pjwm[v{-}1]$};
\draw[usual] (-.5,.5) to[out=90,in=180] (0,1) to[out=0,in=90] (.5,.5);
\draw[usual] (-1,.5) to (-1,2.5);
\draw[usual] (1,.5) to (1,2.5);
\node at (1,2.5) {$\phantom{a}$};
\node at (1,-.5) {$\phantom{a}$};
\end{tikzpicture}
,
\\
(\idtl[{v-\fancest{v}{s}}]\hcirc\pjw[{\fancest{v}{s}{-}1}])\up{S}\pjw[{v[S]{-}1}]
=\pjw[{v{-}1}]\up{S}
&\leftrightsquigarrow
\begin{tikzpicture}[anchorbase,scale=.25,tinynodes]
\draw[pJW] (2,0.5) rectangle (-2,-.5);
\draw[pJW] (2,1.5) rectangle (0,2.5);
\draw[usual] (-.5,2.5) to (-.5,1.5) to[out=270,in=180] (0,1) to[out=0,in=270] (.5,1.5);
\draw[usual] (-1,.5) to (-1,2.5);
\draw[usual] (1,.5) to (1,1.5);
\node at (1,2.5) {$\phantom{a}$};
\node at (1,-.5) {$\phantom{a}$};
\end{tikzpicture}
=
\begin{tikzpicture}[anchorbase,scale=.25,tinynodes]
\draw[pJW] (2,1.5) rectangle (-2,2.5);
\node at (0,1.8) {$\pjwm[v{-}1]$};
\draw[usual] (-.5,1.5) to[out=270,in=180] (0,1) to[out=0,in=270] (.5,1.5);
\draw[usual] (-1,-.5) to (-1,1.5);
\draw[usual] (1,-.5) to (1,1.5);
\node at (1,2.5) {$\phantom{a}$};
\node at (1,-.5) {$\phantom{a}$};
\end{tikzpicture}
.
\end{aligned}
\end{gather*}
Here $\fancest{v}{s}$ denotes the youngest ancestor of $v$ for which the $s$th digit is zero.
\end{proposition}

\begin{proof} 
If suffices to prove these relations in the case of minimal down-admissible stretches. 
To be consistent with the above notation, let us write 
$S^{\prime}=\{s,\dots,s^{\prime}-1\}$ instead of $S$.

In order to verify the first relation we compute, using \eqref{eq:captrapeze}, that
\begin{gather}\label{eq:capsummands} 
\begin{aligned}
\down{S^{\prime}}\pqjw[v{-}1]
&=
\down{S^{\prime}}
{\textstyle\sum_{v[S]\in\supp[v]}}\,
\lambda_{v,S}\trap{S}{v{-}1}
\\	
&=
{\textstyle\sum_{s\in S,s^{\prime}\notin S}}\,
\lambda_{v[S^{\prime}],S{\setminus}S^{\prime}}\,
\upo{S{\setminus}S^{\prime}}\qjw[{v[S]{-}1}]\downo{S}
+ 
{\textstyle\sum_{s\notin S,s^{\prime}\in S}}\,
\lambda_{v,S}\upo{S\cup S^{\prime}}\qjw[{v[S]{-}1}]\downo{S}.
\end{aligned}
\end{gather}
For $s$ with $s\in S,s^{\prime}\notin S$, we 
define $S_{-}=S\setminus S^{\prime}$. For $S$ with 
$s\notin S,s^{\prime}\in S$, we define $S_{+}=S\cup S^{\prime}$.  
It is easy to verify that the sets $S_{-}$ and $S_{+}$ are 
down-admissible for $v[S^{\prime}]$.

Then \fullref{lemma:idemp} implies that each summand in 
\eqref{eq:capsummands} is invariant under left multiplication by a unique 
summand in $\pqjw[{v[S^{\prime}]{-}1}]=
{\textstyle\sum_{v[S^{\prime}][X]\in\supp[{v[S^{\prime}]}]}}\;
\lambda_{v[S^{\prime}],X}\,\trap{X}{v[S^{\prime}]{-}1}$,
while it is killed under left multiplication 
by any other summand. 
This proves the first equation; the second absorption equation follows by reflection symmetry.

For later use, note that the relevant summands of 
$\pqjw[{v[S^{\prime}]{-}1}]$ 
are the $\lambda_{v[S^{\prime}],X}\trap{X}{v[S^{\prime}]{-}1}$ 
for which $s,s^{\prime}\in X$ or $s,s^{\prime}\notin X$. 

Now we are ready to prove the projector shortening relations.
We start by expanding
\begin{gather*}
\down{S^{\prime}}(\idtl[{v-\fancest{v}{s}}]\hcirc\pqjw[{\fancest{v}{s}{-}1}])
=
\down{S^{\prime}}
{\textstyle\sum_{\fancest{v}{s}[S]\in\supp[\fancest{v}{s}]}}\,
\lambda_{\fancest{v}{s},S}\,(\idtl[{v-\fancest{v}{s}}]\hcirc\trap{S}{\fancest{v}{s}{-}1}).
\end{gather*}
Note that the down-admissible sets $S$ for 
$\fancest{v}{s}=\pbase{a_{j},\dots,a_{s+1},0,0,\dots,0}{\ppar}$ are 
exactly the down-admissible sets of $v$ which are contained 
in $\{s+1,\dots,j-1\}$ and that $\lambda_{\fancest{v}{s},S}=\lambda_{v,S}$.
Recall that, by \eqref{eq:capR}, we have 
\begin{gather}\label{eq:capR2}
\down{S^{\prime}}
(\idtl[{v-\fancest{v}{s}}]\hcirc\trap{S}{\fancest{v}{s}{-}1})
=
\begin{cases}
\upo{S}\qjw[{\fancest{v}{s{-}1}[S\cup S^{\prime}]{-}1}]\downo{S\cup S^{\prime}} 
& \text{ if }s^{\prime}\notin S,
\\
\upo{S\cup S^{\prime}}\qjw[{\fancest{v}{s{-}1}[S]{-}1}]\downo{S}
& \text{ if }s^{\prime}\in S.
\end{cases}
\end{gather}
In the resulting elements we either see $\upo{S}$, 
with $s,s^{\prime}\notin S$ or $\upo{S\cup S^{\prime}}$ 
with $s,s^{\prime}\in S\cup S^{\prime}$.

Now, if $X$ is down-admissible for $v[S^{\prime}]$, we 
compute 
\begin{gather*}
\begin{gathered}
(
\lambda_{v[S^{\prime}],X}\,\trap{X}{v[S^{\prime}]{-}1}
)
\down{S^{\prime}}
(
\lambda_{\fancest{v}{s},S}\,\idtl[{v-\fancest{v}{s}}]\hcirc\trap{S}{\fancest{v}{s}{-}1}
)
\\
= 
\begin{cases}
c_{1}(v,Y)\,\upo{S_{-}}\qjw[{v[Y]-1}]\downo{Y}
& \text{ if }s\notin 
X=:S_{-},\,s^{\prime}\notin S,\, 
X(\geq s^{\prime})=S,\, 
Y:=S_{-}\cup S^{\prime},
\\
c_{2}(v,Y)\,\upo{S_{+}}\qjw[{v[Y]-1}]\downo{Y}
& \text{ if }s\in X=:S_{+},\, s^{\prime}\in S,\, 
X(\geq s^{\prime})=S,\,Y:=S_{+}\setminus S^{\prime},
\\
0 & \text{ otherwise},
\end{cases}
\end{gathered}
\end{gather*}
where the scalars $c_{1}(v,Y)$ and $c_{2}(v,Y)$ are computed as follows.
\begin{gather*}
\begin{aligned}
c_{1}(v,Y) 
&= 
\lambda_{v[S^{\prime}],X}\,\lambda_{\fancest{v}{s},S}\,\lambda^{-1}_{\fancest{v[S^{\prime}]}{s{-}1},S} 
\\
&=
\lambda_{v[S^{\prime}],X}\,\lambda_{\fancest{v}{s},S}\,\lambda^{-1}_{\fancest{v}{s},S}\, 
=
\lambda_{v[S^{\prime}],Y{\setminus}S^{\prime}}.
\\
c_{2}(v,Y)
&=
\lambda_{v[S^{\prime}],X}\,\lambda_{\fancest{v}{s},S}\,\lambda^{-1}_{\fancest{v[S^{\prime}]}{s{-}1},S\cup S^{\prime}}
\\ 
&=
(-1)^{a_{s}\ppar^s}\tfrac{\fancest{v}{s{-}1}[Y]}{\fancest{v}{s}[Y]}\lambda_{v,Y}\,\lambda_{\fancest{v}{s},S}\,
(-1)^{a_{s}\ppar^s}\tfrac{\fancest{v}{s}[Y]}{\fancest{v}{s{-}1}[Y]}\lambda^{-1}_{\fancest{v}{s},S}=\lambda_{v,Y}.
\end{aligned}
\end{gather*}
Thus, by \eqref{eq:capR2}, we have
\begin{align*}
\pqjw[{v[S^{\prime}]{-}1}]\down{S^{\prime}}(\idtl[{v-\fancest{v}{s}}]\hcirc\pqjw[{\fancest{v}{s}{-}1}])
&= 
{\textstyle\sum_{X,S}}\,
(
\lambda_{v[S^{\prime}],X}\,\trap{X}{v[S^{\prime}]{-}1}
) 
\down{S^{\prime}}
(
\lambda_{\fancest{v}{s},S}\,\idtl[{v-\fancest{v}{s}}]\hcirc\trap{S}{\fancest{v}{s}{-}1}
)\\
&= 
\down{S^{\prime}}\pqjw[v{-}1].
\end{align*}
This establishes the third relation. The analogous relation for cups follow by reflection.
\end{proof}

The characteristic $\ppar$ analog of \eqref{eq:0trace} is:

\begin{proposition}(Partial trace.)\label{proposition:p-properties-trace}
\begin{enumerate}[label=(\alph*)]

\setlength\itemsep{.15cm}

\item \label{proposition:p-properties-trace-a} For $v\notin\eve$, $a_{s}$ being the first non-zero digit of $v$, we have
\begin{gather}\label{eq:pp-trace}
\begin{tikzpicture}[anchorbase,scale=.25,tinynodes]
\draw[pJW] (-1.5,-1) rectangle (2.5,1);
\node at (0.5,-.1) {$\pjwm[v{-}1]$};
\draw[usual] (-1,1) to[out=90,in=0] (-1.5,1.5) to[out=180,in=90] 
(-2,1) to (-2,0) node[left,xshift=2pt]{$a_{s}\ppar^{s}$} to (-2,-1) 
to[out=270,in=180] (-1.5,-1.5) to[out=0,in=270] (-1,-1);
\node at (0,2.5) {$\phantom{a}$};
\node at (0,-2.5) {$\phantom{a}$};
\end{tikzpicture}
=(-1)^{a_{s}\ppar^s}2\cdot
\begin{tikzpicture}[anchorbase,scale=.25,tinynodes]
\draw[pJW] (-1.5,-1) rectangle (2.5,1);
\node at (0.5,-.1) {$\pjwm[{\mother{-}1}]$};
\node at (0,2.5) {$\phantom{a}$};
\node at (0,-2.5) {$\phantom{a}$};
\end{tikzpicture}.
\end{gather}
On the other hand, if $v\in\eve$ and $v\geq\ppar$, then the (partial) trace of $\pjw[v{-}1]$ is zero.

\item \label{proposition:p-properties-trace-b} Let $S$ be down-admissible for
$v$ and $S^{\prime}$ the smallest minimal down-admissible stretch for $v$. Then
the partial trace on the bundle of strands specified by $S^{\prime}$ evaluates as:
\begin{gather*}
\pTr_{S^{\prime}}(\loopdown{S}{v{-}1})
=
\begin{cases} 
\loopdown{S{\setminus}S^{\prime}}{\mother{-}1} & \text{if }S^{\prime}\subset S,
\\
(-1)^{v{-}\mother}2\cdot\loopdown{S}{\mother{-}1} & \text{if }S^{\prime}\not\subset S.
\end{cases}
\end{gather*}
\end{enumerate}
\end{proposition}

\begin{proof} 
The second claim in \fullref{proposition:p-properties-trace}.(a), concerning the case of $v\in\eve$, 
follows from $\pqjw[v{-}1]=\qjw[v{-}1]$ and \eqref{eq:0trace}, which produces a scalar $a\in\Q$ with $\ord(a)>0$.
The case $v\notin\eve$ follows immediately by applying \eqref{eq:0trace} 
to the two expressions in the bracket in \eqref{eq:recursion-formula}.

In \fullref{proposition:p-properties-trace}.(a) we have already seen the case $S=\emptyset$, 
so we assume that $S\neq\emptyset$. We then apply the projector shortening 
relations from \fullref{proposition:absorb}, and get the following two cases 
for $\pTr_{S^{\prime}}(\loopdown{S}{v{-}1})$, depending on whether $S^{\prime}\subset S$ or $S^{\prime}\not\subset S$.
\begin{gather*}
\begin{tikzpicture}[anchorbase,scale=.25,tinynodes]
\draw[pJW] (2,-.5) rectangle (-2,.5);
\draw[pJW] (2,1.5) rectangle (-2,2.5);
\draw[pJW] (2,3.5) rectangle (-2,4.5);
\node at (0,-.2) {$\pjwm[v{-}1]$};
\node at (0,3. 8) {$\pjwm[v{-}1]$};
\draw[usual] (-1.5,.5) to[out=90,in=180] (-1,1) to[out=0,in=90] (-.5,.5);
\draw[usual] (-1.5,3.5) to[out=270,in=180] (-1,3) to[out=0,in=270] (-.5,3.5);
\draw[usual] (1.5,.5) to (1.5,1.5);
\draw[usual] (1.5,2.5) to (1.5,3.5);
\node at (1,2.5) {$\phantom{a}$};
\node at (1,-.5) {$\phantom{a}$};
\draw[usual] (-1.5,4.5) to[out=90,in=90] (-3,4.5) to (-3,2) to (-3,-.5) to[out=270,in=270] (-1.5,-.5);
\end{tikzpicture}
=
\begin{tikzpicture}[anchorbase,scale=.25,tinynodes]
\draw[pJW] (2,-.5) rectangle (-1,.5);
\draw[pJW] (2,1.5) rectangle (-2,2.5);
\draw[pJW] (2,3.5) rectangle (-1,4.5);
\draw[usual] (-1.5,-.5) to (-1.5,.5) to[out=90,in=180] (-1,1) to[out=0,in=90] (-.5,.5);
\draw[usual] (-1.5,4.5) to (-1.5,3.5) to[out=270,in=180] (-1,3) to[out=0,in=270] (-.5,3.5);
\draw[usual] (1.5,.5) to (1.5,1.5);
\draw[usual] (1.5,2.5) to (1.5,3.5);
\node at (1,2.5) {$\phantom{a}$};
\node at (1,-.5) {$\phantom{a}$};
\draw[usual](-1.5,4.5) to[out=90,in=90] (-3,4.5) to (-3,2) 	to (-3,-.5) to[out=270,in=270] (-1.5,-.5);
\end{tikzpicture}
= 
(-1)^{v{-}\mother}2\cdot\;
\begin{tikzpicture}[anchorbase,scale=.25,tinynodes]
\draw[pJW] (2,-.5) rectangle (-1,.5);
\draw[pJW] (2,1.5) rectangle (-1,2.5);
\draw[pJW] (2,3.5) rectangle (-1,4.5);
\draw[usual] (-.5,.5) to (-.5,1.5);
\draw[usual] (-.5,2.5) to (-.5,3.5);	
\draw[usual] (1.5,.5) to (1.5,1.5);
\draw[usual] (1.5,2.5) to (1.5,3.5);
\node at (1,2.5) {$\phantom{a}$};
\node at (1,-.5) {$\phantom{a}$};
\end{tikzpicture}
\quad
\text{if }
S^{\prime}\subset S
,
\\
\begin{tikzpicture}[anchorbase,scale=.25,tinynodes]
\draw[pJW] (2,-.5) rectangle (-2,.5);
\draw[pJW] (2,1.5) rectangle (-2,2.5);
\draw[pJW] (2,3.5) rectangle (-2,4.5);
\node at (0,-.2) {$\pjwm[v{-}1]$};
\node at (0,3.8) {$\pjwm[v{-}1]$};
\draw[usual] (-1,.5) to[out=90,in=180] (-.5,1) to[out=0,in=90] (0,.5);
\draw[usual] (-1,3.5) to[out=270,in=180] (-0.5,3) to[out=0,in=270] (0,3.5);
\draw[usual] (-1.5,.5) to (-1.5,1.5);
\draw[usual] (1.5,.5) to (1.5,1.5);
\draw[usual] (-1.5,2.5) to (-1.5,3.5);
\draw[usual] (1.5,2.5) to (1.5,3.5);
\node at (1,2.5) {$\phantom{a}$};
\node at (1,-.5) {$\phantom{a}$};
\draw[usual] (-1.5,4.5) to[out=90,in=90] (-3,4.5) to (-3,2)	 to (-3,-.5) to[out=270,in=270] (-1.5,-.5);
\end{tikzpicture}
=
\begin{tikzpicture}[anchorbase,scale=.25,tinynodes]
\draw[pJW] (2,-.5) rectangle (-.5,.5);
\draw[pJW] (2,1.5) rectangle (-2,2.5);
\draw[pJW] (2,3.5) rectangle (-.5,4.5);
\draw[usual] (-1,-.5)to (-1,.5) to[out=90,in=180] (-.5,1) to[out=0,in=90] (0,.5);
\draw[usual] (-1,4.5) to (-1,3.5) to[out=270,in=180] (-.5,3) to[out=0,in=270] (0,3.5);
\draw[usual] (1.5,.5) to (1.5,1.5);
\draw[usual] (-1.5,2.5) to (-1.5,3.5);
\draw[usual] (1.5,2.5) to (1.5,3.5);
\node at (1,2.5) {$\phantom{a}$};
\node at (1,-.5) {$\phantom{a}$};
\draw[usual](-1.5,2.5) to (-1.5,4.5) to[out=90,in=90] (-3,4.5) to (-3,2) to (-3,-.5) to[out=270,in=270] (-1.5,-.5) to  (-1.5,1.5) ;
\end{tikzpicture}
= 
(-1)^{v{-}\mother}2\cdot\;
\begin{tikzpicture}[anchorbase,scale=.25,tinynodes]
\draw[pJW] (2,-.5) rectangle (-.5,.5);
\draw[pJW] (2,1.5) rectangle (-1,2.5);
\draw[pJW] (2,3.5) rectangle (-.5,4.5);
\draw[usual] (-1,-.5)to (-1,.5) to[out=90,in=180] (-.5,1) to[out=0,in=90] (0,.5);
\draw[usual] (-1,4.5) to (-1,3.5) to[out=270,in=180] (-.5,3) to[out=0,in=270] (0,3.5);
\draw[usual] (1.5,.5) to (1.5,1.5);
\draw[usual] (1.5,2.5) to (1.5,3.5);
\node at (1,2.5) {$\phantom{a}$};
\node at (1,-.5) {$\phantom{a}$};
\end{tikzpicture}
= 
(-1)^{v{-}\mother}2\cdot\;
\begin{tikzpicture}[anchorbase,scale=.25,tinynodes]
\draw[pJW] (2,-.5) rectangle (-1.5,.5);
\draw[pJW] (2,1.5) rectangle (-1.5,2.5);
\draw[pJW] (2,3.5) rectangle (-1.5,4.5);
\draw[usual] (-1,.5) to[out=90,in=180] (-.5,1) to[out=0,in=90] (0,.5);
\draw[usual] (-1,3.5) to[out=270,in=180] (-.5,3) to[out=0,in=270] (0,3.5);
\draw[usual] (1.5,.5) to (1.5,1.5);
\draw[usual] (1.5,2.5) to (1.5,3.5);
\node at (1,2.5) {$\phantom{a}$};
\node at (1,-.5) {$\phantom{a}$};
\end{tikzpicture}
\quad
\text{if }
S^{\prime}\not\subset S
.
\end{gather*}
Here we have used \fullref{proposition:p-properties-trace} for the second equation in the bottom row.
\end{proof}

\begin{example}\label{example:partial-p-trace}
Note that \eqref{eq:pp-trace} and \eqref{eq:0trace} (for eves) give a recursive way to compute 
traces. For example, for $v=\pbase{a,b,c}{\ppar}$ we have 
$\mother=\pbase{a,b,0}{\ppar}$, $\motherr{v}{2}=\pbase{a,0,0}{\ppar}$, and
\begin{gather*}
\begin{tikzpicture}[anchorbase,scale=.25,tinynodes]
\draw[pJW] (-1.5,-1) rectangle (2.5,1);
\node at (.5,-.1) {$\pjwm[v{-}1]$};
\draw[usual] (-1,1) to[out=90,in=0] (-1.5,1.5) node[above,yshift=-2pt]{$\{2{,}1{,}0\}$} to[out=180,in=90] 
(-2,1) to (-2,0) to (-2,-1)
to[out=270,in=180] (-1.5,-1.5) to[out=0,in=270] (-1,-1);
\node at (0,2.5) {$\phantom{a}$};
\node at (0,-2.5) {$\phantom{a}$};
\end{tikzpicture}
=
\begin{tikzpicture}[anchorbase,scale=.25,tinynodes]
\draw[pJW] (-1.5,-1) rectangle (3,1);
\node at (.75,-.1) {$\pjwm[v{-}1]$};
\draw[usual] (-1,1) to[out=90,in=0] (-1.5,1.5) to[out=180,in=90] 
(-2,1) to (-2,0) node[left,xshift=2pt]{$c$} to (-2,-1)
to[out=270,in=180] (-1.5,-1.5) to[out=0,in=270] (-1,-1);
\draw[usual] (0,1) to[out=90,in=0] (-1.5,2.5) to[out=180,in=90] 
(-3,1) to (-3,0) to (-3,-1) to[out=270,in=180] (-1.5,-2.5) to[out=0,in=270] (0,-1);
\draw[usual] (1,1) to[out=90,in=0] (-1.5,3.5) to[out=180,in=90] 
(-4,1) to (-4,0) to (-4,-1) to[out=270,in=180] (-1.5,-3.5) to[out=0,in=270] (1,-1);
\node at (0,2.5) {$\phantom{a}$};
\node at (0,-2.5) {$\phantom{a}$};
\end{tikzpicture}
=(-1)^{c}2\cdot
\begin{tikzpicture}[anchorbase,scale=.25,tinynodes]
\draw[pJW] (-1.5,-1) rectangle (3,1);
\node at (.75,-.1) {$\pjwm[\mother{-}1]$};
\draw[usual] (0,1) to[out=90,in=0] (-1.5,2.5) to[out=180,in=90] 
(-3,1) to (-3,0) node[right,xshift=-2pt]{$b\ppar$} to (-3,-1) to[out=270,in=180] (-1.5,-2.5) to[out=0,in=270] (0,-1);
\draw[usual] (1,1) to[out=90,in=0] (-1.5,3.5) to[out=180,in=90] 
(-4,1) to (-4,0) to (-4,-1) to[out=270,in=180] (-1.5,-3.5) to[out=0,in=270] (1,-1);
\node at (0,2.5) {$\phantom{a}$};
\node at (0,-2.5) {$\phantom{a}$};
\end{tikzpicture}
=(-1)^{b\ppar+c}2^{2}\cdot
\begin{tikzpicture}[anchorbase,scale=.25,tinynodes]
\draw[pJW] (-1.5,-1) rectangle (3,1);
\node at (.75,-.1) {$\pjwm[\motherr{v}{2}{-}1]$};
\draw[usual] (1,1) to[out=90,in=0] (-1.5,3.5) to[out=180,in=90] 
(-4,1) to (-4,0) node[right,xshift=3pt,yshift=-16pt]{$a\ppar^{2}{-}1$} to (-4,-1) to[out=270,in=180] (-1.5,-3.5) to[out=0,in=270] (1,-1);
\node at (0,2.5) {$\phantom{a}$};
\node at (0,-2.5) {$\phantom{a}$};
\end{tikzpicture}
=0.
\end{gather*}
The proposition also implies that the (full) trace of 
the $\ppar$JW projectors are zero unless $v<\ppar$.
\end{example}

\subsection{Morphisms between \texorpdfstring{$\ppar$}{p}JW projectors---the linear structure}\label{subsec:hom-TL}

First, we state direct consequences of classical absorption, see 
\fullref{proposition:p-properties}, and non-classical absorption, see \fullref{proposition:absorb}.

\begin{lemmaqed}\label{lemma:loops-1}
\begin{enumerate}[label=(\alph*)]

\setlength\itemsep{.15cm}

\item \label{lemma:loops-1-a} 
If $S=\bigsqcup_{j=0}^{k}S_{i_{j}}$ with $S_{i_{k}}>\cdots>S_{i_{1}}>S_{i_{0}}$, each down-admissible for $v$,
and $S^{\prime}=\bigsqcup_{j=0}^{l}S_{i_{j}}^{\prime}$ with $S_{i_{l}}^{\prime}>\cdots>S_{i_{1}}^{\prime}>S_{i_{0}}^{\prime}$, 
each up-admissible for $v$, then 
\begin{gather*}
\Down{S}\pjw[v{-}1]=\Down{S_{i_{0}}}\cdots\Down{S_{i_{k}}}\pjw[v{-}1],
\quad\quad 
\Up{S^{\prime}}\pjw[v{-}1]=\Up{S_{i_{l}}^{\prime}}\cdots\Up{S_{i_{0}}^{\prime}}\pjw[v{-}1].
\end{gather*}

\item \label{lemma:loops-1-b} 
Let $S$ and $S^{\prime}$ be down- and up-admissible for $v$, respectively. Then we have
\begin{gather*}
\begin{aligned}
\pjw[{v[S]{-}1}]\down{S}\pjw[v{-}1]
&=
\down{S}\pjw[v{-}1]
=\pjw[{v[S]{-}1}]\down{S}(\idtl[v-\fancest{v(S)}{S}]\hcirc\pjw[\fancest{v}{S}{-}1]).
\\
\pjw[{v(S^{\prime}){-}1}]\up{S^{\prime}}\pjw[v{-}1] 
&=
\pjw[{v(S^{\prime}){-}1}]\up{S^{\prime}}
=(\idtl[v-\fancest{v(S^{\prime})}{S^{\prime}}]\hcirc\pjw[{\fancest{v(S^{\prime})}{S^{\prime}}{-}1}])\up{S^{\prime}}\pjw[v{-}1].
\end{aligned}
\end{gather*}
Here $\fancest{v}{S}$ denotes the youngest ancestor 
of $v$ for which all digits with indices in $S$ are zero.\qedhere
\end{enumerate}
\end{lemmaqed}

Let $\morstuff{F}$ be a cap (or cup) configuration, \ie a
Temperley--Lieb diagram consisting only of caps (or cups), with source (target)
$v-1$. We say that $\morstuff{F}$ is ancestor-centered for $v$, and write
$\morstuff{F}\in\ancest$, if each cap or cup has its center immediately to the
right of an ancestor strand of $v$.

\begin{example}\label{example:anccent}
A way to illustrate ancestor-centering is imagine a line with a marker 
$\star$ to the right of each 
ancestor strand of $v$. (There are $v-1$ strands in total and $\generation$ many $\star$.)
For example, for $\ppar=3$ and $v=13=\pbase{1,1,1}{3}$ we have
\begin{gather*}
\begin{tikzpicture}[anchorbase,scale=.25,tinynodes]
\draw[usual] (0,0) node[below]{$\bullet$} to (0,2);
\draw[usual] (-1,0) node[below]{$\bullet$} to (-1,2);
\draw[usual] (-2,0) node[below]{$\bullet$} to (-2,2);
\draw[usual] (-3,0) node[below]{$\bullet$} to (-3,2);
\draw[usual] (-4,0) node[below]{$\bullet$} to (-4,2);
\draw[usual] (-5,0) node[below]{$\bullet$} to (-5,2);
\draw[usual] (-7,0) node[below]{$\bullet$} to[out=90,in=0] (-7.5,1) to[out=180,in=90] (-8,0) node[below]{$\bullet$};
\draw[usual] (-6,0) node[below]{$\bullet$} to[out=90,in=0] (-7.5,2) to[out=180,in=90] (-9,0) node[below]{$\bullet$};
\draw[usual] (-10,0) node[below]{$\bullet$} to[out=90,in=0] (-10.5,1) to[out=180,in=90] (-11,0) node[below]{$\bullet$};
\node at (-7.5,-.55) {$\star$};
\node at (-10.5,-.55) {$\star$};
\end{tikzpicture}
\in\ancest,\quad
\begin{tikzpicture}[anchorbase,scale=.25,tinynodes]
\draw[usual] (0,0) node[below]{$\bullet$} to (0,2);
\draw[usual] (-1,0) node[below]{$\bullet$} to (-1,2);
\draw[usual] (-2,0) node[below]{$\bullet$} to (-2,2);
\draw[usual] (-3,0) node[below]{$\bullet$} to (-3,2);
\draw[usual] (-4,0) node[below]{$\bullet$} to (-4,2);
\draw[usual] (-6,0) node[below]{$\bullet$} to[out=90,in=0] (-6.5,1) to[out=180,in=90] (-7,0) node[below]{$\bullet$};
\draw[usual] (-5,0) node[below]{$\bullet$} to[out=90,in=0] (-6.5,2) to[out=180,in=90] (-8,0) node[below]{$\bullet$};
\draw[usual] (-9,0) node[below]{$\bullet$} to (-9,2);
\draw[usual] (-10,0) node[below]{$\bullet$} to[out=90,in=0] (-10.5,1) to[out=180,in=90] (-11,0) node[below]{$\bullet$};
\node at (-7.5,-.55) {$\star$};
\node at (-10.5,-.55) {$\star$};
\end{tikzpicture}
\notin\ancest.
\end{gather*}
Moreover, we have $\down{i}\idtl,\idtl\up{i}\in\ancest$, with mid point right to 
$\pbase{a_{j},\dots,a_{i+1},0,\dots,0}{\ppar}\in\ancest$. 
More generally, the morphisms $\down{S}\idtl$ and $\idtl\up{S}$ are ancestor-centered if $S$ is down- or up-admissible, 
respectively.
\end{example}

The following is the analog of \eqref{eq:0kill}.

\begin{lemma}\label{lemma:anccent} 
For a cap configuration $\morstuff{d}$ we have $\morstuff{d}\pjw[v{-}1]=0$
unless $\morstuff{d}\in\ancest$. Analogously for cup configurations.
\end{lemma}

\begin{proof} 
By assumption, $\morstuff{d}$ contains a cap which is 
not centered around an ancestor of $v$. By expanding $\morstuff{d}\pqjw[v{-}1]$ 
along \eqref{eq:pjwdef}, we see that this cap hits a JW projector in every 
summand in \eqref{eq:pjwdef}, and thus annihilates $\pqjw[v{-}1]$.
\end{proof}

\begin{lemma}\label{lemma:sbasis}
\begin{enumerate}[label=(\alph*)]

\setlength\itemsep{.15cm}

\item \label{lemma:sbasis-a} Suppose that $S$ and $S^{\prime}$ are down-admissible for $w$ and $v$, respectively, with $w[S]=v[S^{\prime}]$.  Then we have
\begin{gather}\label{eq:morfilt}
\pqjw[w{-}1]\Up{S}\Down{S^{\prime}}\pqjw[v{-}1] 
= 
\upo{S}\qjw[{v[S^{\prime}]{-}1}]\downo{S^{\prime}} 
+{\textstyle\sum_{X,Y}}\,
c_{X,Y}\upo{X}\qjw[{v[Y]{-}1}]\downo{Y},
\end{gather}
for some coefficients $c_{X,Y}\in \Q$, where $X$ and $Y$ are down-admissible for $w$ and $v$, respectively, and $w[X]=v[Y]<v[S^{\prime}]$.
(In other words 
$\pqjw[w{-}1]\Up{S}\Down{S^{\prime}}\pqjw[v{-}1]-
\upo{S}\qjw[{v[S^{\prime}]{-}1}]\downo{S^{\prime}}\in\catstuff{td}_{>}$.)

\item \label{lemma:sbasis-b} We have isomorphisms of $\Q$-vector spaces
\begin{gather}\label{eq:corollary-pbasis}
\Hom_{\TL[\Q]}(\pqjw[v{-}1],\pqjw[w{-}1])
\cong
\spanQ{\upo{S}\qjw[{v[S]{-}1}]\downo{S^{\prime}}}
\cong
\spanQ{\pqjw[w{-}1]\Up{S}\Down{S^{\prime}}\pqjw[v{-}1]},
\end{gather} 
where $(S,S^{\prime})$ ranges over pairs of sets that are down-admissible for $w$ and $v$, respectively, such that $w[S]=v[S^{\prime}]$.
In particular, $\End_{\TL[\Q]}(\pqjw[v{-}1])\cong
\spanQ{\loopdown{S}{v{-}1}}$.
\end{enumerate}
\end{lemma}

Note that the second isomorphism in \eqref{eq:corollary-pbasis} is unitriangular by \eqref{eq:morfilt}. 
We will refer to morphisms of the form $\upo{S}\qjw[{v[S]{-}1}]\downo{S^{\prime}}$ as \emph{standard morphisms} 
and to morphisms of the form $\pqjw[w{-}1]\Up{S}\Down{S^{\prime}}\pqjw[v{-}1]$ as \emph{$\ppar$-morphisms}.

\begin{proof} 
The proof of (a) proceeds by iterating \fullref{lemma:capidem}. 
Let $S=\bigsqcup_{i}S_{i}$ and $S^{\prime}=\bigsqcup_{j}S^{\prime}_{j}$ be the partitions 
into minimal admissible stretches of consecutive integers with the 
usual ordering. Then we expand
\begin{gather*}
\Down{S^{\prime}}\pqjw[v{-}1]=\Down{S^{\prime}_{1}}\cdots\Down{S^{\prime}_{l}}\pqjw[v{-}1]
={\textstyle\sum_{X}}\,\down{S^{\prime}_{1}} 
\cdots\down{S^{\prime}_{l{-}1}}\cdot 
\lambda_{v,X}\cdot\down{S^{\prime}_{l}} 
\upo{X}\qjw[{v[X]{-}1}]\downo{X}
\\  
\in{\textstyle\sum_{X{\supset}S^{\prime}_{l}}}\, 
\down{S^{\prime}_{1}}\cdots\down{S^{\prime}_{l{-}2}}\cdot
\lambda_{v[X{\setminus}S^{\prime}_{l}],X{\setminus}S^{\prime}_{l}} 
\cdot\down{S^{\prime}_{l{-}1}}\upo{R{\setminus}S^{\prime}_{l}}\qjw[{v[X]{-}1}]\downo{X}  
+\catstuff{td}_{>},
\end{gather*} 
since by \eqref{eq:captrapeze}, we have
$\lambda_{v,X}\cdot\down{S^{\prime}_{l}}\upo{X}=\lambda_{v,X{\setminus}S^{\prime}_{l}}\cdot\upo{X{\setminus}S^{\prime}_{l}}$
if $S^{\prime}_{l}\subset X$, and otherwise $\max(S^{\prime}_{l})+1\in X$ and
thus $v[X]<v[S^{\prime}]$. Here we write $\catstuff{td}_{>}$ for
the ideal of morphisms of smaller through-degree than the leading term. We now iterate this argument to find 
\begin{gather*}
\Down{S^{\prime}}\pqjw[v{-}1] 
\in\qjw[{v[S^{\prime}]{-}1}]\downo{S^{\prime}}+\catstuff{td}_{>},
\quad
\pqjw[w{-}1]\Up{S} 
\in\upo{S}\qjw[{w[S]{-}1}]
+\catstuff{td}_{>},
\end{gather*}
which together imply \eqref{eq:morfilt}.

To see the first isomorphism in \eqref{eq:corollary-pbasis}:
For a given $\morstuff{F}\in\Hom_{\TL[\Q]}(v-1,w-1)$, we compute
\begin{gather*}
\begin{aligned}
\pqjw[w{-}1]\morstuff{F}\,\pqjw[v{-}1] 
&= 
{\textstyle\sum_{S,S^{\prime}}}\,(
\lambda_{w,S}\upo{S}\qjw[{w[S]{-}1}]\downo{S}
)
\morstuff{F}
(
\lambda_{v,S^{\prime}}\upo{S^{\prime}}\qjw[{v[S^{\prime}]{-}1}]\downo{S^{\prime}}
)
\\
&= 
{\textstyle\sum_{S,S^{\prime}}}\,\delta_{w[S],v[S^{\prime}]}
\upo{S}
(\lambda_{w,S}\lambda_{v,S^{\prime}}\qjw[{w[S]{-}1}]
\downo{S}\morstuff{F}\upo{S^{\prime}}\qjw[{v[S^{\prime}]{-}1}]
)\downo{S^{\prime}}
\\
&= 
{\textstyle\sum_{S,S^{\prime}}}\,\delta_{w[S],v[S^{\prime}]}c_{X,S,S^{\prime}}\upo{S}\qjw[{v[S^{\prime}]{-}1}]\downo{S^{\prime}},
\end{aligned}
\end{gather*}
where $c_{X,S,S^{\prime}}\in\Q$. In the last two lines, we have used 
\fullref{lemma:idemp} and the fact the 
JW projectors have no endomorphisms besides scalar multiples of the identity, \cf \eqref{eq:0kill}.
Finally, \eqref{eq:morfilt} implies then the second isomorphism in \eqref{eq:corollary-pbasis}. 
\end{proof}

\begin{lemma}(\fullref{theorem:main-tl-section}.(Basis).)\label{lemma:admsum}
\begin{enumerate}[label=(\alph*)]

\setlength\itemsep{.15cm}

\item \label{lemma:admsum-a} Suppose that a $\ppar$-admissible morphism is
expressed as 
\begin{gather}\label{eq:admsum}
{\textstyle\sum_{S,S^{\prime}}}\,r_{S,S^{\prime}}\cdot\pqjw[w{-}1]\Up{S}\Down{S^{\prime}}\pqjw[v{-}1] 
\in\Hom_{\TL[\Q]}(v-1,w-1),
\end{gather}
where $r_{S,S^{\prime}}\in\Q$ and the sum $(S,S^{\prime})$ ranges over pairwise distinct pairs of
sets that are down-admissible for $w$ and $v$, respectively, such that
$w[S]=v[S^{\prime}]$. Then every coefficient $r_{S,S^{\prime}}$ is
$\ppar$-admissible.

\item \label{lemma:admsum-b} We have the $\F$-vector space isomorphisms 
\begin{gather*}
\Hom_{\TL[\F]}(\pjw[v{-}1],\pjw[w{-}1])\cong
\spanFp{\pjw[w{-}1]\Up{S}\Down{S^{\prime}}\pjw[v{-}1]
},
\end{gather*} 
where $(S,S^{\prime})$ ranges over the same set as above. In particular 
\begin{gather*} 
\End_{\TL[\F]}(\pjw[v{-}1])\cong\spanFp{\loopdown{S}{v{-}1}|S \text{ down-admissible for } v}.
\end{gather*} 

\end{enumerate}
\end{lemma}

\begin{proof} 
For the first claim, we proceed by induction on the through-degree.
Note that the through-degree of
$\pqjw[w{-}1]\Up{S}\Down{S^{\prime}}\pqjw[v{-}1]$ is $w[S]=v[S^{\prime}]$. Let
$(S,S^{\prime})$ be the pair labeling the summand with maximal through-degree.
Then $r_{S,T}$ is $\ppar$-admissible since it is the coefficient of the (maximal
through-degree) basis element
$\up{S}\idtl[{v[S^{\prime}]{-}1}]\down{S^{\prime}}$ in \eqref{eq:admsum}. Thus,
we can subtract
$r_{S,S^{\prime}}\cdot\pqjw[w{-}1]\Up{S}\Down{S^{\prime}}\pqjw[v{-}1]$ to obtain
another $\ppar$-admissible sum, which now has strictly lower
through-degree since
$r_{S,S^{\prime}}\cdot\pqjw[w{-}1]\Up{S}\Down{S^{\prime}}\pqjw[v{-}1]$ was the
only summand with this maximal through-degree. If the resulting sum is non-zero, then the
remaining coefficients are now $\ppar$-admissible by the induction hypothesis.
The basis step for the induction concerns the morphism of minimal possible
through-degree, which is $\ppar$-admissible (and thus also its coefficient)
since there are no correction terms in \eqref{eq:morfilt}.

To see (b), for any given $\morstuff{F}\in\Hom_{\TL[\F]}(v-1,w-1)$, we choose a 
lift $\tilde{\morstuff{F}}\in\Hom_{\TL[\Z]}(v-1,w-1)\subset\Hom_{\TL[\Q]}(v-1,w-1)$. By \eqref{eq:corollary-pbasis}, the 
$\ppar$-admissible morphism $\pqjw[w{-}1]\tilde{\morstuff{F}}\,\pqjw[v{-}1]$ can be expanded in the $\ppar$-morphism 
basis over $\Q$. By (a), all appearing coefficients are $\ppar$-admissible 
and can be specialized to $\F$. This results in an expansion of 
$\pjw[w{-}1]\morstuff{F}\,\pjw[v{-}1]$ in terms of the $\ppar$-morphisms over $\F$. 
Note that all such morphisms are still linearly independent, since they have distinct through-degrees.
\end{proof}

\subsection{Morphisms between \texorpdfstring{$\ppar$}{p}JW projectors---the algebra structure}\label{subsec:relations}

\begin{lemma}\label{lemma:nil-ploop}
\begin{enumerate}[label=(\alph*)]

\setlength\itemsep{.15cm}

\item \label{lemma:nil-ploop-a} 
The algebra $\End_{\TL[\F]}(\pjw[v{-}1])$ is commutative.

\item \label{lemma:nil-ploop-b} Every $\loopdown{S}{v{-}1}$ is nilpotent.
As a consequence, every element of non-maximal 
through-degree in $\End_{\TL[\F]}(\pjw[v{-}1])$ is nilpotent.
\end{enumerate}
\end{lemma}

\begin{proof}
By \fullref{lemma:admsum}.(b), $\End_{\TL[\F]}(\pjw[v{-}1])$ has a basis that is invariant under reflection. 
Thus, for all $a,b\in\End_{\TL[\F]}(\pjw[v{-}1])$ we have $\flip[a]=a$ and $\flip[b]=b$, and then $ab=\flip[a]\flip[b]=\flip[(ba)]=ba$. 
This implies that $\End_{\TL[\F]}(\pjw[v{-}1])$ is commutative.

To see (b), we shall use induction on $\td{\loopdown{S}{v{-}1}}$. 
We work over $\Q$ and start by expanding 
$\loopdown{S}{v{-}1}\in\trap{S}{v{-}1}+\catstuff{td}_{>}$ into a sum of orthogonal quasi-idempotents and 
noting that $\trap{S}{v{-}1}$ has eigenvalue divisible by $\ppar$. If $S$ was maximal, 
then we have $(\loopdown{S}{v{-}1})^{2}=(\trap{S}{v{-}1})^{2}=0$ in $\End_{\TL[\F]}(\pjw[v{-}1])$. 
Otherwise, if $S\neq\emptyset$, we conclude $\td{(\loopdown{S}{v{-}1})^{2}}<\td{\loopdown{S}{v{-}1}}$. 
By \fullref{lemma:admsum}.(b), $(\loopdown{S}{v{-}1})^{2}$ is a linear combination of $\ppar$loops $\loopdown{S^{\prime}}{v{-}1}$ with 
$\td{\loopdown{S^{\prime}}{v{-}1}}<\td{\loopdown{S}{v{-}1}}$. Then the induction hypothesis implies that $(\loopdown{S}{v{-}1})^{2}$, and thus 
also $\loopdown{S}{v{-}1}$, is nilpotent.
\end{proof}

\begin{lemma}(Containment---\fullref{theorem:main-tl-section}.(2).)\label{lemma:nilpotent}
Let $S$ be a stretch that is down- or up-admissible for $v$ and $S^{\prime}\subset S$ 
down-admissible for $v[S]$ or up-admissible for $v(S)$ respectively. Then we have
\begin{gather*}
\Down{S^{\prime}}\Down{S}\pjw[v{-}1]=0,
\quad
\Up{S}\Up{S^{\prime}}\pjw[v{-}1]=0.
\end{gather*}
\end{lemma}

\begin{proof}
Note that by projector absorption, we 
have $\Down{S^{\prime}}\Down{S}\pjw[v{-}1]=\down{S^{\prime}}\down{S}\pjw[v{-}1]$. 
This is a cap configuration consisting of a pair of collections of 
concentric caps. The right one is not ancestor-centered and, thus, kills $\pjw[v{-}1]$ by \fullref{lemma:anccent}. 
\end{proof}

\begin{lemma}(Far-commutativity---\fullref{theorem:main-tl-section}.(3).)\label{lemma:far-commutation} 
Suppose that $S$ and $S^{\prime}$ are down-admissi\-ble, $T$ 
and $T^{\prime}$ up-admissible 
and $\dist(S,S^{\prime})>1$, $\dist(S,T)>1$, and $\dist(T,T^{\prime})>1$. 
The following hold.

\begin{gather*}
\Down{S}\Down{S^{\prime}}\pjw[v{-}1]=\Down{S^{\prime}}\Down{S}\pjw[v{-}1],
\quad
\Down{S}\Up{T}\pjw[v{-}1]=\Up{T}\Down{S}\pjw[v{-}1],
\quad
\Up{T}\Up{T^{\prime}}\pjw[v{-}1]=\Up{T^{\prime}}\Up{T}\pjw[v{-}1].
\end{gather*}
\end{lemma}

\begin{proof} 
These relations follow from projector absorption. 
For example, for the first relation we compute
\begin{gather*}
\Down{S}\Down{S^{\prime}}\pjw[v{-}1]
= 
\down{S}\pjw[w]\down{S^{\prime}}\pjw[v{-}1]
=
\down{S}\down{S^{\prime}}\pjw[v{-}1]
=
\down{S^{\prime}}\down{S}\pjw[v{-}1]
=
\down{S^{\prime}}\pjw[z]\down{S}\pjw[v{-}1]
= 
\Down{S^{\prime}}\Down{S}\pjw[v{-}1].
\end{gather*}
Here we have used an isotopy of caps in the third equality.
\end{proof}

\begin{lemma}(Adjacency relations $1$---\fullref{theorem:main-tl-section}.(4).)\label{lemma:adjacent-generators-1}
If $\dist(S,S^{\prime})=1$ and $S^{\prime}>S$, then
the following equations hold whenever one side, and thus also the other one, is admissible
\begin{gather*}
\Down{S^{\prime}}\Up{S}\pjw[v{-}1]=\Down{S}\Down{S^{\prime}}\pjw[v{-}1],
\quad
\Down{S}\Up{S^{\prime}}\pjw[v{-}1]=\Up{S^{\prime}}\Up{S}\pjw[v{-}1].
\end{gather*}
\end{lemma}

\begin{proof} 
The first relation follows from projector shortening and absorption, as can be best verified graphically, \ie
\begin{gather*}
\Down{S^{\prime}}\Up{S}\pjw[v{-}1]
=
\begin{tikzpicture}[anchorbase,scale=.25,tinynodes]
\draw[pJW] (2.5,0.5) rectangle (-2,-.5);
\draw[pJW] (2.5,2.5) rectangle (-2,1.5);
\draw[pJW] (2.5,4.5) rectangle (-2,3.5);
\node at (.25,-.2) {$v{-}1$};
\draw[usual] (-1.25,1.5) to[out=270,in=180] (-.75,1) to[out=0,in=270] (-.25,1.5);
\draw[usual] (.25,2.5) to[out=90,in=180] (.75,2.75) to[out=0,in=90] (1.25,2.5);
\draw[usual] (0,2.5) to[out=90,in=180] (.75,3) to[out=0,in=90] (1.5,2.5);
\draw[usual] (-.25,2.5) to[out=90,in=180] (.75,3.25) to[out=0,in=90] (1.75,2.5);
\draw[usual] (-1.75,.5) to (-1.75,1.5) 
(-1.75,2.5) to (-1.75,3.5)
(-1.25,2.5) to (-1.25,3.5) 
(0,.5) to (0,1.5)
(.25,.5) to (.25,1.5)
(1.25,.5) to (1.25,1.5)
(1.5,.5) to (1.5,1.5)
(1.75,.5) to (1.75,1.5)
(2.25,.5) to (2.25,1.5)
(2.25,2.5) to (2.25,3.5);
\node at (1,2.5) {$\phantom{a}$};
\node at (1,-.5) {$\phantom{a}$};
\end{tikzpicture}
=
\begin{tikzpicture}[anchorbase,scale=.25,tinynodes]
\draw[pJW] (2.5,0.5) rectangle (-2,-.5);
\draw[pJW] (2.5,2.5) rectangle (-.75,1.5);
\draw[pJW] (2.5,4.5) rectangle (-2,3.5);
\node at (.25,-.2) {$v{-}1$};
\draw[usual] (-1.25,3.5)to (-1.25,1.5) to[out=270,in=180] (-.75,1) to[out=0,in=270] (-.25,1.5);
\draw[usual] (.25,2.5) to[out=90,in=180] (.75,2.75) to[out=0,in=90] (1.25,2.5);
\draw[usual] (0,2.5) to[out=90,in=180] (.75,3) to[out=0,in=90] (1.5,2.5);
\draw[usual] (-.25,2.5) to[out=90,in=180] (.75,3.25) to[out=0,in=90] (1.75,2.5);
\draw[usual] (-1.75,.5) to (-1.75,3.5) 
(0,.5) to (0,1.5)
(.25,.5) to (.25,1.5)
(1.25,.5) to (1.25,1.5)
(1.5,.5) to (1.5,1.5)
(1.75,.5) to (1.75,1.5)
(2.25,.5) to (2.25,1.5)
(2.25,2.5) to (2.25,3.5);
\node at (1,2.5) {$\phantom{a}$};
\node at (1,-.5) {$\phantom{a}$};
\end{tikzpicture}
=
\begin{tikzpicture}[anchorbase,scale=.25,tinynodes]
\draw[pJW] (2.5,0.5) rectangle (-2,-.5);
\draw[pJW] (2.5,2.5) rectangle (.75,1.5);
\draw[pJW] (2.5,4.5) rectangle (-2,3.5);
\node at (.25,-.2) {$v{-}1$};
\draw[usual] (-1.25,3.5) to (-1.25,1.5) to[out=270,in=180] (-.75,1) to[out=0,in=270] (-.25,1.5) 
to (-.25,2.5) to[out=90,in=180] (.75,3.25) to[out=0,in=90] (1.75,2.5);
\draw[usual] (0,.5) to (0,2.5) to[out=90,in=180] (.75,3) to[out=0,in=90] (1.5,2.5);
\draw[usual] (.25,0.5) to (.25,2.5) to[out=90,in=180] (.75,2.75) to[out=0,in=90] (1.25,2.5);
\draw[usual] (-1.75,.5) to (-1.75,3.5) 
(1.25,.5) to (1.25,1.5)
(1.5,.5) to (1.5,1.5)
(1.75,.5) to (1.75,1.5)
(2.25,.5) to (2.25,1.5)
(2.25,2.5) to (2.25,3.5);
\node at (1,2.5) {$\phantom{a}$};
\node at (1,-.5) {$\phantom{a}$};
\end{tikzpicture}
=
\begin{tikzpicture}[anchorbase,scale=.25,tinynodes]
\draw[pJW] (2.5,0.5) rectangle (-2,-.5);
\draw[pJW] (2.5,4.5) rectangle (-2,3.5);
\node at (.25,-.2) {$v{-}1$};
\draw[usual] (-1.25,3.5)to[out=270,in=90] (1.75,2) to (1.75,0.5);
\draw[usual] (0,.5) to (0,1.5) to[out=90,in=180] (.75,2) to[out=0,in=90] (1.5,1.5) to (1.5,.5);
\draw[usual] (.25,.5) to (.25,1.25) to[out=90,in=180] (.75,1.75) to[out=0,in=90] (1.25,1.25) to (1.25,.5);
\draw[usual] (-1.75,.5) to (-1.75,3.5)
(2.25,.5) to (2.25,3.5);
\node at (1,2.5) {$\phantom{a}$};
\node at (1,-.5) {$\phantom{a}$};
\end{tikzpicture}
=
\Down{S\cup S^{\prime}}\pjw[v{-}1]=\Down{S}\Down{S^{\prime}}\pjw[v{-}1].
\end{gather*}
Here we have used projector shortening twice, 
then projector absorption and an isotopy.
The second relation is analogous.
\end{proof}

The following four statements will be proved jointly by induction in $v$. 
The proofs depend on each other in a non-trivial way.

\begin{lemma}(The endomorphisms.)\label{lemma:dualnumbers} 
Let $v\in \N$ with minimal 
down-admissible stretches $S_{j},\dots,S_{0}$. Then we have the algebra isomorphism
\begin{gather*}
\End_{\TL[\F]}(\pjw[v{-}1])
\cong
\F[]\big[\loopdown{S_{j}}{v{-}1},\dots,\loopdown{S_{0}}{v{-}1}\big]
\Big/\big\langle(\loopdown{S_{j}}{v{-}1})^{2},\dots, (\loopdown{S_{0}}{v{-}1})^{2}\big\rangle,
\end{gather*}
and if $S$ is down-admissible for $v$, then $\loopdown{S}{v{-}1} = \prod_{k|S_{k}{\subset}S}\loopdown{S_{k}}{v{-}1}$. 
Furthermore, if $S$ is down-admissible for $v$, then we have 
\begin{gather}\label{eq:DUD}
\Down{S}\Up{S}\Down{S}\pjw[v{-}1]=0, 
\quad
\pjw[v{-}1]\Up{S}\Down{S}\Up{S}=0.
\end{gather}
\end{lemma}

\begin{lemma}(Adjacency relations $2$---\fullref{theorem:main-tl-section}.(4).)\label{lemma:adjacent-generators-2}
Let $S^{\prime}>S$ be 
down-admissible stretches of 
consecutive integers for $v$ with $\dist(S,S^{\prime})=1$. Then we have
\begin{gather*}
\Down{S^{\prime}}\Down{S}\pjw[v{-}1]=\Up{S}\Down{S^{\prime}}\funcH[S]\pjw[v{-}1], 
\quad
\pjw[v{-}1]\Up{S}\Up{S^{\prime}}=\pjw[v{-}1]\funcH[S]\Up{S^{\prime}}\Down{S}.
\end{gather*}
\end{lemma}

\begin{lemma}(Overlap relations---\fullref{theorem:main-tl-section}.(5).)\label{lemma:overlap}
Suppose that $S$ is a minimal down-admissible stretch for $v$ and $S^{\prime}\geq S$ a 
minimal down-admissible stretch for $v[S]$ with $S^{\prime}\cap S=\{s\}$ and $S^{\prime}\not\subset S$, then we have 
\begin{gather*}
\Down{S^{\prime}}\Down{S}\pjw[v{-}1]=\Up{\{s\}}\Down{S}\Down{S^{\prime}{\setminus}\{s\}}\pjw[v{-}1],
\quad
\pjw[v{-}1]\Up{S}\Up{S^{\prime}}=\pjw[v{-}1]\Up{S^{\prime}{\setminus}\{s\}}\Up{S}\Down{\{s\}}.
\end{gather*}
\end{lemma}

\begin{lemma}(Zigzag---\fullref{theorem:main-tl-section}.(6).)\label{lemma:real-zigzag} 
Suppose that $S$ is an 
up-admissible stretch for $v$. If $S$ is also down-admissible for $v$, then we have
\begin{gather*}
\Down{S}\Up{S}\pjw[v{-}1]=\Up{S}\Down{S}\funcG[S]\pjw[v{-}1] 
+\Up{T}\Up{S}\Down{S}\Down{T}\funcF[S]\pjw[v{-}1].
\end{gather*} 
Here $T$ denotes the smallest minimal down-admissible 
stretch with $T>S$, provided it exists. If not, then the 
equation holds without the second term on the right-hand side.

Furthermore, if $S$ is not down-admissible for $v$, then we have
\begin{gather*}
\Down{S}\Up{S}\pjw[v{-}1]= 
-2 
\Up{\hull}\Down{\hull}\pjw[v{-}1].
\end{gather*} 
Here $\hull[S]$ denotes the down-admissible hull of $S$, if it exists. 
If not, then the right-hand side is defined to be zero. 
\end{lemma}

\section{Inductive proof of the relations}\label{sec:inductive-proof}

In this section we will use the 
far-commutativity relations from \fullref{lemma:far-commutation}, the containment 
relations from \fullref{lemma:nilpotent}, and the adjacency relations from 
\fullref{lemma:adjacent-generators-1}, sometimes without explicitly mentioning them. \makeautorefname{lemma}{Lemmas}
Further, we only prove \fullref{lemma:adjacent-generators-2} and \ref{lemma:overlap}, and \makeautorefname{lemma}{Lemma}
\eqref{eq:DUD} for the first shown relations as the other ones are equivalent by reflection.

\begin{convention}
Throughout this section, unless stated otherwise, we use the convention that $S$ denotes 
either a minimal down- or up-admissible stretch for $v$, and $U>T>\hull$ 
are the following minimal down-admissible stretches for $v$. 
To declutter the notation, we will suppress $\cup$ symbols in many expressions, 
for example $\Down{STU}:=\Down{S\cup T\cup U}$. 
Further, we 
introduce shorthand notation for the states where we have already proven 
the above Lemmas for certain $v\in N$.

\begin{enumerate}[label=$\bullet$]

\setlength\itemsep{.15cm}

\item $\condWm(v)$ means \fullref{lemma:real-zigzag} holds for all zigzags of 
the form $\Down{X}\pjw[w{-}1]\Up{X}$ where $w\leq v$ and $X$ is down-admissible 
for $w$, except possibly for the case $w=v$ and $X=S$, the smallest minimal down-admissible stretch for $v$.

\item $\condA(v)$ means \fullref{lemma:adjacent-generators-2} on adjacent generators holds for all $w\leq v$.

\item $\condO(v)$ means \fullref{lemma:overlap} on overlapping generators holds for all $w\leq v$.

\item $\condW(v)$ means \fullref{lemma:real-zigzag}, holds for all zigzags of the form $\Down{X}\pjw[w{-}1]\Up{X}$ where $w\leq v$.

\item $\condE(v)$ means \fullref{lemma:dualnumbers}, which describes $\End_{\TL[\F]}(\pjw[w{-}1])$, holds for all $w\leq v$.

\end{enumerate}

Here we would like to draw the readers attention to the fact that the 
relevant quantity for zigzags is not where they start, but how high they reach.
\end{convention}

The inductive proof of these conditions will proceed in the order shown. 
As base cases we observe that $\condA(v)$, $\condO(v)$, $\condE(v)$ and $\condW(v)$ are 
all vacuously satisfied for $1\leq v\leq\ppar$. Then, assuming that these conditions 
all hold for $v-1$, we will first deduce $\condWm(v)$, then $\condA(v)$ and 
$\condO(v)$, followed by $\condW(v)$, and finally $\condE(v)$.

\begin{lemma}\label{lemma:tandem-first}
$\condWm(v)$ follows if we have $\condW(v-1)$.
\end{lemma}

\begin{proof}
We need to show that we can resolve all zigzags of the form $\Down{Y}\Up{Y}\pjw[{v[Y]{-}1}]$ 
where $Y$ denotes a down-admissible stretch for $v$ such 
that $Y\neq S$, the smallest minimal down-admissible stretch for $v$. If $S\not\subset Y$, 
then this is possible using projector absorption and $\condW(v-1)$. 
In the remaining cases we write $Y_{+}:=Y\setminus S$ and employ the same 
trick, but for $Y_{+}\not\supset S$. If $Y_{+}$ is down-admissible for $v$, we get
\begin{align*}
\Down{Y}\Up{Y}\pjw[{v[Y]{-}1}]
&=\Down{S}\uwave{\Down{Y_{+}}\Up{Y_{+}}}\Up{S}\pjw[{v[Y]{-}1}]
\\
&=\uwave{\Down{S}\Up{Y_{+}}}\uwave{\Down{Y_{+}}\Up{S}}\funcG[Y]\pjw[{v[Y]{-}1}]+  
\uwave{\Down{S}\Up{T}\Up{Y_{+}}}\uwave{\Down{Y_{+}}\Down{T}\Up{S}}\funcF[Y]\pjw[{v[Y]{-}1}]
\\
&=\Up{Y_{+}}\Up{S}\Down{S}\Down{Y_{+}}\funcG[Y]\pjw[{v[Y]{-}1}]
+\Up{T}\Up{Y_{+}}\Up{S}
\Down{S}\Down{Y_{+}}\Down{T}\funcF[Y]\pjw[{v[Y]{-}1}]
\\
&=\Up{Y}\Down{Y} 
\funcG[Y]\pjw[{v[Y]{-}1}]+
\Up{T}\Up{Y}\Down{Y}\Down{T} 
\funcF[Y]\pjw[{v[Y]{-}1}].
\end{align*}
Here $T$ denotes the smallest minimal down-admissible stretch $T>Y$ for 
$v[Y]$, if it exists. We have also \uwave{underlined} the 
locations where relations are applied. If $Y_{+}$ is not down-admissible for $v$, then we instead get
\begin{align*}
\Down{Y}\Up{Y}\pjw[{v[Y]{-}1}]
&=\Down{S}\uwave{\Down{Y_{+}}\Up{Y_{+}}}\Up{S}\pjw[{v[Y]{-}1}]
=-2\uwave{\Down{S}\Up{\hull[Y_{+}]}}\uwave{\Down{\hull[Y_{+}]}\Up{S}}\pjw[{v[Y]{-}1}]
\\
&=-2\Up{\hull[Y_{+}]}\Up{S}\Down{S}\Down{\hull[Y_{+}]}\pjw[{v[Y]{-}1}]
=-2\Up{\hull[Y]}\Down{\hull[Y]}\pjw[{v[Y]{-}1}],
\end{align*}
or zero, if $\hull[{Y_{+}}]$ (and thus $\hull[Y]$) does not exist.
\end{proof}

\subsection{Adjacency relations}
Next we focus on establishing $\condA(v)$. These relations are irrelevant for
$\ppar=2$, so we will assume $\ppar>2$ in this subsection. For this we need an
approximate result first.

\begin{lemma}\label{lemma:adjacent-generators-3}
Suppose that $S<T<U$ are adjacent minimal down-admissible stretches for $v$. Then we have
\begin{gather*}
\Down{T}\Down{S}\pjw[v{-}1] 
\in\Up{S}\Down{T}\funcH[S]\pjw[v{-}1] 
+V_{>U}.
\end{gather*}
Here $V_{>U}=\spanFp{\Up{S}\Down{T}\pjw[v{-}1]\mid\exists t\in T\text{ such that }t>U}$ 
is the span of morphisms with $T$ exceeding $U$. 

Similarly, if the stretches are up-admissible for $v$, then we have
\begin{gather*}
\Up{S}\Up{T}\pjw[v{-}1]
\in\funcH[S]\Up{T}\Down{S}\pjw[v{-}1]+W_{>U}.
\end{gather*}
where $W_{>U}=\spanFp{\Up{S}\Down{T}\pjw[v{-}1]\mid \exists s\in S\text{ such that }s>U}$. 

In either case, if $U$ is a largest down-admissible stretch for $v$, 
or if no down-admissible stretch exists above $T$, then the 
relations from \fullref{lemma:adjacent-generators-2} hold on the nose.
\end{lemma}

\begin{proof} 
Let us write $\funch$ for the scalar appearing in
$\funch\pjw[v{-}1]=\funcH[S]\pjw[v{-}1]$. (The functions $\funcf$, $\funcg$, and
$\funch$ were defined in \fullref{subsec:main-TL}.) We would like to
identify
\begin{gather*}
\pjw[w{-}1]\Down{T}\Down{S}\pjw[v{-}1]
=
\begin{tikzpicture}[anchorbase,scale=.25,tinynodes]
\draw[pJW] (2.75,0.5) rectangle (-2.25,-.5);
\draw[pJW] (2.75,2.5) rectangle (-2.25,1.5);
\draw[pJW] (2.75,4.5) rectangle (-2.25,3.5);
\node at (.25,-.2) {$\pjwm[v{-}1]$};
\draw[usual] (-1.25,.5) to[out=90,in=180] (-.75,1) to[out=0,in=90] (-.25,.5);
\draw[usual] (0,.5) to (0,.9) to[out=90,in=270]  (-1.25,1.5); 
\draw[usual] (.25,2.5) to[out=90,in=180] (.75,3) to[out=0,in=90] (1.25,2.5);
\draw[usual] (1.5,2.5) to (1.5,2.9) to[out=90,in=270]  (.25,3.5); 
\draw[usual] (-2,.5) to (-2,1.5);
\draw[usual] (-1.25,2.5) to (-1.25,3.5);
\draw[usual] (.25,.5) to (.25,1.5);
\draw[usual] (-2,2.5) to (-2,3.5);
\draw[usual] (1.25,.5) to (1.25,1.5);
\draw[usual] (1.5,.5) to (1.5,1.5);
\draw[usual] (1.75,.5) to (1.75,1.5);
\draw[usual] (1.75,2.5) to (1.75,3.5);
\draw[usual] (2.5,.5) to (2.5,1.5);
\draw[usual] (2.5,2.5) to (2.5,3.5);
\node at (1,2.5) {$\phantom{a}$};
\node at (1,-.5) {$\phantom{a}$};
\end{tikzpicture}
\quad\text{and}\quad
\funch\cdot\;
\begin{tikzpicture}[anchorbase,scale=.25,tinynodes]
\draw[pJW] (2.5,0.5) rectangle (-2.25,-.5);
\draw[pJW] (2.5,2.5) rectangle (-2.25,1.5);
\draw[pJW] (2.5,4.5) rectangle (-2.25,3.5);
\node at (.25,-.2) {$\pjwm[v{-}1]$};
\draw[usual] (-1.25,3.5) to[out=270,in=180] (-.75,3) to[out=0,in=270] (-.25,3.5);
\draw[usual] (-1.25,2.5) to[out=90,in=270] (0,3.1) to (0,3.5); 
\draw[usual] (0,.5) to[out=90,in=180] (.5,1) to[out=0,in=90] (1,.5);
\draw[usual] (1.25,.5) to (1.25,.9) to[out=90,in=270]  (.25,1.5); 
\draw[usual] (-2,.5) to (-2,1.5);
\draw[usual] (-1.25,.5) to (-1.25,1.5);
\draw[usual] (.25,2.5) to (.25,3.5);
\draw[usual] (-2,2.5) to (-2,3.5);
\draw[usual] (1.5,.5) to (1.5,1.5);
\draw[usual] (1.5,2.5) to (1.5,3.5);
\draw[usual] (2.25,.5) to (2.25,1.5);
\draw[usual] (2.25,2.5) to (2.25,3.5);
\node at (1,2.5) {$\phantom{a}$};
\node at (1,-.5) {$\phantom{a}$};
\end{tikzpicture}
=
\funch\pjw[w{-}1]\Up{S}\Down{T}\pjw[v{-}1].
\end{gather*}
By projector absorption it suffices to do this in the case when $S$ is the smallest minimal down-admissible stretch of $v$.
We will start by computing the characteristic zero analogs of both sides. 

Suppose that $V\subset\N[0]$ is down-admissible for $v$, then by \fullref{lemma:capidem} we have 
\begin{gather*}
\begin{gathered}
\down{S}\lambda_{v,V}\,\trap{V}{v{-}1}=
\begin{cases}
\lambda_{v[S],V{\setminus}S}\,\upo{V{\setminus}S}\qjw[{v[V]{-}1}]\downo{V}
& \text{if }S\subset V,T\not\subset V, 
\\
\lambda_{v,V}\,\upo{VS}\qjw[{v[V]-1}]\downo{V}
& \text{if }S\not\subset V,T\subset V,
\\
0 & \text{otherwise}.
\end{cases}
\end{gathered}
\end{gather*}
After another application of \fullref{lemma:capidem} we get
\begin{gather}\label{eq:ddlhs}
\begin{gathered}
\down{T}\down{S}\lambda_{v,V}\,\trap{V}{v{-}1}=
\begin{cases}
\lambda_{v[S],V{\setminus}S}\,\upo{VT{\setminus}S}\qjw[{v[V]-1}]\downo{V}
& \text{if }S\subset V,T\not\subset V,U\subset V, 
\\
\lambda_{v,V}\,\lambda^{-1}_{v[S],VS}\,\lambda_{v[T](S),VS{\setminus}T}\, 
\upo{VS{\setminus}T} \qjw[{v[V]-1}]\downo{V}
& \text{if }S\not\subset V,T\subset V,U\not\subset V,
\\
0 & \text{otherwise}.
\end{cases}
\end{gathered}
\end{gather}
These possibilities for $V$ index the $\ppar$-morphism basis for $\Hom_{\TL[\F]}(\pjw[v{-}1],\pjw[w{-}1])$ from \fullref{lemma:admsum}. 
The term of maximal through-degree arises for $V=T$. 
Now, by the unitriangularity of the basis change between $\ppar$-morphisms and standard morphisms, 
we can read off the coefficient of $\Up{S}\Down{T}\pjw[v{-}1]$ 
in the $\ppar$-morphism expansion of $\Down{T}\Down{S}\pjw[v{-}1]$ 
as the coefficient of $\upo{S} \qjw[{v[T]-1}]\downo{T}$ in \eqref{eq:ddlhs}. 
Writing $S=\{i,i+1,\dots,i_{1}-1\}$ and 
$T=\{i_{1},\dots,i_{2}-1\}$, we compute this coefficient as
\begin{gather}\label{eq:q}
\begin{aligned}
q&= 
\tfrac{\lambda_{v,T}\,\lambda_{v[T](S),S}\,}
{\lambda^{-1}_{v[S],ST}\,}  
=
\tfrac{\lambda_{v,T}\,}
{\lambda_{v[S],T}\,}  
=(-1)^{\ppar^{i_1}}\tfrac{
\pbase{a_{j},\dots,a_{i_{2}},-a_{i_{2}-1},\dots,-a_{i_{1}+1},-a_{i_{1}},0,\dots,0}{\ppar}}
{\pbase{a_{j},\dots,a_{i_{2}},-a_{i_{2}-1},\dots,-a_{i_{1}+1},1-a_{i_{1}},0,\dots,0}{\ppar}}
\\
&=(-1)^{\ppar^{i_1}}\tfrac{\fancest{v}{S}[T]}{\fancest{v}{S}[T]+\ppar^{i_{1}}}.
\end{aligned}
\end{gather}
From this, we immediately get $q\equiv\funcg(a_{i_{1}}-1)=\funch\bmod\ppar$, as desired. 

To finish the proof, we also need to show that $\Up{UT}\Down{SU}\pjw[v{-}1]$ (the basis morphism of second 
highest through-degree 
in $\Hom_{\TL[\F]}(\pjw[v{-}1],\pjw[w{-}1])$) does not occur 
in the $\ppar$-morphism expansion of $\Down{T}\Down{S}\pjw[v{-}1]$. 
Thanks to triangularity of the basis change, this can be verified by computing the coefficient $q^{\prime}$ of 
$\upo{UT}\qjw[{v[SU]{-}1}]\downo{SU}$ in the difference 
$\pqjw[w{-}1]\down{T}\down{S}\pqjw[{v[T]-1}]-q\pqjw[w{-}1]\upo{S}\downo{T}\pqjw[v{-}1]$, 
and showing that it reduces to zero modulo $\ppar$.

To this end, we again use \fullref{lemma:capidem} to expand 
\begin{gather}\label{eq:ddrhs}
\begin{gathered}
\pqjw[w{-}1]\upo{S}\downo{T}\lambda_{v,V}\,\trap{V}{v{-}1}=
\begin{cases}
\lambda_{v,V}\,\lambda^{-1}_{v[T],VT}\,\lambda_{v[T](S),VT{\setminus}S}\,\upo{VT {\setminus}S}\qjw[{v[V]-1}]\downo{V}
& \text{if }S\subset V,T\not\subset V,U\subset V,
\\
\lambda_{v[T],V{\setminus}T}\,\upo{VS{\setminus}T}\qjw[{v[V]-1}]\downo{V}
& \text{if }S\not\subset V,T\subset V,U\not\subset V,
\\
0 & \text{otherwise}.
\end{cases}
\end{gathered}
\end{gather}

Focusing on the case $V=S\cup U$, we compute the crucial coefficient $q^{\prime}$ as
\begin{gather}\label{eq:discr}
\begin{aligned}
q^{\prime} = 
\lambda_{v[S],U}\,-  
\tfrac{\lambda_{v,T}\,}
{\lambda_{v[S],T}\,}  
\tfrac{\lambda_{v[S],SU}\,\lambda_{v[T](S),TU}\,}
{\lambda_{v[T],SUT}\,}  
&= 
\lambda_{v,U}\,
\big( 
1- 
\tfrac{\lambda_{v,T}\,\lambda_{v[T](S),TU}\, }
{\lambda_{v[S],T}\,\lambda_{v[T],TU}}
\big)
\\
&=
\lambda_{v,U}\,
\big( 
1 - 
\tfrac{\lambda_{v,T}\,\lambda_{v[SU](T),T}}
{\lambda_{v[S],T}\,\lambda_{v[U](T),T}}
\big),
\end{aligned}
\end{gather}
where we have used $\lambda_{v[S],U}\,=\lambda_{v,U}\,$, $\lambda_{v[S],SU}\,=\lambda_{v,U}\,\lambda_{v[U],S}\,$ 
and $\lambda_{v[T],SUT}\,=\lambda_{v[T],UT}\,\lambda_{v[U],S}\,$ in the first line, and in the second line
\begin{gather*}
\begin{aligned}
\lambda_{v[T](S),TU}
&=\lambda_{v[T](S),U}\,\lambda_{v[T](S)[U],T}=\lambda_{v[T],U}\,\lambda_{v[SU](T),T},
\\
\lambda_{v[T],TU}
&=\lambda_{v[T],U}\,\lambda_{v[T][U],T}=\lambda_{v[T],U}\,\lambda_{v[U](T),T}.
\end{aligned}
\end{gather*}
Now we note that
\begin{gather*}
\begin{aligned}
\tfrac{\lambda_{v[SU](T),T}}
{\lambda_{v[U](T),T}}
&=
\tfrac{\pbase{a_{j},\dots,a_{i_{3}},-a_{i_{3}-1},\dots,-a_{i_{2}},0,\dots,0,a_{i_{1}}-1,0,\dots,0}{\ppar}}
{\pbase{a_{j},\dots,a_{i_{3}},-a_{i_{3}-1},\dots,-a_{i_{2}},0,\dots,0,a_{i_{1}},0,\dots,0}{\ppar}} 
=\tfrac{\fancest{v[U]}{S}-\ppar^{i_{1}}}{\fancest{v[U]}{S}},
\end{aligned}
\end{gather*}
and together with \eqref{eq:q} we can continue
\begin{gather*} 
\begin{aligned}
\eqref{eq:discr} 
&= 
\lambda_{v,U}\,
\Big( 
\tfrac{
(\fancest{v}{S}[T]+\ppar^{i_{1}})\fancest{v[U]}{S}
-\fancest{v}{S}[T](\fancest{v[U]}{S}-\ppar^{i_{1}})}{(\fancest{v}{S}[T]+\ppar^{i_{1}} )\fancest{v[U]}{S}}
\Big)
\\
&= 
\lambda_{v,U}\,
\Big( 
\tfrac{\ppar^{i_{1}}(\fancest{v[U]}{S}+\fancest{v}{S}[T])}{(\fancest{v}{S}[T]+\ppar^{i_{1}})\fancest{v[U]}{S}}
\Big).
\end{aligned}
\end{gather*}
This is divisible by $
\tfrac{\ppar^{-|U|}\ppar^{|S|}\ppar^{|S|+|T|+|U|}}{\ppar^{|S|}\ppar^{|S|}}=\ppar^{|T|}$ and, thus, 
$q^{\prime}$ is zero modulo $\ppar$. This completes the proof of the first claim of the lemma. The second one is analogous.
\end{proof}

\begin{lemma}\label{lemma:tandem-second}
$\condA(v)$ follows if we have $\condA(v-1)$, $\condE(v-1)$ and $\condWm(v)$.  
\end{lemma}

The proof will be split into two parts. 
First we give a proof that works under a technical 
assumption, which is generically satisfied. 
In the second part, we treat the remaining cases.

\begin{proof}[Proof, with caveat.]
By $\condA(v-1)$ and projector absorption we may assume that 
$S$ is a smallest down-admissible stretch. At first, we will also 
assume that $S<T$ are minimal down-admissible stretches for $v$ and 
that $T$ is also down-admissible for $v[S]$. By \fullref{lemma:isotopies}, 
this implies that $S$ is up-admissible for $v$ and $T$ is down-admissible for $v(S)$. 

We already know that the desired equation holds up to certain potential error terms, \ie
\begin{gather}\label{eq:adjacenttwoerror}
\Down{T}\Down{S}\pjw[v{-}1] 
=\funch_{1}\Up{S}\Down{T}\pjw[v{-}1] 
+{\textstyle\sum_{X}}\,\big( 
c_{X}\Up{XUT}\Down{SUX} 
+d_{X}\Up{XS}\Down{TX}\big)
\pjw[v{-}1],
\end{gather}
where the summation runs over down-admissible subsets 
$X>U$, $c_{X},d_{X}\in\F$ and where we write $\funch_{1}:=\funcg(a_{\max(S)+1}-1)$ 
for $v=\pbase{a_{j},\dots,a_{0}}{\ppar}$. 
We now multiply this equation with $\Down{S}$ 
on the left and with $\Up{T}$ on the right 
and rewrite it using $w=v[T]$ and \fullref{lemma:adjacent-generators-1} into
\begin{gather}\label{eq:adjacenttwoerrortwo}
\begin{aligned}
\Down{ST}\Up{TS}\pjw[w{-}1]=
&\funch_{1}\Down{S}\Up{S}\Down{T}\Up{T}\pjw[w{-}1] 
\\
&
+{\textstyle\sum_{X}}\,\big(
c_{X}\loopdown{XUTS}{w{-}1} 
+d_{X}\loopdown{X}{w{-}1}\Down{S}\Up{S}\Down{T}\Up{T}\big)\pjw[w{-}1].
\end{aligned}
\end{gather}
This equation 
can be simplified using $\condWm(v)$. In this proof attempt, we only consider the 
\emph{generic case} where $T$ (and thus also $TS$) is down-admissible for $w$. 
So, using $\condWm(v)$ for the pair $(v,ST)$ we get:
\begin{gather*}
\Down{ST}\Up{TS}\pjw[{v[T]{-}1}] 
=\funcg_{1}\loopdown{TS}{v[T]{-}1} 
+\funcf_{1}\loopdown{UTS}{v[T]{-}1},
\end{gather*}
where $\funcg_{1}:=\funcg(b_{\max(S\cup T)+1})$ and 
$\funcf_{1}:=\funcf(b_{\max(S\cup T)+1})$ are computed 
from $v[T]=\pbase{b_{i},\dots,b_{0}}{\ppar}$. Further, 
using $\condWm(v)$ for $(v[T](S),S)$ and $(v,T)$ as well as $\condE(v-1)$ we compute
\begin{gather*}
\Down{S}\Up{S}\Down{T}\Up{T}\pjw[{v[T]{-}1}] 
= 
(\funcg_{2}\loopdown{S}{v[T]{-}1}+
\funcf_{2}\loopdown{TS}{v[T]{-}1})(\funcg_{1}\loopdown{T}{v[T]{-}1} 
+\funcf_{1}\loopdown{UT}{v[T]{-}1})
\\
=\funcg_{2}\funcg_{1}\loopdown{TS}{v[T]{-}1}
+\funcg_{2}\funcf_{1}\loopdown{UTS}{v[T]{-}1},
\end{gather*}
where $\funcg_{1}$ and $\funcf_{1}$ are as above, while 
$\funcf_{2}:=\funcf(b_{\max(S)+1})=\funcf(\ppar-a_{\max(S)+1})$. We also have
\begin{gather*}
\funcg_{2}:=\funcg(b_{\max(S)+1})=\funcg(\ppar-a_{\max(S)+1})
=\funch_{1}^{-1}.
\end{gather*} 
(Note that $\funch_{1}$ is invertible since $\ppar>2$.) Using these two computations and 
$\condE(v-1)$, the equation \eqref{eq:adjacenttwoerrortwo} transforms into
\begin{gather*}
0=0+{\textstyle\sum_{X}}\, 
\big( 
(c_{X}+\funcg_{2}\funcf_{1}d_{X})\loopdown{XUTS}{w{-}1} 
+\funcg_{2}\funcg_{1}d_{X}\loopdown{XTS}{w{-}1} 
\big).
\end{gather*}
Since the $\ppar$-loops $\loopdown{Y}{w{-}1}$ form a basis 
of $\End_{\TL[\F]}\left(\pjw[w{-}1]\right)$ and the scalars $\funcg_{1}$ and $\funcg_{2}$ are non-zero by admissibility and $\ppar>2$,
we conclude $d_{X}=0$ and then $c_{X}=0$. Thus all error terms in \eqref{eq:adjacenttwoerror} vanish. This completes the proof in the case where $S$ and $T$ are minimal.

In the general case, we partition $S$ and $S^{\prime}$ into minimal 
down-admissible stretches $S_{1}<\cdots<S_{k}$ and $S^{\prime}_{1}<\cdots<S^{\prime}_{l}$, 
respectively. Then we have
\begin{align*}
\Down{S^{\prime}}\Down{S}\pjw[v{-}1]
&=
\Down{S^{\prime}_{1}}\Down{S^{\prime}_{2}}\cdots\Down{S^{\prime}_{l}} 
\Down{S_{1}}\cdots\Down{S_{k-1}}\Down{S_{k}}\pjw[v{-}1]
\\
&=
\Down{S_{1}}\cdots\Down{S_{k-1}}\Down{S^{\prime}_{1}}\Down{S_{k}}\Down{S^{\prime}_{2}}\cdots\Down{S^{\prime}_{l}}\pjw[v{-}1]
\\
&=
\Down{S_{1}}\cdots\Down{S_{k-1}}\Up{S_{k}}\Down{S^{\prime}_{1}}\funcH[S_{k}]\Down{S^{\prime}_{2}}\cdots\Down{S^{\prime}_{l}}\pjw[v{-}1]
\\
&=
\Up{S_{k}}\cdots\Up{S_{1}}\Down{S^{\prime}_{1}}\cdots
\Down{S^{\prime}_{l}}\funcH[S_{k}]\pjw[v{-}1]
=\Up{S}\Down{S^{\prime}}\funcH[S]\pjw[v{-}1].
\end{align*}
Here we have first used far-commutativity, then $\condA(v-1)$ on the 
adjacent minimal stretches $S_{k}<S^{\prime}_{1}$, and finally \fullref{lemma:adjacent-generators-1}. 
Note also that $\funcH[S_{k}]=\funcH[S]$ far-commutes with $\Down{S^{\prime}_{2}}\cdots\Down{S^{\prime}_{l}}$.
\end{proof}

\begin{proof}[Proof of the remaining cases]
In the previous proof we made the assumption that $T$, and thus also $T\cup S$, is down-admissible for $w=v[T]$. 
Now suppose this is not the case. 
At first we can proceed in a very similar way as 
in the previous proof. Whenever we use zigzag 
relations, we have to replace $T$ by $\hull[T]=T\cup U$ and set the $\funcf$-term to zero. Hence, we get
\begin{gather*}
\Down{ST}\Up{TS}\pjw[w{-}1] 
=\funcg_{1}\loopdown{UTS}{w{-}1},
\\
\Down{S}\Up{S}\Down{T}\Up{T}\pjw[w{-}1] 
= 
(\funcg_{2}\loopdown{S}{w{-}1}+
\funcf_{2}\loopdown{UTS}{w{-}1})
(\funcg_{1}\loopdown{UT}{w{-}1})=\funcg_{2}\funcg_{1}\loopdown{UTS}{w{-}1},
\end{gather*}
and the equation \eqref{eq:adjacenttwoerrortwo} transforms into
\begin{gather*}
0=0+{\textstyle\sum_{X}}\, 
\big( 
(c_{X}+\funcg_{2}\funcg_{1}d_{X})\loopdown{XUTS}{w{-}1}\big).
\end{gather*}
This implies that the coefficients $c_{X}$ and $d_{X}$ are unit 
multiples of each other for every $X$. Next we will use a 
different strategy to show that $d_{X}=0$, which thus implies $c_{X}=0$ and finishes the proof. 
The strategy is to multiply both sides of \eqref{eq:adjacenttwoerror} by $\loopdown{U}{v{-}1}$ on 
the right, to equate the first two terms, to kill all terms with coefficients 
$c_{X}$, and to preserve all terms with coefficients $d_{X}$. 


The first two terms are rewritten as
\begin{align*}
\Down{T}\Down{S}\loopdown{U}{v{-}1}
&=\Down{T}\Up{U}\Down{S}\Down{U}\pjw[v{-}1]
=\Up{U}\Up{T}\Down{S}\Down{U}\pjw[v{-}1]  
\\
\funch_{1}\Up{S}\Down{T}\loopdown{U}{v{-}1}
&=\funch_{1}\Up{S}\Up{U}\Up{T}\Down{U}\pjw[v{-}1]=\funch_{1}\Up{U}\Up{S}\Up{T}\Down{U}\pjw[v{-}1],
\end{align*}
which are equal by virtue of $\condA(v-1)$ since $v[T](S)[U]<v$. 
We also note that the scalar that appears is exactly $\funch_{1}^{-1}$. 
After subtracting these terms from the multiple of \eqref{eq:adjacenttwoerror}, we are left with
\begin{gather}\label{eq:cxdx}
0={\textstyle\sum_{X}}\,\big( 
c_{X}\Up{XUT}\Down{SUX}\loopdown{U}{v{-}1}   
+d_{X}\Up{XS}\Down{TX}\loopdown{U}{v{-}1}\big).
\end{gather}
We first claim that $\Up{XUT}\Down{SUX}\loopdown{U}{v{-}1}=0$. 
To verify this, we distinguish between the two cases in which $X$ is distant or adjacent to $U$. 
In the first case, we get
\begin{gather*}
\Up{XUT}\Down{SUX}\loopdown{U}{v{-}1}  
=\Up{XUT}\Down{S}\Down{U}\Up{U}\Down{U}\Down{X}\pjw[v{-}1]=0,
\end{gather*}
since $\Down{U}\Up{U}\Down{U}\Down{X}\pjw[v{-}1]=0$ thanks to $\condE(v-1)$ as $v[X]<v$. 
In the second case, we get
\begin{gather*}
\Up{XUT}\Down{SUX}\loopdown{U}{v{-}1}  
=\Up{XUT}\Down{S}\Down{UX}\Up{U}\Down{U}\pjw[v{-}1] 
=\Up{XUT}\Down{S}\Down{X}\Up{U}\Up{U}\Down{U}\pjw[v{-}1]=0,
\end{gather*}
since $\Up{U}^2\pjw[v{-}1]=0$. This proves the claim.

Our second claim is that $\Up{XS}\Down{TX}\loopdown{U}{v{-}1}\neq 0$ 
for every $X$ and that these morphisms are linearly independent. 
Again it matters whether $X$ is distant or adjacent to $U$. In the first case we get
\begin{align*}
\Up{XS}\Down{TX}\loopdown{U}{v{-}1}  
=\Up{XS}\Down{T}\Up{U}\Down{U}\Down{X}\pjw[v{-}1]  
&=\Up{X}\Up{U}\Up{S}\Up{T}\Down{U}\Down{X}\pjw[v{-}1]
\\
&\sim\Up{X}\Up{U}\Up{T}\Down{S}\Down{U}\Down{X}\pjw[v{-}1].
\end{align*}
Here we have used $\condA(v-1)$ for 
$v[U\cup  X]<v$ to proceed to the second 
line. (We use $\sim$ to indicate unit proportionality.) In the second case we compute
\begin{alignat*}{2}
\Up{XS}\Down{TX}\loopdown{U}{v{-}1}  
&=\Up{XS}\Down{TU}\Down{X}\Down{U}\pjw[v{-}1] &&
\\  
&\sim\Up{XS}\Down{T}\Down{U}\Up{U}\Down{X}\pjw[v{-}1]&&
\\
&\sim\Up{XS}\Down{T}\Up{U}\Down{UX}\pjw[v{-}1]
&&=\Up{XS}\Up{UT}\Down{UX}\pjw[v{-}1]
\\
& &&\sim \Up{XUT}\Down{SUX}\pjw[v{-}1]
\end{alignat*}
This time we have used $\condA(v-1)$, namely on the ancestor $\fancest{v}{T}<v$ 
using projector absorption, to get to the second and the fourth line, 
and $\condWm(v)$ in the form of a zigzag relation for $v[X](U)<v$, noting 
that $U$ is down-admissible for $v[X]$, to get to the third line. The 
proportionality constants that appear in these steps are units and 
$\Up{XUT}\Down{SUX}\pjw[v{-}1]$ are linearly independent as $X$ varies. 

Finally, the two claims and equation \eqref{eq:cxdx} imply that $d_{X}=0$ 
for every $X$, and thus also $c_X=0$, which finishes the proof of $\condA(v)$. 
\end{proof}

\subsection{Overlap relations}
Next, we focus on establishing $\condO(v)$. We again start with an approximate version.

\begin{lemma}\label{lemma:overlapapprox}
Suppose that $S<T$ are adjacent minimal down-admissible stretches for $v$ and $S^{\prime}\geq S$ 
is a minimal down-admissible stretch for $v[S]$ with $S^{\prime}\cap S=\{s\}$ and $S^{\prime}\not\subset S$, then we have 
\begin{gather*}
\Down{S^{\prime}}\Down{S}\pjw[v{-}1]=\Up{\{s\}}\Down{S}\Down{S^{\prime}{\setminus}\{s\}}\pjw[v{-}1] 
+V_{>T},
\quad
\pjw[v{-}1]\Up{S}\Up{S^{\prime}}=\pjw[v{-}1]\Up{S^{\prime}{\setminus}\{s\}}\Up{S}\Down{\{s\}}+W_{>T}.
\end{gather*}
Here we use the notations $V_{>T}=\spanFp{\Up{X}\Down{Y}\pjw[v{-}1]\mid\exists y\in Y\text{ such that }y>T}$ and 
$W_{>T}=\spanFp{\pjw[v{-}1]\Up{Y}\Down{X}\mid\exists y\in Y\text{ such that }y>T}$.
In either case, if $T$ is a largest down-admissible stretch for $v$ 
then the relations from \fullref{lemma:overlap} hold on the nose.
\end{lemma}

\begin{proof}
We will use the notation $w=v[T](R)$ and 
$\{s\}=S\cap S^{\prime}$, $R=S\setminus \{s\}$, and 
note $S^{\prime}=T{ \cup}\{s\}$. We will also consider the 
minimal down-admissible stretch $U>T$ for $v$, if it exists. 
For the purpose of this proof it is useful to explicitly write 
down the relevant parts of the continued fraction expansions of $v$, $w$ and other entities
\begin{align*}
v 
&=\pbase{\dots,0,a_{x},0,\dots,a_{u},0,\dots,0,1,{\color{red}0},0,\dots,0,a_{r},\dots}{\ppar},
\\
v[S] 
&=\pbase{\dots,0,a_{x},0,\dots,a_{u},0,\dots,0,0,{\color{red}\ppar-1},\ppar-1,\dots,\ppar-1,\ppar-a_{r},\dots}{\ppar},
\\
v[T] 
&=\pbase{\dots,0,a_{x},0,\dots,a_{u}-1,\ppar-1,\dots,\ppar-1,\ppar-1,{\color{red}0},0,\dots,0,a_{r},\dots}{\ppar} 
=w[R],
\\
v[T][S] 
&=\pbase{\dots,0,a_{x},0,\dots,a_{u}-1,\ppar-1,\dots,\ppar-1,\ppar-2,{\color{red}\ppar-1},
\ppar-1,\dots,\ppar-1,\ppar-a_{r},\dots}{\ppar},
\\
v[S][S^{\prime}]
&=\pbase{\dots,0,a_{x},0,\dots,a_{u}-1,\ppar-1,\dots,
\ppar-1,\ppar-1,{\color{red}1},\ppar-1,\dots,\ppar-1,\ppar-a_{r},\dots}{\ppar}=w.
\end{align*}
Here we have highlighted the digit in position $s$ in red.

From this description, it is straightforward to see 
that $\Hom_{\TL[\F]}(\pjw[v{-}1],\pjw[w{-}1])$ is spanned by morphisms of the following four different types
\begin{gather*}
\Up{XR}\Down{TX}\pjw[v{-}1],
\quad
\Up{XUS^{\prime}}\Down{SUX}\pjw[v{-}1],
\quad
\Up{X\{s\}}\Down{STX}\pjw[v{-}1],
\quad
\Up{XUTR}\Down{UX}\pjw[v{-}1],
\end{gather*}
where $X$ denotes a down-admissible subset for $v$ with $X>U$, 
which may be empty. The basis elements of highest and second highest 
through-degree among the above are $\Up{R}\Down{T}\pjw[v{-}1]$ and 
$\Up{\{s\}}\Down{ST}\pjw[v{-}1]$, and all other basis elements are in the subspace $V_{>T}$. 

Our task is to show that $\Up{R}\Down{T}\pjw[v{-}1]$ 
appears with coefficient $0$ and $\Up{\{s\}}\Down{ST}\pjw[v{-}1]$ 
appears with coefficient $1$ if we expand $\Down{S^{\prime}}\Down{S}\pjw[v{-}1]$ 
in this basis. We again start with a computation in characteristic zero.

Two applications of \fullref{lemma:capidem} establish
\begin{gather}\label{eq:ddlhsagain}
\begin{gathered}
\down{S^{\prime}}\down{S}\lambda_{v,V}\,\trap{V}{v{-}1}=
\\
\begin{cases}
\lambda_{v[S],V{\setminus}S}\,\upo{VS^{\prime}{\setminus}S}\qjw[{v[V]-1}]\downo{V}
& \text{if }S\subset V,T\not\subset V,U\subset V, 
\\
\lambda_{v,V}\,\lambda^{-1}_{v[S],VS}\,\lambda_{w,VR{\setminus}T}\, 
\upo{VR{\setminus}T} \qjw[{v[V]-1}]\downo{V}
& \text{if }S\not\subset V,T\subset V,U \not\subset V,
\\
0 & \text{otherwise}.
\end{cases}
\end{gathered}
\end{gather}
where we have used $V\cup S\setminus S^{\prime}=V\cup R\setminus T$ in the second case.

The coefficient $q$ of the maximal through-degree 
basis element $\Up{R}\Down{T}\pjw[v{-}1]$ in 
$\Down{S^{\prime}}\Down{S}\pjw[v{-}1]$ is equal to the
coefficient shown for $\upo{R} \qjw[{v[T]-1}]\downo{T}$ in \eqref{eq:ddlhsagain}. This is
\begin{gather*}
q=\lambda_{v,T}\,\lambda^{-1}_{v[S],TS}\,\lambda_{w,R} \equiv
(-1)^{\ppar^{s}}\frac{
\pbase{\dots,0,a_{x},0,\dots,a_{u}-1,\ppar-1,\dots,\ppar-1,\ppar-1,{\color{red}0}}{\ppar}
}{
\pbase{\dots,0,a_{x},0,\dots,a_{u}-1,\ppar-1,\dots,\ppar-1,\ppar-1,{\color{red}1}}{\ppar}
}
\equiv 0 \bmod \ppar.
\end{gather*}
This shows that $\Up{R}\Down{T}\pjw[v{-}1]$ appears with coefficient $0$ in 
$\Down{S^{\prime}}\Down{S}\pjw[v{-}1]$. The term $\Up{\{s\}}\Down{ST}\pjw[v{-}1]$, however, does 
not seem to appear at all in \eqref{eq:ddlhsagain}. 
Since it is of second highest through-degree in $\Hom_{\TL[\F]}(\pjw[v{-}1],\pjw[w{-}1])$, 
its coefficient is congruent to the coefficient of 
$\up{\{s\}}\pqjw[{v[ST]-1}]\down{ST}$ in the $\ppar$-morphism expansion of 
$q\upo{R}\qjw[{v[T]-1}]\downo{T}$.

Using \fullref{lemma:capidem}, it is straightforward to 
compute that $\upo{R} \qjw[{v[T]-1}]\downo{T}$ equals $\up{R}\pqjw[{v[T]-1}]\down{T}-\lambda_{w,\{s\}}\up{\{s\}}\pqjw[{v[ST]-1}]\down{ST}$ 
up to terms of lower through-degree. 
Thus, we compute the coefficient of interest as
\begin{gather*}
-\lambda_{w,\{s\}} q \equiv
-\frac{
\pbase{\dots,a_{u}-1,\ppar-1,\dots,\ppar-1,{\color{red}-1}}{\ppar}
}{
\pbase{\dots,a_{u}-1,\ppar-1,\dots,\ppar-1,{\color{red}0}}{\ppar}
}
\frac{
\pbase{\dots,a_{u}-1,\ppar-1,\dots,\ppar-1,{\color{red}0}}{\ppar}
}{
\pbase{\dots,a_{u}-1,\ppar-1,\dots,\ppar-1,{\color{red}1}}{\ppar}
}
\equiv
1 \bmod \ppar,
\end{gather*}
and this finishes the proof.
\end{proof}

\begin{lemma}\label{lemma:tandem-third}
$\condO(v)$ follows if we have $\condE(v-1)$, $\condA(v-1)$, $\condO(v-1)$, and $\condWm(v)$.
\end{lemma}

\begin{proof}[Proof, with caveat.]
As usual, $\condO(v-1)$ and projector absorption 
allows us to restrict to the case when $S$ is the 
smallest minimal down-admissible stretch for $v$.
By \fullref{lemma:overlapapprox} we then have
\begin{gather}\label{eq:overlapstart}
\begin{aligned}
\Down{S^{\prime}}\Down{S}\pjw[v{-}1] 
&=\Up{\{s\}}\Down{S}\Down{T}\pjw[v{-}1]
\\
&+{\textstyle\sum_{X\neq \emptyset}}\,c_{X}\Up{XR}\Down{TX}\pjw[v{-}1] 
&+{\textstyle\sum_{X}}\,d_{X}\Up{XUS^{\prime}}\Down{SUX}\pjw[v{-}1]
\\
&+{\textstyle\sum_{X\neq \emptyset}}\,e_{X}\Up{X\{s\}}\Down{STX}\pjw[v{-}1] 
&+{\textstyle\sum_{X}}\,f_{X}\Up{XUTR}\Down{UX}\pjw[v{-}1].
\end{aligned}
\end{gather}
Here $U>T$ denotes another adjacent minimal down-admissible stretch for $v$, 
if it exists, and $X$ ranges over down-admissible subsets $X>U$ for $v$. 
Our task is to show that the scalars $c_{X},d_{X},e_{X},f_{X}\in\F$ are all zero. 

We start by multiplying both sides of \eqref{eq:overlapstart} by $\Up{T}$ on the right. After rearranging, we get
\begin{gather}\label{eq:overlapsecond}
\begin{aligned}
\Down{S^{\prime}}\Up{ST}\pjw[{v[T]{-}1}] 
&=\Up{\{s\}}\Down{S}\Down{T}\Up{T}\pjw[{v[T]{-}1}]
\\
&+{\textstyle\sum_{X\neq \emptyset}}\,c_{X}\Up{XR}\Down{X}\Down{T}\Up{T}\pjw[{v[T]{-}1}]
\\
&+{\textstyle\sum_{X}}\,d_{X}\Up{XUS^{\prime}}\Down{STUX}\pjw[{v[T]{-}1}]
\\
&+{\textstyle\sum_{X\neq \emptyset}}\,e_{X}\Up{X\{s\}}\Down{SX}\Down{T}\Up{T}\pjw[{v[T]{-}1}]
\\
&+{\textstyle\sum_{X}}\,f_{X}\Up{XUTR}\Down{TUX}\pjw[{v[T]{-}1}].
\end{aligned}
\end{gather}
The next step is to apply the zigzag relations and for this 
we shall assume that we are in the \emph{generic case}, where $T$ 
is down-admissible for $v[T]$ (and thus $\ppar>2$). This also implies that $S^{\prime}$ is 
down-admissible for $v[T](R)$, and using the zigzag relations provided by $\condWm(v)$ for $v[T](R)$ we compute
\begin{gather*}
\begin{aligned}
\Down{S^{\prime}}\Up{ST}\pjw[{v[T]{-}1}]
&=\Down{S^{\prime}}\Up{S^{\prime}}\Up{R}\pjw[{v[T]{-}1}]
\\
&=\funcg(a_u)\Up{S^{\prime}}\Down{S^{\prime}}\Up{R}\pjw[{v[T]{-}1}]
\\
&\;\;+\funcf(a_u)\Up{US^{\prime}}\Down{S^{\prime}U}\Up{R}\pjw[{v[T]{-}1}]
\\
&=\funcg(a_u)\Up{S^{\prime}}\Down{ST}\pjw[{v[T]{-}1}]
\\
&\;\;+\funcf(a_u)\Up{US^{\prime}}\Down{STU}\pjw[{v[T]{-}1}].
\end{aligned}
\end{gather*}
Similarly we compute
\begin{gather}\label{eq:overlapright}
\begin{aligned}
\Up{\{s\}}\Down{S}\Down{T}\Up{T}\pjw[{v[T]{-}1}]
&=\funcg\Up{\{s\}}\Down{S}\Up{T}\Down{T}\pjw[{v[T]{-}1}]
+\funcf\Up{\{s\}}\Down{S}\Up{UT}\Down{TU}\pjw[{v[T]{-}1}]
\\
&=\funcg\Up{\{s\}}\Up{TS}\Down{T}\pjw[{v[T]{-}1}]
+\funcf\Up{\{s\}}\Up{UTS}\Down{TU}\pjw[{v[T]{-}1}]
\\
&=\funcg(\ppar-2)\funcg \Up{T}\Down{\{s\}}\Up{\{s\}}\Up{R}\Down{T}\pjw[{v[T]{-}1}]
\\
&\;\;+\funcg(\ppar-2)\funcf \Up{UT}\Down{\{s\}}\Up{\{s\}}\Up{R}\Down{TU}\pjw[{v[T]{-}1}]
\\
&=\funcg\Up{S^{\prime}}\Down{\{s\}}\Up{R}\Down{T}\pjw[{v[T]{-}1}]
+\funcf\Up{US^{\prime}}\Down{\{s\}}\Up{R}\Down{TU}\pjw[{v[T]{-}1}]
\\
&=\funcg\Up{S^{\prime}}\Down{ST}\pjw[{v[T]{-}1}]
+\funcf\Up{US^{\prime}}\Down{STU}\pjw[{v[T]{-}1}],
\end{aligned}
\end{gather}
where we write $\funcg=\funcg(a_u-1)$ and $\funcf=\funcf(a_u-1)$, and we 
have used $\condWm(v)$ on the pair $(v,T)$ and smaller instances, as well as 
$\condA(v-1)$. Thus, we have equated the first two terms in \eqref{eq:overlapsecond}. We also simplify the $c_{X}$ terms
\begin{align*}
\Up{XR}\Down{X}\Down{T}\Up{T}\pjw[{v[T]{-}1}] 
&=\funcg\Up{XR}\Down{X}\Up{T}\Down{T}\pjw[{v[T]{-}1}]+ 
\funcf\Up{XR}\Down{X}\Up{UT}\Down{TU}\pjw[{v[T]{-}1}]
\\
&=\funcg\Up{XTR}\Down{TX}\pjw[{v[T]{-}1}]+ 
\funcf\Up{XUTR}\Down{TUX}\pjw[{v[T]{-}1}],
\end{align*}
where we have again used $\condWm(v)$ on the pair $(v,T)$, then $\condA(v-1)$, 
and smaller instances of zigzag relations in the case 
when $X\neq \emptyset$ is adjacent to $U$ for the final step. To be explicit, the sequence of transformations is
\begin{gather*}
\Down{X}\Up{UT}\Down{TU}\pjw[{v[T]{-}1}]=
\Down{TU}\Down{X}\Down{TU}\pjw[{v[T]{-}1}]\sim
\\
\Down{TU}\Up{UT}\Down{X}\pjw[{v[T]{-}1}]\sim
\Up{UT}\Down{TUX}\pjw[{v[T]{-}1}].
\end{gather*}
The simplification of the $f_{X}$ term proceeds in complete analogy to \eqref{eq:overlapright} and we get
\begin{gather*}
\Up{X\{s\}}\Down{SX}\Down{T}\Up{T}\pjw[{v[T]{-}1}]= 
\funcg\Up{XS^{\prime}}\Down{STX}\pjw[{v[T]{-}1}]
+\funcf\Up{XUS^{\prime}}\Down{STUX}\pjw[{v[T]{-}1}],
\end{gather*}
having again used only $\condWm(v)$ and $\condA(v-1)$.

Finally, after all these simplifications, \eqref{eq:overlapsecond} gives the following linear system
\begin{gather*}
0=\funcg c_{X},\quad
0=\funcf c_{X}+ f_{X},\quad
0=d_{X}+\funcf e_{X},\quad
0=\funcg e_{X},
\end{gather*}
which, since $\funcg\neq 0$, implies that all unwanted scalars are zero. 
\end{proof}

\begin{proof}[Proof of the remaining cases]
Now suppose that $T$ is not down-admissible for $v[T]$, 
which happens exactly if $a_u=1$ in the notation from above. 
In this case we have $\hull[T]=T\cup U$ for $v[T]$ and 
$\hull[S^{\prime}]=S^{\prime}\cup U$ for $v[T](R)$. 

We proceed exactly as above, with the only differences being that 
no $\funcg$ terms arise and $\funcf=-2$.
The linear system resulting from \eqref{eq:overlapsecond} is
\begin{align*}
0=-2c_{X}+f_{X},
\quad
0=d_{X}-2e_{X}.
\end{align*}
(Note that if $\ppar=2$, we immediately see $d_{X}=0=f_{X}$.) 
To see that all coefficients are zero, 
we multiply \eqref{eq:overlapstart} by $\loopdown{U}{v{-}1}$, 
expecting that this should allow us to equate the first two 
terms, kill the $d_{X}$ and $f_{X}$ terms, and not hurt the $c_{X}$ 
and $e_{X}$ terms. Let us check these assertions in turn.

For the first term we get
\begin{gather*}
\Down{S^{\prime}}\Down{S}\Up{U}\Down{U}\pjw[v{-}1]=
\Down{S^{\prime}}\Up{U}\Down{S}\Down{U}\pjw[v{-}1]=
\Up{US^{\prime}}\Down{SU}\pjw[v{-}1].
\end{gather*}

For the second term we compute
\begin{gather*}
\Up{\{s\}}\Down{S}\Down{T}\Up{U}\Down{U}\pjw[v{-}1]=
\Up{\{s\}}\Up{UTS}\Down{U}\pjw[v{-}1]= 
\Up{US^{\prime}}\Down{SU}\pjw[v{-}1],
\end{gather*}
where the second step works as in \eqref{eq:overlapright} and 
requires $\condA(v-1)$ and $\condWm(v)$. This equates the first two terms. 

Now we claim that the $d_{X}$ and $f_{X}$ terms are killed by the loop along $U$
\begin{gather*}
\Up{XUS^{\prime}}\Down{SUX}\Up{U}\Down{U}\pjw[v{-}1]=0
=\Up{XUTR}\Down{UX}\Up{U}\Down{U}\pjw[v{-}1].
\end{gather*}

If $X\neq \emptyset$ is adjacent to $U$, then both assertions follow from
\begin{gather*}
\Down{UX}\Up{U}\Down{U}\pjw[v{-}1]=
(\Down{U})^2\Down{X}\Down{U}\pjw[v{-}1]=0.
\end{gather*}
If $X$ is distant from $U$ or empty, then we use far-commutativity to see substrings 
of the form $\Down{U}\Up{U}\Down{U}\Down{X}\pjw[v{-}1]=0$ by $\condE(v-1)$.

Now we claim that the $c_{X}$ and $e_{X}$ terms survive the multiplication by the loop along $U$: 
\begin{gather*}
\Up{XR}\Down{TX}\Up{U}\Down{U}\pjw[v{-}1]\neq 0
\neq\Up{X\{s\}}\Down{STX}\Up{U}\Down{U}\pjw[v{-}1].
\end{gather*}
To see this, let us first observe $\Down{X}\Up{U}\Down{U}\pjw[v{-}1]=\Up{U}\Down{U}\Down{X}\pjw[v{-}1]$. 
This is clear if $X$ is distant from $U$, and it follows 
from $\condA(v-1)$ and $\condWm(v)$, otherwise. Using this observation, we compute
\begin{align*}
\Up{XR}\Down{TX}\Up{U}\Down{U}\pjw[v{-}1]=
\Up{XR}\Down{T}\Up{U}\Down{UX}\pjw[v{-}1]
&=
\Up{XUTR}\Down{UX}\pjw[v{-}1]\neq 0
\\
\Up{X\{s\}}\Down{STX}\Up{U}\Down{U}=
\Up{X\{s\}}\Down{ST}\Up{U}\Down{UX}
&=
\Up{X\{s\}}\Up{UTS}\Down{UX}
\\
&=\Up{XUS^{\prime}}\Down{SUX}\neq 0
\end{align*}
where the last step works as in 
\eqref{eq:overlapright} and requires $\condA(v-1)$ and $\condWm(v)$.

After these simplifications, we see that 
\eqref{eq:overlapstart} multiplied by $\loopdown{U}{v{-}1}$ 
shows $c_{X}=0=e_{X}$, which (for $\ppar>2$) in turn implies $d_{X}=0=f_{X}$. This completes the proof of $\condO(v)$. 
\end{proof}

Let us also note the following consequence.

\begin{lemma}\label{lemma:overlapcorollary} 
Suppose that a minimal stretch $S$ is down-admissible for $v$ but not for $v[S]$, 
and suppose the down-admissible hull 
$\hull$ exists. 
Then $\condO(v)$ implies
\begin{gather}\label{eq:overlapcor}
\Down{\hull}\Down{S}\pjw[v{-}1]=\Up{S}\Down{\hull}\pjw[v{-}1],
\quad 
\pjw[v{-}1]\Up{S}\Up{\hull}=\pjw[v{-}1]\Up{\hull}\Down{S}.
\end{gather}
\end{lemma}

\begin{proof}
Let $s=\max(S)$ and $S^{\prime}=\{s\} \cup \hull\setminus S$ and $R=S\setminus \{s\}$. Then $S^{\prime}$ is down-admissible for $v[S]$ and we use $\condO(v)$ to compute
\begin{gather*}
\Down{\hull}\Down{S}\pjw[v{-}1]
=
\Down{R}\Down{S^{\prime}}\Down{S}\pjw[v{-}1]
=
\Down{R}\Up{\{s\}}\Down{S}\Down{S^{\prime}\setminus\{s\}}\pjw[v{-}1]
=
\Up{\{s\}}\Up{R}\Down{\hull}\pjw[v{-}1]
=
\Up{S}\Down{\hull}\pjw[v{-}1].
\end{gather*}
The other relation follows by reflection.
\end{proof}

\subsection{Zigzag relations}

\begin{lemma}\label{lem:gen2zigzag} 
The zigzag relations from \fullref{lemma:real-zigzag} hold in generation $2$.
\end{lemma}

\begin{proof} 
Suppose that $S^{\prime}$ is a down-admissible stretch 
for $v$ such that $w=v[S^{\prime}]$ is of generation $2$. Then, 
using the projector shortening property from \fullref{proposition:absorb}, we get 
\begin{gather*}
\Down{S^{\prime}}\Up{S^{\prime}}\pjw[w{-}1]= 
\pTr_{(v{-}w)/2}(\pjw[{(v+w)/2{-}1}]).
\end{gather*}
This partial trace is not covered by \fullref{proposition:p-properties-trace}, 
but since $(v+w)/2$ is of generation at most $2$, it can be straightforwardly computed: 
One first expands $\pqjw[{(v+w)/2{-}1}]$ into a linear combination of standard loops and 
computes their partial traces using \eqref{eq:0trace}. The result follows by changing 
back into the $\ppar$loop basis of $\End_{\TL[\Q]}(\pqjw[w{-}1])$ and reducing the coefficients to $\F$. 

The basis change from $\ppar$loops to standard loops for $w$ of generation $2$ with minimal down-admissible stretches $S<T$ is
\begin{gather*}
\begin{aligned}
\loopdown{\emptyset}{w{-}1}
&=\trap{\emptyset}{w{-}1}
+(-1)^{w-\mother[w]}\tfrac{w[S]}{\mother[w]}\cdot\trap{S}{w{-}1}
+(-1)^{\mother[w]-\motherr{w}{2}}\tfrac{\mother[w][T]}{\motherr{w}{2}}\cdot\trap{T}{w{-}1}
+(-1)^{w-\motherr{w}{2}}\tfrac{w[ST]}{\motherr{w}{2}}\cdot\trap{ST}{w{-}1},
\\ 
\loopdown{S}{w{-}1}
&=\trap{S}{w{-}1}
+(-1)^{w-\motherr{w}{2}}\tfrac{(\mother[w][T])^{2}}{w[T]\motherr{w}{2}}\cdot\trap{T}{w{-}1}, 
\\ 
\loopdown{T}{w{-}1}&=\trap{T}{w{-}1}
+(-1)^{w-\mother[w]}\tfrac{w[ST]}{\mother[w][T]}\cdot\trap{ST}{w{-}1},
\\
\loopdown{ST}{w{-}1}&=\trap{ST}{w{-}1}.
\end{aligned}
\end{gather*}
The inverse basis change can be readily computed from this. The basis change in generation $1$ is easier and left as an exercise for the reader.
\end{proof}

\begin{lemma}\label{lemma:tandem-fourth}
$\condW(v)$ follows if we have $\condWm(v)$, $\condE(v-1)$, $\condA(v)$ and $\condO(v)$.
\end{lemma}

The proof again splits into two parts. First we give a 
proof that works under a technical assumption, which is generically 
satisfied. In the second part, we refine this proof to work in all cases.

\begin{proof}[Proof, with caveat.] 
We need to consider the zigzag  $\Down{S}\pjw[v{-}1]\Up{S}$ where $S$ is the smallest minimal down-admissible stretch of $v$.
Let us also assume that we are in the \emph{generic case}, where $S$ 
is also down-admissible for $v[S]$ (and thus $\ppar>2$), and we denote by $T$ the minimal down-admissible stretch for $v[S]$ that is adjacent and $T>S$. 

By the unitriangularity of the basis change between the $\ppar$loops basis and the standard 
loops basis for $\End_{\TL[\Q]}(\pqjw[{v[S]{-}1}])$ and by the generation $2$ case in 
\fullref{lem:gen2zigzag}, we may assume that
\begin{gather}\label{eq:zigzag-approx}
\Down{S}\Up{S}\pjw[{v[S]{-}1}]
=\funcG[S]\loopdown{S}{{v[S]{-}1}}+
\funcF[S]\loopdown{ST}{{v[S]{-}1}}+
{\textstyle\sum_{U\not\subset {S\cup T}}}\,x_{U}\loopdown{U}{{v[S]{-}1}},
\end{gather}
with error terms $x_{U}\loopdown{U}{{v[S]{-}1}}$ with $x_{U}\in\F$. 
Our job is to show that we have $x_{U}=0$ for all 
such $U$. If we multiply \eqref{eq:zigzag-approx} 
by $\loopdown{S}{{v[S]{-}1}}$, then $\condE(v-1)$ implies $0=0+0+\sum_{X}x_{X}\loopdown{SX}{{v[S]{-}1}}$ 
and thus $x_{X}=0$, where $X$ runs over all $U$ as above, for which $S\not\subset U$. 
On the other hand, if we multiply \eqref{eq:zigzag-approx} by $\loopdown{T}{{v[T]{-}1}}$, then we get
\begin{gather*}
\Down{S}\Up{S}\Up{T}\Down{T}\pjw[{v[S]{-}1}]  
=\funcG[S]\loopdown{ST}{v{-}1}\pjw[{v[S]{-}1}]  
+{\textstyle\sum_{X}}\,x_{X}\loopdown{TX}{{v[S]{-}1}},
\end{gather*}
where now $X$ runs over all remaining $U$ such that $T\not\subset X$. Then, by $\condA(v)$, we also get
\begin{gather*}
\Down{S}\Up{S}\Up{T}\Down{T}\pjw[{v[S]{-}1}]  
=  
\funcG[S]\loopdown{ST}{{v[S]{-}1}}.
\end{gather*} 
This implies $x_{X}=0$ for such $X$. 
The only coefficients 
$x_{U}$ that are left to be considered are the ones for 
which $S\cup T\subset U$. Now we apply the partial 
trace $\mathrm{pTr}_{ST}:=\mathrm{pTr}_{(v[S]-\fancest{v[S]}{ST})}$ to both 
sides of \eqref{eq:zigzag-approx}. For this we 
will use the notation $w=\fancest{v[S]}{S}$ and $u=\fancest{v[S]}{ST}$, and we get
\begin{gather*}
(-1)^{w+1-u}2\pjw[u{-}1] 
= 
\funcG[S](-1)^{w-u}2
\pjw[u{-}1]+ 
\funcF[S]\pjw[u{-}1] 
+{\textstyle\sum_{U\neq \emptyset}}x_{U}\loopdown{U}{u{-}1}.
\end{gather*}
This is because $\mathrm{pTr}_{ST}(\Down{S}\Up{S}\pjw[v{-}1])=\mathrm{pTr}_{T}(\pjw[z{-}1])$
where $z=\fancest{v}{S}$, which differs from $w$ by 
increasing its first non-zero digit $a$ by one. 
The coefficient of $\pjw[u{-}1]$ on the right-hand side 
is computed as follows
\begin{gather*}
-(-1)^{w-u}2  
\big(\tfrac{a+1}{a}\big) 
+(-1)^{w-u}\tfrac{2}{a}
=(-1)^{w+1-u}2.
\end{gather*}
After subtracting the multiples of $\pjw[u{-}1]$ from both sides, we conclude $x_{U}=0$. 
\end{proof}

\begin{proof}[Proof of the remaining cases] 
Now suppose that $S$ is smallest minimal down-admissible stretch 
for $v$, but not down-admissible for $v[S]$. Then \eqref{eq:zigzag-approx} takes the form
\begin{gather*}
\Down{S}\Up{S}\pjw[{v[S]{-}1}] 
=-2\loopdown{\hull}{{v[S]{-}1}}+
{\textstyle\sum_{U\not\subset {\hull}}}\,x_{U}\loopdown{U}{{v[S]{-}1}}.
\end{gather*}
We first multiply this by $\loopdown{\hull}{{v[S]{-}1}}$ and deduce
\begin{gather*}
\Down{S}\Up{S}\Up{{\hull}}\Down{{\hull}}\pjw[{v[S]{-}1}]= 
\Down{S}\Up{{\hull}}\Down{S}\Down{{\hull}}\pjw[{v[S]{-}1}]=0
\end{gather*}
from $\condO(v)$. Since $(\loopdown{\hull}{{v[S]{-}1}})^{2}=0$ 
by $\condE(v-1)$, we get $x_U=0$ unless $\hull{\subset}U$. Then a 
partial trace argument as above shows that all remaining $x_{U}$ are also zero.
\end{proof}

\begin{lemma}\label{lemma:I-stopped-counting}
$\condE(v)$ follows if we have $\condE(v-1)$, $\condA(v)$, $\condO(v)$, and $\condW(v)$.
\end{lemma}

\begin{proof}
We first prove \eqref{eq:DUD}. By $\condE(v-1)$ and projector absorption, 
we may assume that $S$ is a smallest minimal down-admissible stretch. 
Suppose first that $S$ is down-admissible for $v[S]$. Then we have
\begin{align*}
\Down{S}\Up{S}\Down{S}\pjw[v{-}1] 
&= 
\funcg \Up{S}\Down{S}\Down{S}\pjw[v{-}1]+\funcf\Up{T}\Up{S}\Down{S}\Down{T}\Down{S}\pjw[v{-}1]
\\
&= 
0+\funch\funcf\Up{T}\Up{S}\Down{S}\Up{S}\Down{T}\pjw[v{-}1]=0.
\end{align*}
where we have used 
$\condW(v)$, $\condA(v)$ and finally $\condE(v-1)$ to 
deduce $\pjw[{v[T](S){-}1}]\Up{S}\Down{S}\Up{S}=0$. 
Now, suppose that $S$ is not down-admissible for $v[S]$. Then we instead get
\begin{align*}
\Down{S}\Up{S}\Down{S}\pjw[v{-}1] 
&= 
-2\Up{\hull}\Down{\hull}\Down{S}\pjw[v{-}1]
\\
&= 
-2\Up{\hull}\Up{S}\Down{\hull}\pjw[v{-}1]=0.
\end{align*}
Here we have used $\condO(v)$ and \fullref{lemma:nilpotent}. 

Next we need to prove that $(\loopdown{S_{k}}{v{-}1})^{2}=0$ and 
$\loopdown{X}{v{-}1}=\prod_{k\mid S_{k}\subset X}\loopdown{S_{k}}{v{-}1}$. 
The first relation simply follows from \eqref{eq:DUD}, \ie
\begin{gather*}
(\loopdown{S_{k}}{v{-}1})^{2}=\Up{S_{k}}\Down{S_{k}}\Up{S_{k}}\Down{S_{k}}\pjw[v{-}1]=0.
\end{gather*}
Next, suppose we already know that $\loopdown{X}{v{-}1}=\prod_{S_{k}\subset X}\loopdown{S_{k}}{v{-}1}$ 
for subsets $X$ that decompose into $l-1$ minimal stretches, and let us 
now consider $X\cup S_{i}$ where $X>S_{i}$. The result is clear if $X$ and $S_{i}$ 
are distant, so we will assume that $S_{i}$ is adjacent to 
$X$. By projector absorption and \eqref{eq:DUD}, we may further assume that $S_{i}=S$ is the 
smallest minimal down-admissible stretch for $v$. Then we compute
\begin{gather*}
\loopdown{S}{v{-}1}\loopdown{X}{v{-}1}
=
\begin{tikzpicture}[anchorbase,scale=.25,tinynodes]
\draw[pJW] (2,-.5) rectangle (-2,.5);
\draw[pJW] (2,2) rectangle (-2,3);
\draw[pJW] (2,4.5) rectangle (-2,5.5);
\node at (0,-.2) {$\pjwm[v{-}1]$};
\node at (0.75,1.25) {$\dots$};
\node at (0.75,3.75) {$\dots$};
\node at (0,2.3) {$\pjwm[v{-}1]$};
\node at (0,4.8) {$\pjwm[v{-}1]$};
\draw[usual] (-.5,.5) to[out=90,in=180] (0,1) to[out=0,in=90] (0.5,.5);
\draw[usual] (-.5,2) to[out=270,in=180] (0,1.5) to[out=0,in=270] (.5,2);
\draw[usual] (-1.5,3) to[out=90,in=180] (-1,3.5) to[out=0,in=90] (-.5,3);
\draw[usual] (-1.5,4.5) to[out=270,in=180] (-1,4) to[out=0,in=270] (-.5,4.5);
\draw[usual] (-1.5,.5) to (-1.5,2);
\draw[usual] (1.5,.5) to (1.5,2);
\draw[usual] (-1.5,.5) to (-1.5,2);
\draw[usual] (0,3) to (0,4.5)(1.5,3) to (1.5,4.5);
\node at (1,2.5) {$\phantom{a}$};
\node at (1,-.5) {$\phantom{a}$};
\end{tikzpicture}
=
\begin{tikzpicture}[anchorbase,scale=.25,tinynodes]
\draw[pJW] (0,-.5) rectangle (2,.5);
\draw[pJW] (2,2) rectangle (-2,3);
\draw[pJW] (2,4.5) rectangle (-1,5.5);
\node at (0,2.3) {$\pjwm[v{-}1]$};
\draw[usual] (-.5,-.5) to (-.5,.5) to[out=90,in=180] (0,1) to[out=0,in=90] (0.5,.5);
\draw[usual] (-.5,2) to[out=270,in=180] (0,1.5) to[out=0,in=270] (.5,2);
\draw[usual] (-1.5,3) to[out=90,in=180] (-1,3.5) to[out=0,in=90] (-.5,3);
\draw[usual] (-1.5,5.5) to (-1.5,4.5) to[out=270,in=180] (-1,4) to[out=0,in=270] (-.5,4.5);
\draw[usual] (-1.5,-.5) to (-1.5,2);
\draw[usual] (1.5,.5) to (1.5,2);
\draw[usual] (-1.5,.5) to (-1.5,2);
\draw[usual] (0,3) to (0,4.5)(1.5,3) to (1.5,4.5);
\node at (1,2.5) {$\phantom{a}$};
\node at (1,-.5) {$\phantom{a}$};
\end{tikzpicture}
=
\begin{tikzpicture}[anchorbase,scale=.25,tinynodes]
\draw[pJW] (0,-.5) rectangle (2,.5);
\draw[pJW] (2,2) rectangle (-2,3);
\draw[pJW] (2,3.25) rectangle (-0.25,4.25);
\draw[pJW] (2,4.5) rectangle (-1,5.5);
\node at (0.75,1.25) {$\dots$};
\node at (0,2.3) {$\pjwm[v{-}1]$};
\draw[usual] (-.5,-.5) to (-.5,.5) to[out=90,in=180] (0,1) to[out=0,in=90] (0.5,.5);
\draw[usual] (-.5,2) to[out=270,in=180] (0,1.5) to[out=0,in=270] (.5,2);
\draw[usual] (-1.5,3) to[out=90,in=180] (-1,3.5) to[out=0,in=90] (-.5,3);
\draw[usual] (-1.5,5.5) to (-1.5,4.5) to[out=270,in=180] (-1,4) to[out=0,in=270] (-.5,4.5);
\draw[usual] (-1.5,-.5) to (-1.5,2);
\draw[usual] (1.5,.5) to (1.5,2);
\draw[usual] (-1.5,.5) to (-1.5,2);
\draw[usual] (0,3) to (0,3.25)(1.5,3) to (1.5,3.25)(0,4.25) to (0,4.5)(1.5,4.25) to (1.5,4.5);
\node at (1,2.5) {$\phantom{a}$};
\node at (1,-.5) {$\phantom{a}$};
\end{tikzpicture}
=
\begin{tikzpicture}[anchorbase,scale=.25,tinynodes]
\draw[pJW] (0,-.5) rectangle (2,.5);
\draw[pJW] (2,2) rectangle (-1,3);
\draw[pJW] (2,3.25) rectangle (-0.25,4.25);
\draw[pJW] (2,4.5) rectangle (-1,5.5);
\draw[usual] (-.5,-.5) to (-.5,.5) to[out=90,in=180] (0,1) to[out=0,in=90] (0.5,.5);
\draw[usual] (-.5,2) to[out=270,in=180] (0,1.5) to[out=0,in=270] (.5,2);
\draw[usual] (-1.5,3) to[out=90,in=180] (-1,3.5) to[out=0,in=90] (-.5,3);
\draw[usual] (-1.5,5.5) to (-1.5,4.5) to[out=270,in=180] (-1,4) to[out=0,in=270] (-.5,4.5);
\draw[usual] (-1.5,-.5) to (-1.5,3);
\draw[usual] (1.5,.5) to (1.5,2);
\draw[usual] (-1.5,.5) to (-1.5,2);
\draw[usual] (0,3) to (0,3.25)(1.5,3) to (1.5,3.25)(0,4.25) to (0,4.5)(1.5,4.25) to (1.5,4.5);
\node at (1,2.5) {$\phantom{a}$};
\node at (1,-.5) {$\phantom{a}$};
\end{tikzpicture}
=
\begin{tikzpicture}[anchorbase,scale=.25,tinynodes]
\draw[pJW] (0,-.5) rectangle (2,.5);
\draw[pJW] (2,2) rectangle (-1,3);
\draw[pJW] (2,4.5) rectangle (-1,5.5);
\draw[usual] (-.5,-.5) to (-.5,.5) to[out=90,in=180] (0,1) to[out=0,in=90] (0.5,.5);
\draw[usual] (-.5,2) to[out=270,in=180] (0,1.5) to[out=0,in=270] (.5,2);
\draw[usual] (-1.5,3) to[out=90,in=180] (-1,3.5) to[out=0,in=90] (-.5,3);
\draw[usual] (-1.5,5.5) to (-1.5,4.5) to[out=270,in=180] (-1,4) to[out=0,in=270] (-.5,4.5);
\draw[usual] (-1.5,-.5) to (-1.5,3);
\draw[usual] (1.5,.5) to (1.5,2);
\draw[usual] (-1.5,.5) to (-1.5,2);
\draw[usual] (0,3) to (0,4.5)(1.5,3) to (1.5,4.5);
\node at (1,2.5) {$\phantom{a}$};
\node at (1,-.5) {$\phantom{a}$};
\end{tikzpicture}
=
\begin{tikzpicture}[anchorbase,scale=.25,tinynodes]
\draw[pJW] (-1,-.5) rectangle (2,.5);
\draw[pJW] (2,2) rectangle (-1,3);
\draw[pJW] (2,4.5) rectangle (-1,5.5);
\draw[usual] (-.5,.5) to[out=90,in=180] (0,1) to[out=0,in=90] (0.5,.5);
\draw[usual] (-.5,2) to[out=270,in=180] (0,1.5) to[out=0,in=270] (.5,2);
\draw[usual] (-1.5,3) to[out=90,in=180] (-1,3.5) to[out=0,in=90] (-.5,3);
\draw[usual] (-1.5,5.5) to (-1.5,4.5) to[out=270,in=180] (-1,4) to[out=0,in=270] (-.5,4.5);
\draw[usual] (-1.5,-.5) to (-1.5,3);
\draw[usual] (1.5,.5) to (1.5,2);
\draw[usual] (-1.5,.5) to (-1.5,2);
\draw[usual] (0,3) to (0,4.5)(1.5,3) to (1.5,4.5);
\node at (1,2.5) {$\phantom{a}$};
\node at (1,-.5) {$\phantom{a}$};
\end{tikzpicture}
\xrightarrow{\pTr_{v{-}\mother}}
\begin{tikzpicture}[anchorbase,scale=.25,tinynodes]
\draw[pJW] (-1,-.5) rectangle (2,.5);
\draw[pJW] (2,2) rectangle (-1,3);
\draw[usual] (-.5,.5) to[out=90,in=180] (0,1) to[out=0,in=90] (0.5,.5);
\draw[usual] (-.5,2) to[out=270,in=180] (0,1.5) to[out=0,in=270] (.5,2);
\draw[usual] (1.5,.5) to (1.5,2);
\node at (1,2.5) {$\phantom{a}$};
\node at (1,-.5) {$\phantom{a}$};
\end{tikzpicture}
= 
\loopdown{X}{\mother{-}1}.
\end{gather*}
The equation $\mathrm{pTr}_{v{-}\mother}(\loopdown{S}{v{-}1}\loopdown{X}{v{-}1})=\loopdown{X}{\mother{-}1}$ and 
\fullref{proposition:p-properties-trace} imply 
$\loopdown{S}{v{-}1}\loopdown{X}{v{-}1}=x\loopdown{X}{v{-}1}+y\loopdown{XS}{v{-}1}$ where $(-1)^{v-\mother}2x+y=1$. 
Equivalently, we can write $\loopdown{X}{v{-}1}(-x+\loopdown{S}{v{-}1})=y\loopdown{XS}{v{-}1}$. 
Suppose that $x\neq 0$, then $(-x+\loopdown{S}{v{-}1})$ would be a unit in 
$\End_{\TL[\F]}\left(\pjw[v{-}1]\right)$, so we can write 
\begin{gather*}
\loopdown{X}{v{-}1}=(-x+\loopdown{S}{v{-}1})^{-1}y\loopdown{XS}{v{-}1}.
\end{gather*}
However, the left-hand side has through-degree $v[X]$, while 
the right-hand side has through-degree at most $v[X\cup S]<v[X]$, 
a contradiction. Thus, we have $x=0$ and $y=1$, and 
consequently $\loopdown{S}{v{-}1}\loopdown{X}{v{-}1}=\loopdown{XS}{v{-}1}$.
\end{proof}

This completes the proof of \fullref{theorem:main-tl-section}, which 
by \fullref{proposition:TLtilt}, completes the proof of \fullref{theorem:main}.

\begin{remark}\label{rem:smallcases}
In addition to the eve base cases $1\leq v \leq\ppar$ for the induction we have
explicitly seen certain relations in cases of low generation. For example, for
$v$ of generation $1$, the description of the endomorphism algebra can be
deduced from the proof of \fullref{lemma:nil-ploop} while the adjacency and
overlap relations hold vacuously. For $v$ of generation $2$, we have seen the
adjacency relations in \fullref{lemma:adjacent-generators-3} and the overlap
relations in \fullref{lemma:overlapapprox}. Finally, zigzag relations for loops
based at $w$ of generation $2$ were treated in \fullref{lem:gen2zigzag}.
\end{remark}

\section{Some conclusions}\label{section:final}

\subsection{The fractal nature of \texorpdfstring{$\zigzag$}{Z}}\label{subsec:fractal}
The quiver underlying $\zigzag$ is a graph with countably infinitely many connected components. 
In each connected component there is a unique vertex $e-1$ with $e\in\eve$, 
and we denote the vertex set of this component by $(e)_{\ppar}$.

\begin{example}\label{example:block}
We have $(1)_{3}=\{0<4<6<10<12<16<18<22<\cdots\}$, \cf \eqref{eq:funny-example}.
\end{example}

The decomposition of the quiver implies that the algebra $\zigzag$ decomposes as
\begin{gather*}
\zigzag 
= 
{\textstyle\coprod_{e\in\eve}}\,\zigzag_{e{-}1}
, \qquad
\zigzag_{e{-}1}:={\textstyle\bigoplus_{v,w\in(e)_{\ppar}}}\,\idemy[w{-}1]\zigzag\idemy[v{-}1].
\end{gather*}
Let $M_{\ppar}$ denote the free monoid on the set $L=\{0,\dots,\ppar-1\}$.
We represent the elements of $M_{\ppar}$ by words $[b_{k},\dots,b_{0}]$ for $b_{i}\in L$ 
and the multiplication is given by 
\begin{gather*}
[b_{k},\dots,b_{0}]\odot 
[a_{j},\dots,a_{0}]
:=[a_{j},\dots,a_{0},b_{k},\dots,b_{0}].
\end{gather*}
(Note that the empty word $\emptyset$ is the neutral element.) 
Elements of $M_{\ppar}$ with differing numbers of leading zeros are considered as distinct, 
and so they should not be interpreted as $\ppar$-adic expansions of natural numbers. 
However, $\N$ carries an action of $M_{\ppar}$ defined by 
\begin{gather*}
M_{\ppar} 
\times 
\N\to\N,\quad
[b_{k},\dots,b_{0}]\odot 
\pbase{a_{j},\dots,a_{0}}{\ppar}
:=\pbase{a_{j},\dots,a_{0},b_{k},\dots,b_{0}}{\ppar}.
\end{gather*}

The fractal nature, \ie the self-similarity,
of $\zigzag$ is now captured by the following proposition.

\begin{proposition}\label{proposition:fractal}
The monoid $M_{\ppar}$ acts on $\zigzag$ by 
algebra endomorphisms: For each $w\in M_{\ppar}$, there is an 
algebra endomorphism $\phi_{w}\in\End(\zigzag)$ acting on idempotents by
\begin{gather*}
\phi_{w}(\pjw[v{-}1]) 
:= 
\pjw[{w\odot v{-}1}],
\end{gather*}
and on arrows by reindexation. Moreover, we have
\begin{gather*}
\phi_{z}\phi_{w}
= 
\phi_{z\odot w}\text{ for }w,z\in M_{\ppar}.
\end{gather*}
Finally, if $w=[0,\dots,0]$, then $\phi_{w}$ maps any $\zigzag_{e{-}1}$
isomorphically to $\zigzag_{w\odot e{-}1}$. 
In this sense, $\zigzag$ is generated under the 
action of $M_{\ppar}$ by the summands $\zigzag_{e{-}1}$ for 
$e\in\{1,\dots,\ppar-1\}$.
\end{proposition}

\begin{proof}
This is a direct consequence of \fullref{theorem:main-tl-section}.
\end{proof}

We also note that, since $\zigzag$ is 
the direct sum of $\N$ many copies of $\zigzag_{e{-}1}$ for 
$e\in\{1,\dots,\ppar-1\}$, 
the underlying quiver of $\zigzag$ 
is a \emph{fractal graph} in the sense of Ille--Woodrow \cite{IlWo-fractal-graphs}, 
albeit in the trivial sense that any countable graph without edges and more than one vertex can be considered as a fractal factor. 

\subsection{A few words about tilting modules}\label{sec:tilting}

Let us work over the ground field $\K$.
First, recall the category of finite-dimensional modules for $\SLtwo$ has 
simple $\lmod(v-1)$, Weyl $\wmod(v-1)$, dual Weyl $\dwmod(v-1)$ and 
indecomposable tilting modules $\tmod(v-1)$ for $v\in\N$, the latter being the
indecomposable objects of $\tilt$,
see \eg \cite[Section 1]{Wi-algebraic-sheaves} for a concise summary of 
the main definitions and properties regarding $\tilt$.

Let us now elaborate a bit further on the representation-theoretic implications of \fullref{corollary:main}. 
Almost all of these are, of course, well-understood. 
However, the reader might find it helpful to see how they can be 
derived from our results in the previous sections.

It is well-known that
\begin{gather*}
\tilt
=
{\textstyle\bigoplus_{e\in\eve}}\,
\tilt_{e{-}1},
\quad
\tilt_{e{-}1}=
\{
\tmod(v-1)\mid
v-1\in (e)_{\ppar}
\},
\end{gather*}
whose direct summands are called \emph{blocks}, which are equivalent as additive, $\K$-linear 
categories. From our discussion we immediately get the following.

\begin{proposition}\label{proposition:connected}
There is an equivalence of additive, $\K$-linear categories
\begin{gather*}
\mainfunctor^{\prime}_{e{-}1}\colon
\tilt_{e{-}1}\xrightarrow{\cong}\zigzagmode,
\end{gather*}
sending indecomposable tilting modules to indecomposable projectives. 
Moreover, $\Hom_{\tilt}(\obstuff{X},\obstuff{Y})=0$ if 
$\obstuff{X}\in\tilt_{e{-}1}$ 
and $\obstuff{Y}\in\tilt_{e^{\prime}{-}1}$ for $e,e^{\prime}\in\eve,e\neq e^{\prime}$.
Finally, there is an isomorphism of algebras $\zigzag_{e{-}1}\cong\zigzag_{e^{\prime}{-}1}$ 
for all $e,e^{\prime}\in\eve$ with equal non-zero digits.
\end{proposition}
\begin{proof}
Directly from \fullref{theorem:main} and \fullref{corollary:main}, combined with 
\fullref{theorem:main-tl-section} and the above.
\end{proof}

In fact $\zigzag_{e{-}1}\cong\zigzag_{e^{\prime}{-}1}$ 
for all $e,e^{\prime}\in\eve$, but such isomorphisms involve non-trivial rescalings in our presentation.

Another consequence we get are the tilting--dual Weyl multiplicities.

\begin{proposition}\label{proposition:tilting-weyl}
We have
\begin{gather*}
\big(\tmod(v-1):\dwmod(w-1)\big)
=
\begin{cases}
1
&\text{if }w\in\supp,
\\
0
&\text{else}.
\end{cases}
\end{gather*}
\end{proposition}

\begin{proof}
Note that the basis \fullref{theorem:main-tl-section}.(Basis) is part of the
family of bases constructed in \cite{AnStTu-cellular-tilting}, and the
proposition follows from the construction of these bases in {\loccit}, and our
main statements \fullref{theorem:main} and \fullref{corollary:main}.
\end{proof}

Hence, we get the tilting characters 
$\chi_{v{-}1}^{\tmod}=\sum_{w\in\N}\big(\tmod(v-1):\dwmod(w-1)\big)\chi_{w{-}1}^{\dwmod}$,
where the characters $\chi_{w{-}1}^{\dwmod}=\chi_{w{-}1}^{\wmod}$ 
are the characteristic zero characters well-known \eg by Weyl's character formula. 
By reciprocity, \cf \cite[Proposition 1.14]{RiWi-tilting-p-canonical}, we also get the Weyl--simple multiplicities. 
To state these explicitly, let us recursively define a set $X(v)$ as follows. For
$v\leq\ppar-1$ let $X(v)=\{0\}$. For $v>\ppar-1$ let
\begin{gather*}
X(v)=
\begin{cases}
\ppar X\big((v-a_{0})/\ppar\big)\cup\bigg(a_{0}+1+\ppar X\big((v-a_{0}-\ppar)/\ppar\big)\bigg)
&\text{if }a_{0}\neq\ppar-1,
\\
\ppar X\big((v-a_{0})/\ppar\big)
&\text{if }a_{0}=\ppar-1,
\end{cases}
\end{gather*}
where we again meet {\losp}.
Then we get
\begin{gather*}
\big[\wmod(w-1):\lmod(v-1)\big]
=
\begin{cases}
1
&\text{if }v\in w-2X(w),
\\
0
&\text{else}.
\end{cases}
\end{gather*}
Thus, we get $\chi_{w{-}1}^{\wmod}
=\sum_{v\in\N}
\big[\wmod(w-1):\lmod(v-1)\big]\chi_{v{-}1}^{\lmod}
=\sum_{v\in w-X(w)}\chi_{v{-}1}^{\lmod}$,
which determines the simple characters by inverting the change of basis matrix.

\begin{example}\label{example:block-2}
For $\ppar=3$ we have
$\chi_{22}^{\tmod}
=\chi_{22}^{\dwmod}+\chi_{18}^{\dwmod}+\chi_{16}^{\dwmod}+\chi_{12}^{\dwmod}$, \cf \cite[Figure 1]{JeWi-p-canonical}.
Moreover, we also get $\chi_{22}^{\wmod}=\chi_{22}^{\lmod}+\chi_{18}^{\lmod}+\chi_{12}^{\lmod}+\chi_{10}^{\lmod}$.
\end{example}

The final consequence we would like 
to derive in this paper is the following. 

\begin{proposition}\label{proposition-final}
Let $\catstuff{I}_{v}=\{\tmod(w-1)\mid w\geq v\}$. 
For any thick $\hcirc$-ideal $\catstuff{I}\neq 0$ in $\tilt$ there exists 
$k\in\N[0]$ such that
\begin{gather*}
\catstuff{I}
=
\catstuff{I}_{\ppar^{k}}
\stackrel{\ppar{\neq}2}{=}
\left\{\tmod(v-1)\mid
\ord\big(\dim_{\F}\big(\tmod(v-1)\big)\big)
\leq k
\right\},
\end{gather*}
with the latter equality holding in case $\ppar\neq 2$.
\end{proposition}

\begin{proof}
We will use that $\tmod(1)$ is a $\hcirc$-generator of $\tilt$ and
Weyl and dual Weyl modules have classical characters.

Assume that $\tmod(v-1)\in\catstuff{I}$ for $v$ minimal. Then it is clear by 
\fullref{proposition:tilting-weyl} that $\catstuff{I}=\catstuff{I}_{v}$. Thus, it remains to determine the 
possible minimal $v$.
To this end, note that 
the decomposition of $\tmod(1)\hcirc\tmod(w-1)\cong\tmod(w-1)\hcirc\tmod(1)$ 
into its indecomposable summands is completely determined by \fullref{proposition:tilting-weyl} 
and the classical $\mathrm{SL}_{2}(\C)$ tensor product combinatorics.
Analyzing now how tensoring with $\tmod(1)$ affects the support, 
\fullref{proposition:tilting-weyl} then also implies that $v$ needs to be a prime 
power, and conversely, that a prime power works as a minimal $v$.

Finally, the last statement follows from \fullref{proposition:p-properties-trace} since 
$\pjw[e{-}1]=\pqjw[e{-}1]=\qjw[e{-}1]$ for $e\in\eve$.
\end{proof}

Thus, the thick $\hcirc$-ideals in $\tilt$ are
$\tilt=
\catstuff{I}_{\ppar^{0}}
\supset
\catstuff{I}_{\ppar^{1}}
\supset
\catstuff{I}_{\ppar^{2}}
\supset
\catstuff{I}_{\ppar^{3}}
\supset
\catstuff{I}_{\ppar^{4}}
\supset
\cdots$.

\begin{example}\label{example:final}
The elements of $\catstuff{I}_{\ppar^{1}}$ are the so-called \emph{negligible} modules.
\end{example}

Note that the above implies that $\tilt$ has no projective modules 
since these would form the minimal thick $\hcirc$-ideal.

\section*{Table of notation and central concepts}

In general we use a tilde, \eg $\tilde{f}$, to indicate that we work over $\Q$, an overline, \eg $\overline{f}$, to indicate that we have something that reduces mod $\ppar$ but we want to consider it over $\Q$, and no extra decoration if we work in $\F[p]$.

\begin{table}[htbp]
\begin{center}
\begin{tabular}{r|c|p{10cm}}
Name & Symbol & Description
\\
\toprule
Ringel dual of $\mathrm{SL}_{2}$ & $\zigzag$ & the path algebra of a quiver with relations
presented in \fullref{theorem:main-tl-section}.
\\
---\;\;\;\;& $\funcf$, $\funcg$, $\funch$ & functions $\F\to \F$ used in the presentation
of $\zigzag$, \fullref{subsec:main-TL}.
\\
\hline
JW projector & $\qjw$ & the Jones--Wenzl projectors, corresponding to
projections to $\wmod(v-1)$ in $\tmod(1)^{\otimes(v{-}1)}$; defined over $\Q$,
\fullref{definition:JW}.
\\
$\ppar$JW projector & $\pjw$ & the $\ppar$Jones--Wenzl projectors, corresponding
to projections to $\tmod(v-1)$ in $\tmod(1)^{\otimes(v{-}1)}$; defined over
$\F[p]$, \fullref{definition:p-jw}.
\\
rational $\ppar$JW proj. & $\pqjw$ & the $\ppar$Jones--Wenzl projectors,
corresponding to projections to $\tmod(v-1)$ in $\tmod(1)^{\otimes(v{-}1)}$, but
considered over $\Q$, \fullref{definition:p-jw-q}.
\\
---\;\;\;\;& $\lambda_{v,S}$ & scalars in the definition of projectors \eqref{eq:the-scalar}.
\\
\hline
integral morphisms &  $\down{S} \idtl$ & down or up morphisms given by cups or
caps; these work integrally, \fullref{definition:cup-cap-operators}.
\\
standard morphisms &  $\downo{S} \idtl$ & down or up morphisms given by cups or
caps together with JW projectors; these over $\Q$,
\fullref{definition:jw-cupscaps2}.
\\
$\ppar$ morphisms &  $\Down{S}\pjw[v{-}1]$& down or up morphisms given by cups or
caps together with $\ppar$JW projectors; these over $\F[p]$,
\fullref{definition:updown}.
\\
standard loops & $\trap{S}{v{-}1}$ & compositions of down and up morphisms; form
a basis of endomorphism spaces over $\Q$, \fullref{definition:jw-cupscaps2}.
\\
$\ppar$loops & $\loopdown{S}{v{-}1}$ & compositions of down and up morphisms;
form a basis of endomorphism spaces over $\F[p]$,
\fullref{definition:updown}.
\\
\hline
eve & $e$ & a number with a single non-zero $\ppar$-adic digit,
\fullref{definition:ancestry}.
\\
mother of $v$ & $\mother$ & defined (unless $v$ is an eve) by setting the last
non-zero $\ppar$-adic digit of $v$ to zero, \fullref{definition:ancestry}.
\\
ancestors of $v$ & $\mother,\motherr{v}{2}\cdots$ & positive numbers obtained by
setting last $\ppar$-adic digits of $v$ to zero,
\fullref{definition:ancestry}.
\\
generation of $v$ & $\generation$ & the number of ancestors of $v$,
\fullref{definition:ancestry}.
\\
stretches & --- & sets of consecutive digits in the $\ppar$-adic expansion of a
number, \fullref{definition:adm}.
\\
admissibility & --- & whether a set of digits of a $\ppar$-adic expansion is
suitable for reflecting up or down, \fullref{definition:adm}.
\\
admissible hull & $\hull$ & an admissible set containing $S$,
\fullref{definition:adm}.
\\
\bottomrule
\end{tabular}
\end{center}
\label{table:notation}
\end{table}


\begin{thebibliography}{99999999999}

\bibitem[AJL83]{AnJoLa-sl2-projectives}
H.H.~Andersen, J.~J{\o}rgensen, and P.~Landrock.
\newblock The projective indecomposable modules of $\mathrm{SL}(2,p^{n})$.
\newblock \emph{Proc. Lond. Math. Soc. (3)}, 46(1):38--52, 1983.
\newblock \href
{https://doi.org/10.1112/plms/s3-46.1.38}
{\path{10.1112/plms/s3-46.1.38}}.

\bibitem[AST18]{AnStTu-cellular-tilting}
H.H.~Andersen, C.~Stroppel, and D.~Tubbenhauer.
\newblock Cellular structures using {$\mathrm{U}_{q}$}-tilting modules.
\newblock \emph{Pacific J. Math.}, 292(1):21--59, 2018.
\newblock URL: \url{https://arxiv.org/abs/1503.00224}, \href
{http://dx.doi.org/10.2140/pjm.2018.292.21}
{\path{doi:10.2140/pjm.2018.292.21}}.

\bibitem[AST17]{AnStTu-semisimple-tilting}
H.H.~Andersen, C.~Stroppel, and D.~Tubbenhauer.
\newblock Semisimplicity of {H}ecke and (walled) {B}rauer algebras.
\newblock \emph{J. Aust. Math. Soc.}, 103(1):1--44, 2017.
\newblock URL: \url{https://arxiv.org/abs/1507.07676}, \href
{http://dx.doi.org/10.1017/S1446788716000392}
{\path{doi:10.1017/S1446788716000392}}.

\bibitem[AT17]{AnTu-tilting}
H.H.~Andersen and D.~Tubbenhauer.
\newblock Diagram categories for {$\textbf{U}_{q}$}-tilting modules at roots of
unity.
\newblock \emph{Transform. Groups}, 22(1):29--89, 2017.
\newblock URL: \url{https://arxiv.org/abs/1808.08022}, \href
{http://dx.doi.org/10.1007/s00031-016-9363-z}
{\path{doi:10.1007/s00031-016-9363-z}}.

\bibitem[BS18]{BS18}
J.~Brundan and C.~Stroppel.
\newblock Semi-infinite highest weight categories.
\newblock 2018.
\newblock URL: \url{https://arxiv.org/abs/1808.08022}.

\bibitem[BLS19]{BuLiSe-tl-char-p}
G.~Burrull, N.~Libedinsky, and P.~Sentinelli.
\newblock p-{J}ones--{W}enzl idempotents.
\newblock \emph{Adv. Math.}, 352:246--264, 2019.
\newblock URL: \url{https://arxiv.org/abs/1902.00305}, \href
{https://doi.org/10.1016/j.aim.2019.06.005}
{\path{10.1016/j.aim.2019.06.005}}.

\bibitem[CC76]{CaCl-submodule-weyl-a1}
R.~Carter and E.~Cline.
\newblock The submodule structure of {W}eyl modules for groups of type
{$A_{1}$}.
\newblock \emph{Proceedings of the Conference on Finite Groups (Univ. Utah, Park City, Utah, 1975)}.
\newblock Pages 303--311, 1976.

\bibitem[CKM14]{CaKaMo-webs-skew-howe}
S.~Cautis, J.~Kamnitzer, and S.~Morrison.
\newblock Webs and quantum skew {H}owe duality.
\newblock \emph{Math. Ann.}, 360(1-2):351--390, 2014.
\newblock URL: \url{https://arxiv.org/abs/1210.6437}, \href
{http://dx.doi.org/10.1007/s00208-013-0984-4}
{\path{doi:10.1007/s00208-013-0984-4}}.

\bibitem[Don93]{Do-tilting-alg-groups}
S.~Donkin.
\newblock On tilting modules for algebraic groups.
\newblock \emph{Math. Z.}, 212(1):39--60, 1993.
\newblock \href {http://dx.doi.org/10.1007/BF02571640}
{\path{doi:10.1007/BF02571640}}.

\bibitem[DH05]{DoHe-char-p-sl2}
S.~Doty and A.~Henke.
\newblock Decomposition of tensor products of modular irreducibles for
{$\mathrm{SL}_{2}$}.
\newblock \emph{Q. J. Math.}, 56(2):189--207, 2005.
\newblock URL: \url{https://arxiv.org/abs/math/0205186}, \href
{http://dx.doi.org/10.1093/qmath/hah027} {\path{doi:10.1093/qmath/hah027}}.

\bibitem[Eli15]{El-ladders-clasps}
B.~Elias.
\newblock Light ladders and clasp conjectures.
\newblock 2015.
\newblock URL: \url{https://arxiv.org/abs/1510.06840}.

\bibitem[Eli17]{El-q-satake}
B.~Elias.
\newblock Quantum {S}atake in type {$A$}. {P}art {I}.
\newblock \emph{J. Comb. Algebra}, 1(1):63--125, 2017.
\newblock URL: \url{https://arxiv.org/abs/1403.5570}, \href
{http://dx.doi.org/10.4171/JCA/1-1-4} {\path{doi:10.4171/JCA/1-1-4}}.

\bibitem[EH02a]{ErHe-ringel-schur}
K.~Erdmann and A.~Henke.
\newblock On {R}ingel duality for {S}chur algebras.
\newblock \emph{Math. Proc. Cambridge Philos. Soc.}, 132(1):97--116, 2002.
\newblock \href {http://dx.doi.org/10.1017/S0305004101005485}
{\path{doi:10.1017/S0305004101005485}}.

\bibitem[EH02b]{ErHe-ringel-schur-symmetric-group}
K.~Erdmann and A.~Henke.
\newblock On {S}chur algebras, {R}ingel duality and symmetric groups.
\newblock \emph{J. Pure Appl. Algebra}, 169(2-3):175--199, 2002.
\newblock \href {http://dx.doi.org/10.1016/S0022-4049(01)00071-8}
{\path{doi:10.1016/S0022-4049(01)00071-8}}.

\bibitem[GW93]{GoWe-tl-root-of-unity}
F.M.~Goodman and H.~Wenzl.
\newblock The Temperley--Lieb algebra at roots of unity.
\newblock \emph{Pacific J. Math.}, 161(2):307--334, 1993.

\bibitem[IW19]{IlWo-fractal-graphs}
P.~Ille and R.~Woodrow.
\newblock Fractal graphs.
\newblock \emph{J. Graph Theory}, 91(1):53--72, 2019.
\newblock \href{https://doi.org/10.1002/jgt.22420} {\path{doi.org/10.1002/jgt.22420}}.

\bibitem[JW17]{JeWi-p-canonical}
L.T.~Jensen and G.~Williamson.
\newblock The {$p$}-canonical basis for {H}ecke algebras.
\newblock In \emph{Categorification and higher representation theory}, volume
683 of \emph{Contemp. Math.}, pages 333--361. Amer. Math. Soc., Providence,
RI, 2017.
\newblock URL: \url{https://arxiv.org/abs/1510.01556}, \href
{http://dx.doi.org/10.1090/conm/683} {\path{doi:10.1090/conm/683}}.

\bibitem[KL94]{KaLi-TL-recoupling}
L.H.~Kauffman and S.L.~Lins.
\newblock \emph{Temperley--{L}ieb recoupling theory and invariants of
{$3$}-manifolds}, volume 134 of \emph{Annals of Mathematics Studies}.
\newblock Princeton University Press, Princeton, NJ, 1994.
\newblock \href {http://dx.doi.org/10.1515/9781400882533}
{\path{doi:10.1515/9781400882533}}.

\bibitem[MMMT20]{MaMaMiTu-trihedral}
M.~Mackaay, V.~Mazorchuk, V.~Miemietz, and D.~Tubbenhauer.
\newblock Trihedral {S}oergel bimodules.
\newblock Fund. Math. 248 (2020), no. 3, 219--300.
\newblock URL: \url{https://arxiv.org/abs/1804.08920}, \href
{http://dx.doi.org/10.4064/fm566-3-2019}
{\path{doi:10.4064/fm566-3-2019}}.

\bibitem[MT15]{MiTu-simple-ext-gl2}
V.~Miemietz and W.~Turner.
\newblock Koszul dual {$2$}-functors and extension algebras of simple modules
for {$\mathrm{GL}_{2}$}.
\newblock \emph{Selecta Math. (N.S.)}, 21(2):605--648, 2015.
\newblock URL: \url{https://arxiv.org/abs/1106.5411}, \href
{http://dx.doi.org/10.1007/s00029-014-0164-8}
{\path{doi:10.1007/s00029-014-0164-8}}.

\bibitem[MT11]{MiTu-rational-gl2}
V.~Miemietz and W.~Turner.
\newblock Rational representations of {$\mathrm{GL}_{2}$}.
\newblock \emph{Glasg. Math. J.}, 53(2):257--275, 2011.
\newblock URL: \url{https://arxiv.org/abs/0809.0982}, \href
{http://dx.doi.org/10.1017/S0017089510000686}
{\path{doi:10.1017/S0017089510000686}}.

\bibitem[MT13]{MiTu-weyl-ext-gl2}
V.~Miemietz and W.~Turner.
\newblock The {W}eyl extension algebra of {$\mathrm{GL}_{2}(\overline{\mathbbm{F}}_{p})$}.
\newblock \emph{Adv. Math.}, 246:144--197, 2013.
\newblock URL: \url{https://arxiv.org/abs/1106.5665?context=math}, \href
{http://dx.doi.org/10.1016/j.aim.2013.07.003}
{\path{doi:10.1016/j.aim.2013.07.003}}.

\bibitem[RW18]{RiWi-tilting-p-canonical}
S.~Riche and G.~Williamson.
\newblock Tilting modules and the {$p$}-canonical basis.
\newblock \emph{Ast{\'e}risque}, (397):ix+184, 2018.
\newblock URL: \url{https://arxiv.org/abs/1512.08296}.

\bibitem[TVW17]{TuVaWe-super-howe}
D.~Tubbenhauer, P.~Vaz, and P.~Wedrich.
\newblock Super {$q$}-{H}owe duality and web categories.
\newblock \emph{Algebr. Geom. Topol.}, 17(6):3703--3749, 2017.
\newblock URL: \url{https://arxiv.org/abs/1504.05069}, \href
{http://dx.doi.org/10.2140/agt.2017.17.3703}
{\path{doi:10.2140/agt.2017.17.3703}}.

\bibitem[Wil17]{Wi-algebraic-sheaves}
G.~Williamson.
\newblock Algebraic representations and constructible sheaves.
\newblock \emph{Jpn. J. Math.}, 12(2):211--259, 2017.
\newblock URL: \url{https://arxiv.org/abs/1610.06261}, \href
{http://dx.doi.org/10.1007/s11537-017-1646-1}
{\path{doi:10.1007/s11537-017-1646-1}}.

\end{thebibliography}
\end{document}